\documentclass[AHL,XML,longabstracts,Unicode]{cedram}

\datereceived{2018-01-11}
\dateaccepted{2018-11-21}

\editor{D. Aldous}
\DOI{10.NNNN/ahl.XXXX}

\AtBeginDocument{

}

\usepackage{amsbsy}
\usepackage{amsmath,amsfonts,amssymb,amsthm,mathrsfs,mathtools}
\usepackage{graphicx}
\usepackage{epsfig}%pour les eps
\usepackage{latexsym}%encore des symboles
\usepackage{xcolor}
\usepackage{ae,aecompl}

\usepackage{xcolor}

\usepackage{pstricks}
\usepackage{enumerate}
\usepackage{tikz,animate,media9}						%%%%%%%%%%
\usepackage{pifont}
\usepackage{bm,marvosym}
\usepackage{algorithm}
\usepackage{algorithmic}

\usetikzlibrary{arrows, patterns, calc}		%%%%%%%%%%

\DeclareFontFamily{U}{BOONDOX-calo}{\skewchar\font=45 }
\DeclareFontShape{U}{BOONDOX-calo}{m}{n}{
  <-> s*[1.05] BOONDOX-r-calo}{}
\DeclareFontShape{U}{BOONDOX-calo}{b}{n}{
  <-> s*[1.05] BOONDOX-b-calo}{}
\DeclareMathAlphabet{\mathcalbis}{U}{BOONDOX-calo}{m}{n}
\SetMathAlphabet{\mathcalbis}{bold}{U}{BOONDOX-calo}{b}{n}
\DeclareMathAlphabet{\mathbcalboondox}{U}{BOONDOX-calo}{b}{n}

\providecommand{\abs}[1]{\lvert#1\rvert}

\usepackage{bbm,enumerate,tabulary,booktabs}

\usepackage{enumitem}
	\renewcommand{\theenumi}{{(\roman{enumi})}}
	\renewcommand{\labelenumi}{\theenumi}

\renewcommand{\Re}{\operatorname{Re}}

\newcommand{\Z}{\mathbb{Z}}
\renewcommand{\P}{\mathbb{P}}
\newcommand{\E}{\mathbb{E}}

\def\U{\mathbb{U}}
\def\N{\mathbb{N}}
\newcommand{\R}{\mathbb{R}}
\newcommand{\D}{\mathbb{D}}
\newcommand{\CC}{\mathbb{C}}

\def\S{\mathbb{S}}
\def\e{{\rm e}}
\def\d{{\rm d}}

\def\eps{\varepsilon}

\def\PPP{\mathcal P}
\def\FFF{\mathcal F}

\def\TTT{\mathcal T}

\def\EEE{\mathcal E}
\def\tmub{\widetilde{\mu}_{\bullet}}

\def\Vc{\mathcal{V}}
\def\Vcd{\mathcal{V}^{\mathsf{d}}}

\def\i{\mathrm{i}}

\def\FKn{\mathcal{F}^{(n)}_{K_{n}}}
\def\PKn{\mathcal{P}^{(n)}_{K_{n}}}
\def\dFKn{\dot{\mathcal{F}}^{(n)}_{K_{n}}}
\def\dPKn{\dot{\mathcal{P}}^{(n)}_{K_{n}}}
\def\dFKphin{\dot{\mathcal{F}}^{(\phi(n))}_{K_{\phi(n)}}}
\def\dPKphin{\dot{\mathcal{F}}^{(\phi(n))}_{K_{\phi(n)}}}
\def\dFKphin{\dot{\mathcal{P}}^{(\phi(n))}_{K_{\phi(n)}}}
\def\dPKphin{\dot{\mathcal{P}}^{(\phi(n))}_{K_{\phi(n)}}}
\def\dFn{\dot{\mathcal{F}}_{n}}
\def\dPn{\dot{\mathcal{P}}_{n}}
\def\ab{a_{\bullet}}
\def\ac{a_{\circ}}
\def\bb{b_{\bullet}}
\def\bc{b_{\circ}}
\def\sc{\sigma_{\circ}}
\def\sb{\sigma_{\bullet}}

\def\abn{a_{\bullet}^{n}}
\def\acn{a_{\circ}^{n}}
\def\bbn{b_{\bullet}^{n}}
\def\bcn{b_{\circ}^{n}}
\def\mub{\mu_{\bullet}}
\def\muc{\mu_{\circ}}

\def\mubn{\mu_{\bullet}^{n}}
\def\mucn{\mu_{\circ}^{n}}

\def\mbn{m_{\bullet}^{n}}
\def\mcn{m_{\circ}^{n}}

\def\scn{\sigma_{\circ}^{n}}
\def\sbn{\sigma_{\bullet}^{n}}

\def\Sbn{S^{\bullet,n}}
\def\Scn{S^{\circ,n}}

\def\phibn{\phi^{\bullet}_{n}}
\def\phicn{\phi^{\circ}_{n}}

\def\DbnN{D^{\bullet}_{n,N}}
\def\Dbn{D^{\bullet}_{n}}

\def\Dcn{D^{\circ}_{n}}

\def\nc{n_{\circ}}
\def\nb{n_{\bullet}}

\def\oB{\overline{B}}
\def\hB{\widehat{B}}
\def\oH{\overline{H}}
\def\hH{\widehat{H}}

\def \Xexc{X^{\mathrm{exc}}}
\def \Xbr{X^{\mathrm{br}}}

\def\tTs{\tilde{\Ts}}
\def \Tsb{\Ts^{\bullet}}
\def \Tsc{\Ts^{\circ}}
\def \Tsbn{\Ts^{\bullet,n}}
\def \Tscn{\Ts^{\circ,n}}
\newcommand\BGW{\textrm{BGW}}

\def\tb{\mathcalbis{t}}
\def\Ts{\mathscr{T}}
\renewcommand{\epsilon}{\varepsilon}

\newcommand\Es[1]{\mathbb{E}\left[#1\right]}
\newcommand\Esb[1]{\mathbb{E}\big[#1\big]}

\renewcommand\Pr[1]{\mathbb{P}\left(#1\right)}
\newcommand\tPr[1]{\tilde{\mathbb{P}}\left(#1\right)}
\newcommand\Prb[1]{\mathbb{P}\big(#1\big)}

\def\llbracket{[\hspace{-.10em} [ }
\def\rrbracket{ ] \hspace{-.10em}]}

\def\build#1_#2^#3{\mathrel{
\mathop{\kern 0pt#1}\limits_{#2}^{#3}}}

\def\One{\mathbbm{1}}

\author[\initial{V.} \lastname{F\'eray}]{\firstname{Valentin} \lastname{F\'eray}}
\address{Universit\"at Z\"urich,
Institut f\"ur Mathematik,
Winterthurerstr. 190,
CH-8057 Z\"urich}
\email{valentin.feray@math.uzh.ch}

\author[\initial{I.} \lastname{Kortchemski}]{\firstname{Igor} \lastname{Kortchemski}}
\address{CNRS \& CMAP, \'Ecole polytechnique, route de Saclay
F-91128 Palaiseau Cedex}
\email{igor.kortchemski@math.cnrs.fr}

\title[Random minimal factorizations of a long cycle]{The geometry of random minimal factorizations
of a long cycle\\
via\\
biconditioned bitype random trees}
\alttitle{La géométrie des factorisations minimales aléatoires d'un long cycle\\
via\\
des arbres aléatoires bitype doublement conditionnés}

\begin{abstract}
We study random typical minimal factorizations of the $n$-cycle into transpositions, which are factorizations of $(1, \ldots,n)$ as a product of $n-1$ transpositions. By viewing transpositions as chords of the unit disk and by reading them one after the other, one obtains a sequence of increasing laminations of the unit disk (i.e.\ compact subsets of the unit disk made of non-intersecting chords). 

When an order of $\sqrt{n}$ consecutive transpositions have been read, we establish, roughly speaking, that a phase transition occurs and that the associated laminations converge  to a new one-parameter family of random laminations,  constructed from excursions of specific Lévy processes.

Our main tools involve coding random minimal factorizations by conditioned two-type Bienaym\'e--Galton--Watson trees.
We establish in particular limit theorems for two-type BGW trees 
conditioned on having given numbers of vertices of both types,
and with an offspring distribution depending on the conditioning size.
We believe that this could be of independent interest. 
\end{abstract}

\begin{altabstract}
  Nous étudions les factorisations minimales aléatoires d'un $n$-cycle en transpositions,
  c'est-à-dire les décompositions de $(1,\dots,n)$ comme un produit de $n-1$ transpositions.
  En représentant les transpositions comme des cordes du disque unité
  et en les lisant les unes après les autres,
  on obtient une suite croissantes de laminations du disque unité
  (i.e. de sous-ensembles compacts du disque unité constitués de cordes ne se croisant pas).

  Quand le nombre de transpositions lues est de l'ordre $\sqrt{n}$,
  nous établissons l'existence d'une transition de phase et la convergence des laminations associées
  vers une nouvelle famille de laminations aléatoires à un paramètre,
  construites à partir de processus de Lévy.

  Notre outil principle est le codage de ces factorisations minimal aléatoires par
  des arbres de Bienaym\'e--Galton--Watson bitype conditionnés.
  En particulier, nous obtenons des théorèmes limites pour de tels arbres,
  conditionnés à avoir un nombre fixe de nœuds de chaque couleur,
  et dont la loi de reproduction dépend de la taille à laquelle on conditionne.
  Nous pensons que ce résultat est aussi intéressant en lui-même.
\end{altabstract}


\begin{document}

\maketitle
\bigskip

\begin{figure}[h!]
 \centering
    \makeatletter\edef\animcnt{\the\@anim@num}\makeatother
\animategraphics[label=myAnim,scale=0.4]{20}{film/image}{0}{100} 
   \mediabutton[jsaction={anim.myAnim.playFwd();}]{\scalebox{1.5}[1.2]{\strut $\vartriangleright$}}
   \mediabutton[jsaction={anim.myAnim.pause();}]{\scalebox{1.5}[1.2]{\strut $\shortparallel$}}
    \caption{For $n=10000$, we pick a minimal factorization $(\tb^{(n)}_{1}, \tb^{(n)}_{2}, \ldots, \tb^{(n)}_{n-1})$
    of the $n$-cycle into transpositions, uniformly at random.
    With each transposition $\tau=(j,j')$, we associate the chord $[e^{-2 \pi i\, j/n}, e^{-2 \pi i\, j'/n} ]$ of the unit disk.
    For $c>0$, we consider the union of the chords associated with the first $\lfloor c \sqrt{n} \rfloor$
    transpositions in our factorization.
    The above animation (played with Acrobat Reader) shows the resulting picture for various values of $c$.}
\end{figure}

\clearpage

\section{Introduction}

\subsection{Model and motivation}
We are interested in the geometric structure of  typical minimal factorizations of the $n$-cycle as $n \rightarrow \infty$, seen as compact subsets of the unit disk. More precisely,  for an integer $n \geq 1$, we denote by $ \mathfrak{S}_{n}$  the symmetric group acting on $[n] \coloneqq \{1,2, \ldots,n\}$ and we let $ \mathfrak{T}_{n}$ be the set of all transpositions of $ \mathfrak{S}_{n}$. We denote by $(1,2, \ldots,n)$ the $n$-cycle which maps $i$ to $i+1$ for $1 \leq i \leq n-1$. The elements of the set
\[ \mathfrak{M}_{n} \coloneqq  \left\{ (\tau_{1}, \ldots, \tau_{n-1}) \in \mathfrak{T}_{n}^{n-1} : \tau_{1} \tau_{2} \cdots \tau_{n-1}= (1,2, \ldots,n)  \right\}\]
are called \emph{minimal factorizations of  $(1,2, \ldots,n)$ into transpositions} (since, as is easily seen, at least $n-1$ transpositions are required to factorize a $n$-cycle). To simplify, in the sequel, elements of $ \mathfrak{M}_{n}$ will be simply called \emph{minimal factorizations of size $n$}.  It is known since D\'enes \cite{Den59} that
\begin{equation}
\label{eq:cardFn}  |\mathfrak{M}_{n}|  =n^{n-2},
\end{equation}
and bijective proofs were given by Moszkowski \cite{Mos89},
Goulden \& Pepper \cite{GP93}, Goulden \& Yong \cite{GY02} and Biane \cite{Biane2004}.
Biane also gave an interpretation of minimal factorizations
as maximal chains in the non-crossing partition lattice \cite{Biane1997}
and an elementary bijection with parking functions \cite{Biane2002}.
It is also worth noticing that many variants of this simple enumeration problem
have been considered, going in multiple directions.
Here is a non-exhaustive list of references:
\begin{itemize}
  \item using different sets of generators, such as adjacent transpositions \cite{Sta84,EG87},
    star transpositions \cite{Pak1999,IrvingRattan2009,GouldenJackson2009,Fer12} or cycles of given length \cite{Biane2004};
  \item considering non-minimal factorizations and/or factorizations of more general
    permutations than the full cycle (usually with an extra transitivity assumption).
    Then the numbers of such factorizations are the celebrated Hurwitz numbers,
    which also 
    count ramified coverings of the sphere \cite{Hur91}.
    Various methods have been used to evaluate these numbers:
    the cut-and-join recurrence \cite{GJ99A,GJ99B},
    representation theory \cite{Jac88,SSV97}, enumerative geometry 
    (the ELSV formula expresses them as integral over moduli spaces \cite{ELSV})
    and more recently the frameworks of integrable hierarchy
    (an appropriate generating function of Hurwitz numbers
    is a solution of the so-called KP-hierarchy \cite{Oko00})
    and topological recursion (as developed by Eynard and Orantin, see \cite{ACEH16}
    and references therein).
\end{itemize}

In this article, we focus on the model of minimal factorizations of a cycle and investigate 
the asymptotic behaviour of a large minimal factorizations taken uniformly at random.

Products of random transpositions have been studied in various contexts:
\begin{itemize}
\item generation of a random permutation with random transpositions \cite{DS81,Scr05,DMZZ04} and random transposition random walks \cite{Tot93,BD06,Ber11}.
\item In a direction more closely related to our work, Angel, Holroyd, Romik \& Vir\'ag \cite{AHRV07} have initiated the study of large uniform random factorizations of the reverse permutation $\rho$ (defined by $\rho(i)=n+i-1$ for $1 \leq i \leq n$) by using only nearest-neighbor transpositions:
  \hbox{$\rho=s_1 \dots s_N$},
  where $N=\binom{n}{2}$ is the minimal possible number of factors in such a factorization.
They conjectured a formula for the limiting process of partial products $s_1 \cdots s_{\lfloor c N \rfloor}$,
the proof of which was recently announced in \cite{dauvergne2018archimedean}.
Note that both models of minimal factorizations are combinatorially very different and
that the work of Angel, Holroyd, Romik \& Vir\'ag only serves as inspiration in ours.
\end{itemize}

We study the asymptotic behavior of a typical minimal factorization $ \mathscr{F}^{(n)}$ in two different, but related, directions. We first analyze several general properties of $ \mathscr{F}^{(n)}$
(in particular, the laws of single factors and of partial products)
and then study geometrically the structure of the first $K_{n}$ transpositions of $ \mathscr{F}^{(n)}$
(with $K_{n} \rightarrow \infty$).
Local properties of $\mathscr{F}^{(n)}$, in particular, the trajectory
of a given $i$ in a random minimal factorization,
are investigated in a companion article \cite{FK18}.
\smallskip

\paragraph*{Important convention.} Throughout the article, the multiplication $\sigma\, \tau$ of permutations
should be understood as the composition $\tau \circ \sigma$, {\em i.e.} 
we apply first $\sigma$ and then $\tau$.
In other words, we apply permutations from left to right.
This convention has no incidence on the general shape of our results.

\subsection{An explicit formula for partial products}
As suggested above, we denote by \hbox{$\mathscr{F}^{(n)}= (\tb_{1}^{(n)}, \ldots, \tb_{n-1}^{(n)})$}
a uniform random minimal factorization of size $n$.
We also fix some positive integer $k \le n$ and consider the partial product
$\tb_{1}^{(n)} \cdots \tb_k^{(n)}$.
It turns out that the law of $\tb_{1}^{(n)} \cdots \tb_k^{(n)}$ can be determined explicitly
and that this is a cornerstone for all results in this paper.

To present this explicit formula, we need to introduce some terminology.
 Recall that a \emph{partition} of $[n] \coloneqq \{1,2, \ldots,n\}$ is a collection of (pairwise) disjoint subsets,
 called \emph{blocks}, whose union is $[n]$.
 The \emph{size} of a block $B$, denoted by $|B|$, is its number of elements.
 A \emph{non-crossing partition} of $[n]$ is a partition of the vertices of a regular $n$-gon (labeled by the set $[n]$ in clockwise order) with the property that the convex hulls of its blocks are pairwise disjoint (see the right hand-side of Figure~\ref{fig:observation} for an example).
 We let $ \mathfrak{P}_{n}$ be the set of all non-crossing partitions of $[n]$.
 If $P$ is a non-crossing partition of $[n]$, we denote by $ \mathcal{K}(P)$ the \emph{Kreweras complement} of $P$, see Section~\ref{sec:ncp} below for a definition. 

By considering the blocks obtained from the cycle decomposition of a permutation,
one can naturally associate with $\sigma \in \mathfrak{S}_{n}$ a partition $\PPP(\sigma)$ of $[n]$.
Of course, this map is in general  not injective, but when we restrict it to partial products of minimal factorizations,
it becomes injective. Determining the law of $\tb_{1}^{(n)} \cdots \tb_k^{(n)}$ is therefore equivalent to determining
the law of $\PPP(\tb_{1}^{(n)} \cdots \tb_{k}^{(n)})$.
It is well known (see e.g.\ \cite[Theorem 1]{Biane1997}) that if $(\tau_{1}, \ldots, \tau_{n-1})  \in \mathfrak{M}_{n}$ is a minimal factorization, then for every $1 \leq k \leq n-1$, 
the partition $\PPP(\tau_{1} \tau_{2} \cdots \tau_{k})$ 
has $n-k$ blocks and is non-crossing. 

\begin{proposition}
\label{prop:lawproduct}Fix $1 \leq k \leq n-1$ and let $P \in \mathfrak{P}_{n}$ be a non-crossing partition with $n-k$ blocks. Then 
\[\Pr{\PPP(\tb_{1}^{(n)} \tb_{2}^{(n)} \cdots \tb_{k}^{(n)})=P}= \frac{k! (n-k-1)!}{n^{n-2}}  \cdot  \left( \prod_{B \in P} \frac{|B|^{|B|-2}}{(|B|-1)!} \right) \cdot \left(  \prod_{B \in  \mathcal{K}(P)} \frac{|B|^{|B|-2}}{(|B|-1)!} \right).\]
\end{proposition}

The proof of this proposition is rather elementary by combining simple results on minimal
factorizations and formula \eqref{eq:cardFn}; see Section~\ref{ssec:ncp}.
As a by-product, we can deduce several properties of the law of $ \mathscr{F}^{(n)}$,
namely a stationarity property and the law of the first component.
(Similar results have been obtained for random minimal factorizations of the reverse permutation
into adjacent transpositions in \cite{AHRV07}.)
\begin{corollary}\label{cor:marginals} Fix $n \geq 3$. The following assertions hold.
\begin{enumerate}
\item[(i)] The two random variables $(\tb_{1}^{(n)}, \ldots, \tb_{n-2}^{(n)})$ and $(\tb_{2}^{(n)}, \ldots, \tb^{(n)}_{n-1})$ have the same distribution.
\item[(ii)] We have, for $1 \leq a  \leq n-1$ and $ 1 \leq i \leq n-a$:
\begin{equation}
\label{eq:lawT1}\Pr{\tb_{1}^{(n)}=(a,a+i)}= \frac{(n-2)!}{n^{n-2}} \cdot \frac{i^{i-2}}{(i-1)!} \cdot \frac{(n-i)^{(n-i-2)}} {(n-i-1)!}.	
\end{equation}
In particular,  if $\tb_{1}^{(n)}=(a_{n},b_{n})$  with $a_{n}<b_{n}$, then for every $i \geq 1$,
\begin{equation}
  \Pr{b_{n}-a_{n}=i}  \quad \mathop{\longrightarrow}_{n \rightarrow \infty} \quad  \frac{i^{i-2}}{(i-1)!} e^{-i}.
  \label{eq:limit_length-first_transpo}
\end{equation}
In addition, $a_{n}/n$ converges in distribution to a uniform random variable on $[0,1]$.
\end{enumerate}	
\end{corollary}
Corollary~\ref{cor:marginals} is proved in Section~\ref{ssec:ncp} as well.
By stationarity, (ii) also holds for $\tb_{k}^{(n)}$ for any $1 \le k \le n-1$.
Moreover, note that the limiting probability distribution 
\eqref{eq:limit_length-first_transpo} is the Borel distribution of parameter $1$;
in particular $b_n-a_n$ converges in distribution.

Proposition~\ref{prop:lawproduct} will also be an important tool for the geometric  results of the next section.
A key observation is the following: it is known that non-crossing partitions are in bijection with plane trees;
using this bijection, Proposition~\ref{prop:lawproduct} tells us that 
$\PPP(\tb_{1}^{(n)} \tb_{2}^{(n)} \cdots \tb_{k}^{(n)})$ can be encoded by 
a conditioned
bi-type Bienaym\'e--Galton--Watson tree.
We will discuss this further in the description of the proof strategy in the next section.

\subsection{Minimal factorizations seen as compact subsets of the unit disk}
\label{sec:introcompacts}

We denote by $\D = \{z \in \CC : |z| < 1\}$ the open unit disk of the complex plane, by $\S = \{z \in \CC : |z| = 1\}$ the unit circle and by $\overline{\D} = \D \cup \S$ the closed unit disk. For every $x, y \in \S$, we write $[x,y]$ for the line segment, or chord, between $x$ and $y$ in $\overline{\D}$, with the convention $[x,x] = \{x\}$ (which is a chord by convention).   
With every graph $G$ with vertex set $[n]$ and edge set $E_G$, we associate a compact subset $\dot{G}$ of  $\overline{\D}$ 
defined as
\[\dot{G} \coloneqq \bigcup_{ \{j,j'\} \in E_G}[e^{2 \pi i\, j/n}, e^{2 \pi i\, j'/n} ].\]
In words, each edge is represented by a chord in $\dot{G}$.
We take the convention that isolated vertices of $G$ do not appear in $\dot{G}$.
 
We can now present two natural ways to see a minimal factorization $(\tau_1,\dots,\tau_{n-1})$ as 
{\em a process of subsets of $\overline{\D}$}.
In both representations, at ``time'' $k$ we will consider the $k$ first factors 
$(\tau_1,\dots,\tau_k)$ of the minimal factorization
and associate with this data a subset of $\overline{\D}$.

The first way simply consists  in interpreting each transposition as an edge of a graph.
More formally,
if $\tau_{1}, \tau_{2}, \ldots, \tau_{k} \in \mathfrak{S}_{n}$ are transpositions, we denote by $\FFF(\tau_{1} ,\tau_{2} ,\ldots ,\tau_{k})$ the graph with vertex set $[n]$ and edge set $\{\tau_{1}, \tau_{2}, \ldots, \tau_{k}\}$.
Then $\dot{\FFF}(\tau_{1} ,\tau_{2} ,\ldots ,\tau_{k})$ is its associated compact subset of $\overline{\D}$
(see the left part of Figure~\ref{fig:observation} for an example). 
When $(\tau_{1}, \ldots, \tau_{n-1})$ is a minimal factorization,
it is well-known (see e.g.\ \cite[Theorem 2.2]{GY02}) that $\FFF(\tau_{1} ,\tau_{2} ,\ldots ,\tau_{k})$ is a non-crossing forest 
(in the sense that its connected components are trees and that edges do not cross in  $\dot{\FFF}(\tau_{1} ,\tau_{2} ,\ldots ,\tau_{k})$). 
For $k=n-1$, note that $\FFF(\tau_{1}, \tau_{2}, \ldots, \tau_{n-1})$ is connected and is therefore a non-crossing tree.

\begin{figure}[t] 
%%% Partition
\begin{scriptsize}
\[
\begin{tikzpicture}[scale=.8]
\draw[thin, dashed]	(0,0) circle (2);
\foreach \x in {1, 2, ..., 12}
	\coordinate (\x) at (-\x*360/12 : 2);
\foreach \x in {1, 2, ..., 12}
	\draw
	%[fill=black]	(\x) circle (1pt)
	(-\x*360/12 : 2*1.1) node {\x}
;
\draw	
(1) -- (3)
(1) -- (5)
(6) -- (12)
(7) -- (12)
(9) -- (10)
(11) -- (12);
\end{tikzpicture}
\qquad \qquad
\begin{tikzpicture}[scale=.8]
\draw[thin, dashed]	(0,0) circle (2);
\foreach \x in {1, 2, ..., 12}
	\coordinate (\x) at (-\x*360/12 : 2);
\foreach \x in {1, 2, ..., 12}
	\draw
	%[fill=black]	(\x) circle (1pt)
	(-\x*360/12 : 2*1.1) node {\x}
;
\draw	(1) -- (3) -- (5) -- cycle
	(6) -- (7) -- (11) -- (12) -- cycle
	(9) -- (10)
;
\end{tikzpicture}\]
\[(t_{1}^{(12)}, \ldots, t_{11}^{(12)})=\big( (1, 3), (6, 12), (1, 5), (7, 12), (9, 10), (11, 12), (2, 3), (4,  5), (1, 6), (8, 11), (9, 11)\big) \in \mathfrak{M}_{12}.\]
\end{scriptsize}
\caption{\label{fig:observation} In this example, we take $n=12$, $k=6$ and $(t_i^{(12)})_{i \le 11}$ as above.
We have $ t_{1}^{(12)} t_{2}^{(12)} \cdots t_{6}^{(12)}=(1,3,5)(6,7,11,12)(9,10)$ (recall that we multiply from left to right), so that  $\PPP_{6}= \{  \{1,3,5\}, \{2\}, \{4\}, \{6,7,11,12\}, \{8\}, \{9,10\} \} $.
On the left, we have represented $\dot{\FFF}_{6}$, and on the right $\dot{\PPP}_{6}$.}
\end{figure}
There is a second way to represent the transpositions 
$\tau_{1}, \tau_{2}, \ldots, \tau_{k} \in \mathfrak{S}_{n}$ as a compact subset of $\overline{\D}$.
We consider the product  $\sigma_k=\tau_{1} \tau_{2}\cdots \tau_{k}$ of these transpositions
and its associated non-crossing partition $ \PPP(\tau_{1} \tau_{2}\cdots \tau_{k})$.
We denote by $ \dot{\PPP}(\tau_{1} \tau_{2}\cdots \tau_{k})$ its associated  compact subset of $\overline{\D}$, defined as the union of chords $[{\rm e}^{-2{\rm i}\pi \ell/n}, {\rm e}^{-2{\rm i}\pi \ell'/n}]$ whenever $\ell, \ell' \in [n]$ are two consecutive elements of the same block of the partition (where the smallest and the largest element of a block are consecutive by convention).
Note that by definition singleton blocks do not appear in $ \dot{\PPP}(\tau_{1} \tau_{2}\cdots \tau_{k})$;
see the right part of Figure~\ref{fig:observation} for an example.  

(This point of view is closer to the one used in \cite{AHRV07} on random minimal factorizations
of the decreasing permutation through adjacent transpositions.
Indeed, in the latter reference, the authors also study the product $\sigma_k$ of the $k$ first factors, for various $k$ depending on $n$.
Then they represent this product as a set of dots $(\tfrac{i}{n},\tfrac{\sigma_k(i)}{n})$ in the square,
while in our case, the geometric representation as a non-crossing partition seems more natural.)

The second representation carries less information:
$\FFF(\tau_{1}, \tau_{2}, \ldots, \tau_{k})$ cannot be recovered from $\PPP(\tau_{1} \tau_{2}\cdots \tau_{k})$
since the same permutation $\sigma_k$ can be factorized as a product $\tau_{1} \tau_{2}\cdots \tau_{k}$ in different ways.
Conversely, despite the fact that $\FFF(\tau_{1}, \tau_{2}, \ldots, \tau_{k})$ forgets about
the order of the transpositions $\tau_1,\dots,\tau_k$ it is possible to reconstruct $\PPP(\tau_{1} \tau_{2}\cdots \tau_{k})$
from it.
Indeed,
it is easily seen that $ \dot{\PPP}(\tau_{1} \tau_{2}\cdots \tau_{k})$ is obtained from
$\dot{\FFF}(\tau_{1} ,\tau_{2} ,\ldots ,\tau_{k})$ by replacing each connected component by the convex hull of its vertices (see Figure~\ref{fig:observation} again for an example).

We now consider a uniform random minimal factorization of size $n$, that we denote by
$(\tb_{1}^{(n)}, \ldots, \tb_{n-1}^{(n)})$.
To simplify notation, set, for $1 \leq k \leq n-1$,
\[ \FFF^{(n)}_{k}= \FFF ( \tb_{1}^{(n)}, \ldots, \tb_{k}^{(n)} ), \qquad  \PPP^{(n)}_{k}= \PPP ( \tb_{1}^{(n)}  \cdots \tb_{k}^{(n)} ).\]
Recall that  $\dot{\FFF}^{(n)}_{k}$ and $ \dot{\PPP}^{(n)}_{k}$ denote their associated compact subsets of $\overline{\D}$.

Our main result deals with the limit in distribution of these objects in the space  $(\mathbb{K},d_{H})$, where  $\mathbb{K}$ is the space of all compact subsets of $\overline{\D}$  equipped with the Hausdorff distance $d_{H}$. It is standard that $(\mathbb{K},d_{H})$ is compact.

\begin{theorem}\label{thm:cvlam} 
  Fix $c \in [0,\infty]$. There exists a random compact subset $\mathbf{L}_{c}$ of $\overline{\mathbb{D}}$ such that the following holds. Let $(K_n)_{n \geq 1}$ be a sequence of positive integers with $K_n \le n-1$ for every $n \geq 1$.
\begin{enumerate}
\item[(i)] Assume that  $K_{n} \rightarrow \infty$ and $ \frac{K_{n}}{\sqrt{n}} \rightarrow c$ as $n \rightarrow \infty$, with $c < \infty$. Then
 the following convergence holds jointly in distribution in  $\mathbb{K}^{2}$:
 \[ \big( \dFKn  ,  \dPKn \big)   \quad \mathop{\longrightarrow}^{(d)}_{n \rightarrow \infty} \quad ( \mathbf{L}_{c},\mathbf{L}_{c}).\]
\item[(ii)] Assume that  $ \frac{K_{n}}{\sqrt{n}} \rightarrow \infty$ and that $ \frac{n-K_{n}}{\sqrt{n}} \rightarrow \infty$ as $n \rightarrow \infty$. Then
 the following convergence holds jointly in distribution in  $\mathbb{K}^{2}$:
 \[ \big(\dFKn  , \dPKn \big)   \quad \mathop{\longrightarrow}^{(d)}_{n \rightarrow \infty} \quad ( \mathbf{L}_{\infty},\mathbf{L}_{\infty}).\]
\item[(iii)] Assume that $ \frac{n-K_{n}}{\sqrt{n}} \rightarrow c$ as $n \rightarrow \infty$, with $c <\infty$. Then the following convergences  hold in distribution in $\mathbb{K}$:
\[ \dFKn    \quad \mathop{\longrightarrow}_{n \rightarrow \infty}^{(d)} \quad \mathbf{L}_{\infty}, \qquad  \dPKn \quad \mathop{\longrightarrow}^{(d)}_{n \rightarrow \infty} \quad \mathbf{L}_{c}.\]
\end{enumerate}
\end{theorem}

 \begin{figure}[thb]
 \begin{center}
 \includegraphics[height=4cm]{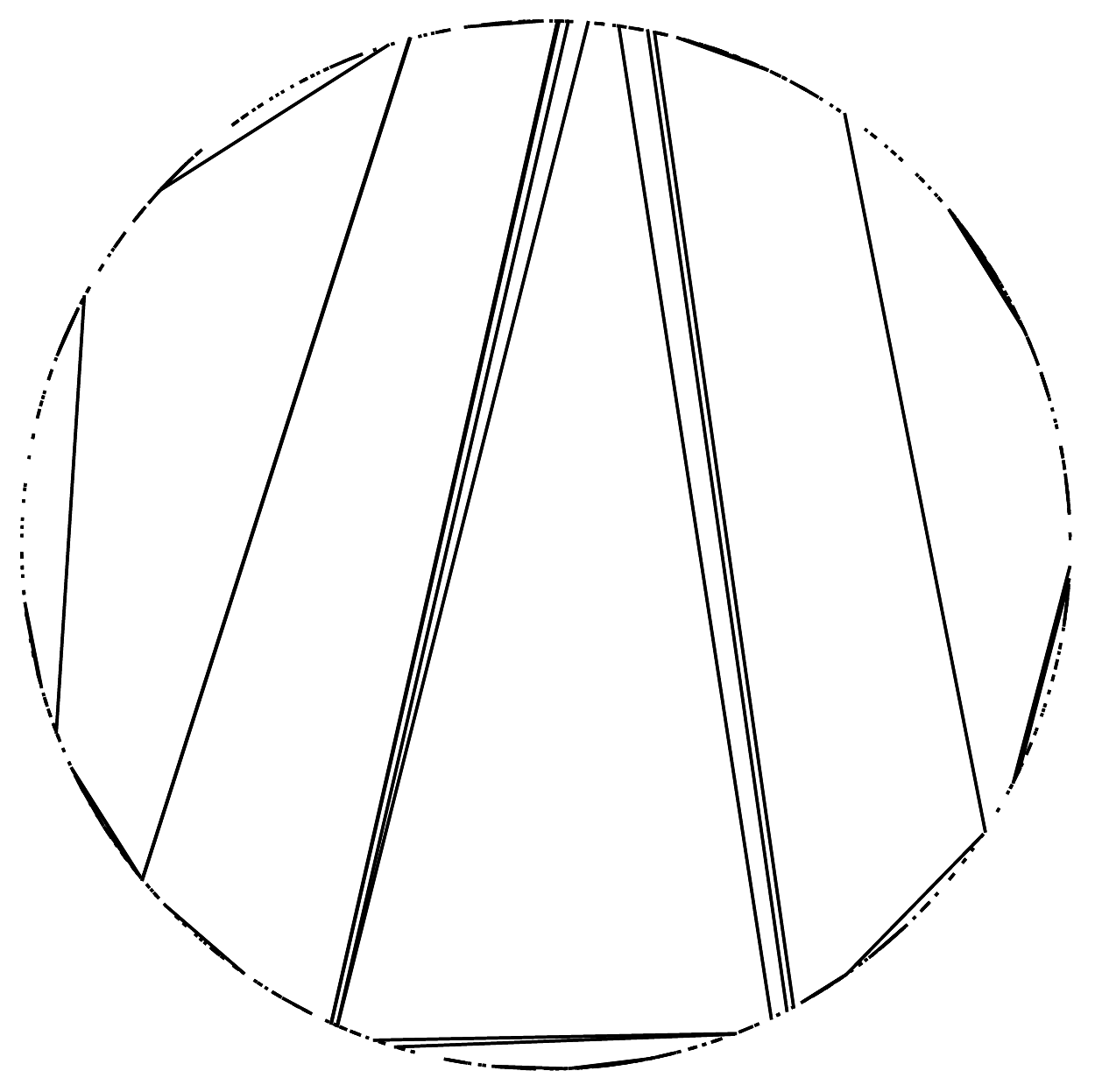} \qquad   
  \includegraphics[height=4cm]{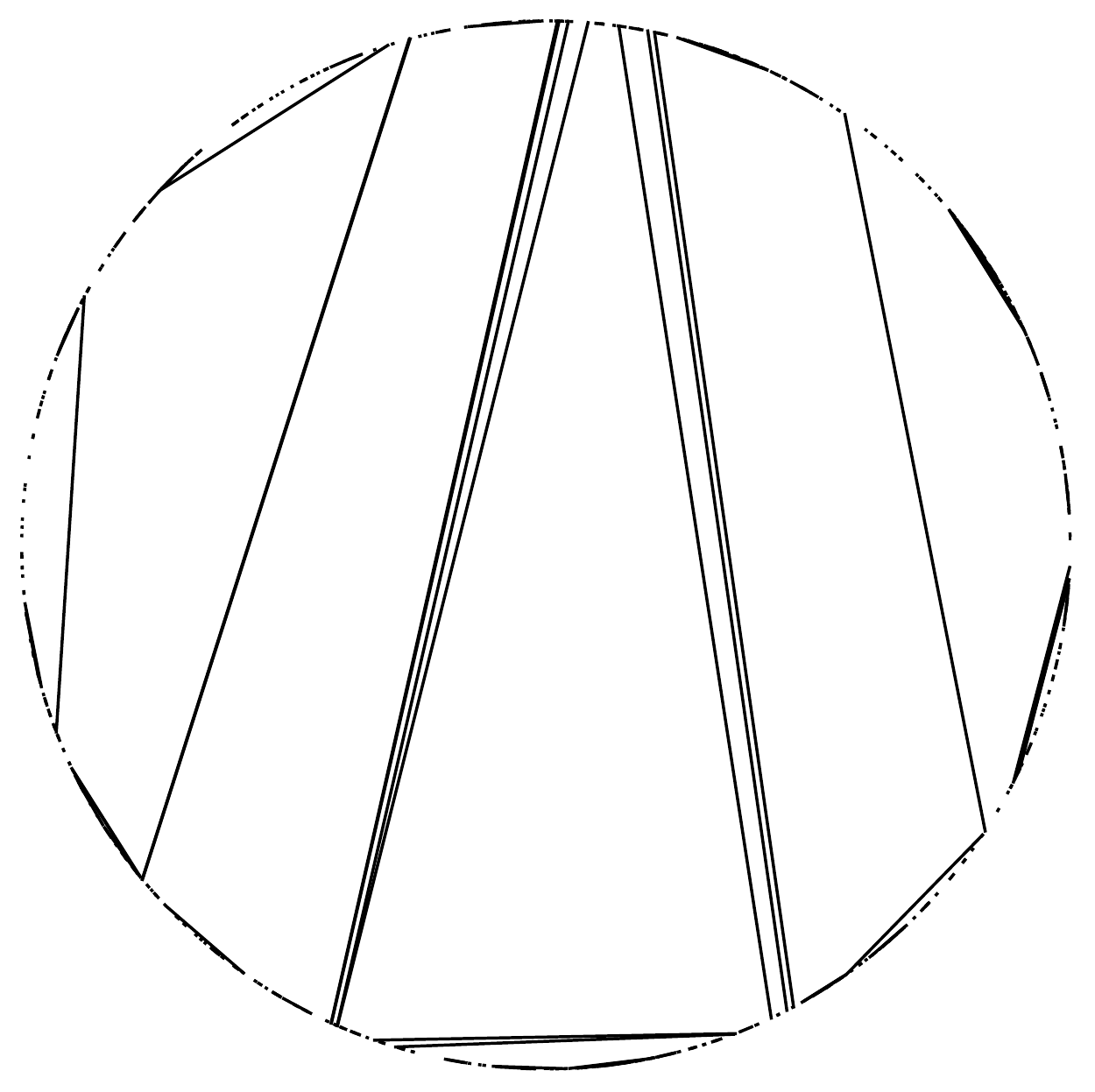} \quad
  \includegraphics[height=4cm]{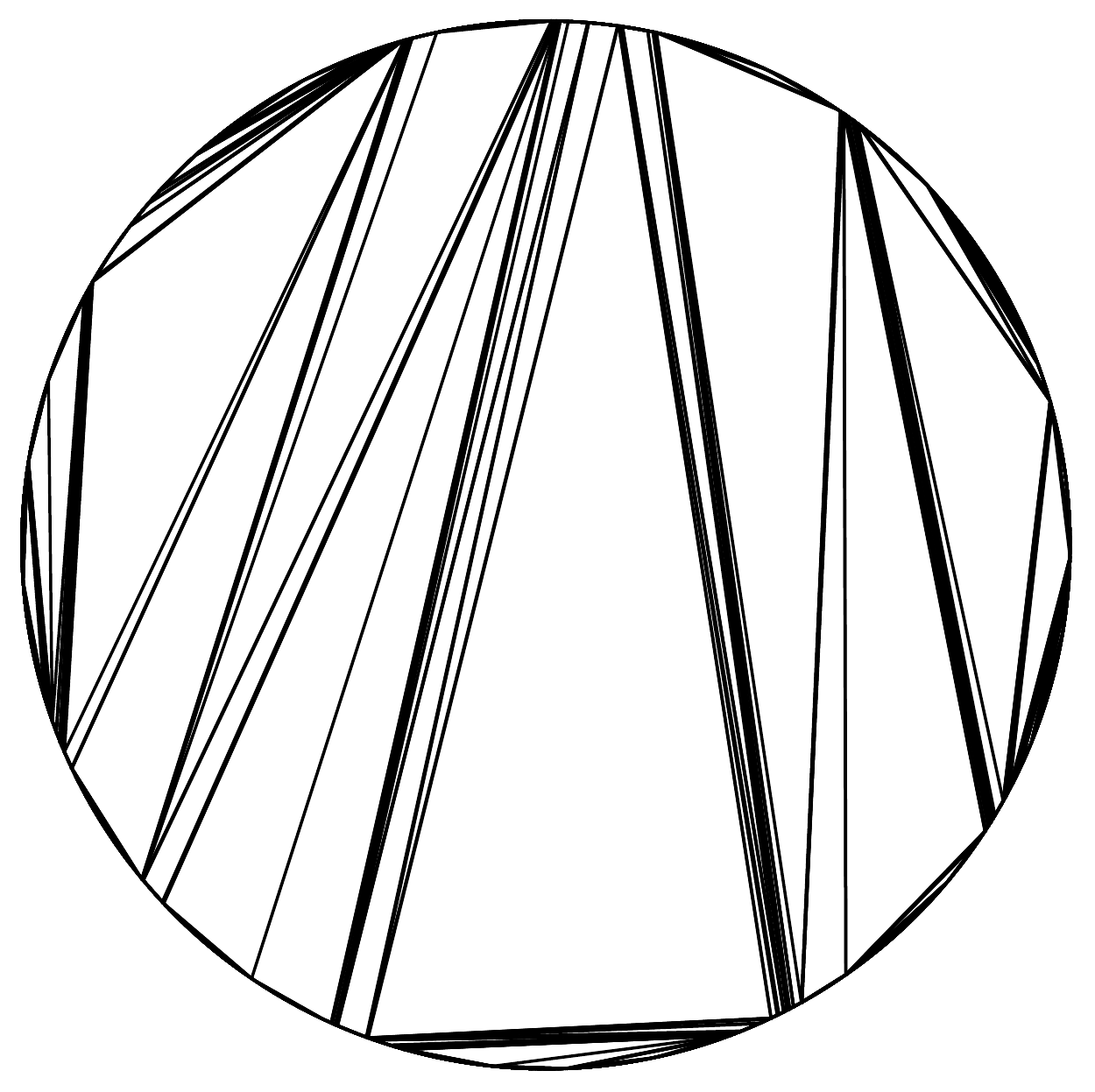}
  \caption{\label{fig:lam}For $n=10000$ and $ K_{n} =500=5\sqrt{n}$
  we have taken uniformly at random a minimal factorization $(\tb_{1}^{(n)}, \ldots, \tb_{n-1}^{(n)})$
  of size $n$ and we have represented from left to right:  
  the non-crossing partition $ \dot{\PPP}^{(n)}_{K_{n}}$,
  the non-crossing forest $\dot{\FFF}^{(n)}_{{K_{n}}}$
  and the non-crossing tree $\dot{\FFF}^{(n)}_{n-1}$.
  Our limit theorem asserts that the first two pictures are approximations of the same realization
  of some random lamination $\mathbf{L}_5$
  (we indeed observe that they look similar to each other),
  while the last one approximates the Brownian triangulation $\mathbf{L}_\infty$.}
 \end{center}
 \end{figure}

We observe that a phase transition occurs when $K_n$ is of order $\sqrt{n}$.
Let us give a heuristic explanation of this fact.
For each $n$, the transpositions $\tb_{1}^{(n)}, \tb_{2}^{(n)}, \dots$ are identically distributed (Corollary~\ref{cor:marginals})
and the limiting distribution of their ``length'' (see \eqref{eq:limit_length-first_transpo}) has a heavy tail with index $1/2$.
If $\tb_{1}^{(n)}, \ldots, \tb_{ K_{n}}^{(n)}$ were independent and distributed as \eqref{eq:limit_length-first_transpo},
one would have to observe $\sqrt{n}$ such independent random variables to have one of length of order $n$.
We therefore expect that the subsets $\dFKn$ and $\dPKn$ become non trivial for $K_n$ of order $\sqrt{n}$.
\medskip

The limiting object $\mathbf{L}_{c}$  is, in all cases, a \emph{lamination},
that is a compact subset of $\overline{\D}$ which is  the union of non-crossing chords. Let us mention that $\mathbf{L}_{0}$ is simply the unit circle $\mathbb{S}$, that $\mathbf{L}_{\infty}$ is the so-called Brownian triangulation (see below), and that $(\mathbf{L}_{c})_{c  > 0}$ is a new one-parameter family of random laminations. See Section~\ref{sec:deflam} for  precise definitions. 

Note that case (i) can be subdivided in two cases, denoted by (i)$_{c=0}$ and  (i)$_{c>0}$. In the first one, the distributional limit is deterministic, while in the second one it is random.
In contrast with (i) and (ii),  observe that in case (iii), 
$\dFKn$ and $\dPKn$ do \emph{not} converge to the same limit.
This can be  explained by the fact that while $k \mapsto \dot{\FFF}^{(n)}_{k}$ is increasing,
$k \mapsto \dot{\PPP}^{(n)}_{k}$ is, roughly speaking, asymptotically increasing, then decreasing. 
We cannot prove any joint convergence results in this case
(the limit should be a coupling of $\mathbf{L}_{c}$ and $\mathbf{L}_{\infty}$
such that $\mathbf{L}_{c} \subset \mathbf{L}_{\infty}$ a.s.,
which excludes the possibility of an independent coupling).

The idea of viewing random graphs as random compact subsets of the unit disk and of studying their convergence in the space $(\mathbb{K},d_{H})$ goes back to Aldous \cite{Ald94a,Ald94b}, who constructed the Brownian triangulation and showed that it is the limit of random uniform triangulations of the $n$-gon.
Since then, the Brownian triangulation has been showed to be the universal limit of various non-crossing discrete structures \cite{Kor14,CK14,KM17,Bet17} and has appeared in the context of random planar maps \cite{LGP08}.
Other random laminations, mostly arising as limits of different natural discrete combinatorial structures, have been constructed in \cite{CLGrecursive,CWmht,Kor14,KM16}.
As in \cite{Kor14}, our random laminations are obtained by applying a deterministic functional  to some random excursion
(here, an excursion is a nonnegative valued càdlàg function on $[0,1]$).
But the random excursions we are starting from are different from the ones in \cite{Kor14}:
in our case, they correspond to a normalized excursion of a Lévy process with
an explicit characteristic function which is not stable, see Eq.~\eqref{eq:Levyc} below.
Let us mention that the laminations appearing in \cite{CLGrecursive,CWmht,KM16} are constructed in a different way.

\paragraph*{Strategy of the proof.} Let us now comment on the strategy of the proof of  Theorem~\ref{thm:cvlam}. We first establish, in case (i)$_{c=0}$, the convergence $\dFKn \rightarrow \mathbf{L}_{0}$ as an elementary consequence of Corollary~\ref{cor:marginals} and Proposition~\ref{prop:lawproduct} (Section~\ref{ssec:partialproof}). The proof follows the above given heuristic to explain that the phase transition occurs at $ K_{n}$ of order $\sqrt{n}$.

We then concentrate on the difficult cases  (i)$_{c>0}$ in Section~\ref{ssec:cv_pos} and (ii) in Section~\ref{sec:Linfty}. In these cases, we first show that  $  \dPKn  \rightarrow \mathbf{L}_{c}$, and then deduce by a short argument, based on that fact that $\dFKn$ and $\dPKn$ are close for the Hausdorff distance, that the convergence $(\dFKn,\dPKn) \rightarrow (\mathbf{L}_{c},\mathbf{L}_{c})$ holds jointly.  The avantage of working with $\PKn$ is that its law is well understood, unlike $\FKn$. Indeed, thanks to Proposition~\ref{prop:lawproduct},  $\PKn$ may be coded (Proposition~\ref{prop:bitype}) by a two-type  alternating Bienaym\'e--Galton--Watson tree conditioned on having $n-K_{n}$ vertices at even generation and $K_{n}+1$ vertices at odd generation, with offspring distribution depending on $n$ (to be really precise, the root has a slightly different offspring distribution from other vertices).  
We are therefore led to understand the structure of two-type BGW trees
conditioned on having a given number of vertices of both types with varying offspring distribution.
Unfortunately, the results on scaling limits of multi-type BGW trees in the literature (see e.g.\ \cite{MM07,Mie08b,Ber16})
consider either trees conditioned on having a total fixed size, or one type of fixed size.
Neither trees conditioned to having, for each type, a fixed number of individuals
nor situations where the offspring distribution varies with $n$,
seem to have been considered.
This creates some specific difficulties, that we have to overcome in Section~\ref{sec:limitheorems}.

We next show that the convergence $\dFKn \rightarrow \mathbf{L}_{\infty}$  in case (iii) simply follows from (ii) by a maximality argument (Section~\ref{sssec:iiiF}).

Then,  more interestingly, we establish, in case (i)$_{c=0}$, the  convergence $\dPKn \rightarrow \mathbf{L}_{0}$ by using, surprisingly, the convergence $\dFKn \rightarrow \mathbf{L}_{\infty}$ of case (iii) (Section~\ref{sssec:P0}).

We conclude in Section~\ref{sssec:iiiP} that   the last missing convergence $\dPKn \rightarrow \mathbf{L}_{\infty}$ in case (iii) holds by combining a short symmetry argument with again the convergence $ \dFKn \rightarrow \mathbf{L}_{\infty}$  established in Section~\ref{sssec:iiiF}.

\paragraph*{Perspective.} We believe that the increasing lamination-valued process $ c \mapsto \dot{\FFF}^{(n)}_{\lfloor c\sqrt{n} \rfloor}$ converges in distribution in the space of càdlàg lamination-valued processes to a limit which geometrically describes a planar version of the standard additive coalescent \cite{AP98} (the same should hold for $ c \mapsto \dot{\PPP}^{(n)}_{\lfloor c\sqrt{n} \rfloor}$,
obviously).
Note that, since $ c \mapsto \dot{\FFF}^{(n)}_{\lfloor c\sqrt{n} \rfloor}$ is increasing,
the limiting process should be an increasing coupling of the limiting laminations $(\mathbf L_c)_{c \in [0,+\infty]}$.
This is investigated in \cite{Thev18}.
In this direction, Theorem~\ref{thm:cvlam} (i) may be viewed as a one-dimensional convergence statement.

Also, the technique developed here to study alternating two-type BGW trees could be of independent interest.
Indeed, alternating two-type BGW trees (and more general multi-type BGW trees) have recently appeared in the context of random planar maps thanks to the Bouttier-Di Francesco-Guitter bijection \cite{BDFG04,MM07,LGM09} and in the context of looptrees \cite{CK15}.

\paragraph*{Notation.} As much as possible, we stick to the following convention: the style fonts $\mathsf{mathfrak}$, $\mathsf{mathcal}$, $\mathsf{mathscr}$ will be respectively used for sets, bijections and random variables.

\paragraph*{Acknowledgements.} I.K. thanks Jean Bertoin for an opportunity to stay at the University of Z\"urich, where this work was initiated. V.F. and I.K. are grateful to the thematic trimester ``Combinatorics and interactions'' at Institut Henri Poincaré, where part of this work has been completed, and to the referee for a careful reading.

V.F. is partially supported by the Swiss National Science Fundation, under the grant agreement nb 200020\_172515.

\setcounter{tocdepth}{2}
\tableofcontents

\section{Minimal factorizations and non-crossing partitions}
\label{sec:ncp}

In this section, we present a useful connection between minimal factorizations and non-crossing partitions.

\begin{table}[htbp]\caption{Table of the main notation and symbols appearing in Section~\ref{sec:ncp}.}
\centering
\begin{tabular}{c c p{0.8\linewidth} }
\toprule
$\mathfrak{T}_{n}$ & & The set of all transpositions of $\mathfrak{S}_{n}$. \\
$\mathfrak{M}_{n}$ & & The set of all minimal factorizations of  $(1,2, \ldots,n)$ into transpositions. \\
$\mathfrak{P}_{n}$ & & The set of all non-crossing partitions of $[n]$. \\
$|B|$ & & The size of a block $B$ of a non-crossing partition. \\
$ \mathcal{K}(P)$ &  & The Kreweras complement of a non-crossing partition $P$.\\
$\PPP(\sigma)$ & & The partition corresponding to the cycle decomposition of a permutation $\sigma$.\\
$ \mathcal{T}(P)$ & & The dual two-type plane tree  associated with a non-crossing partition $P$. \\
\bottomrule
\end{tabular}
\label{tab:secncp}
\end{table}

\subsection{Non-crossing partitions}
\label{ssec:ncp}

A permutation $\sigma \in \mathfrak{S}_{n}$ of size $n$ will be called a {\em geodesic permutation} if
there exists a minimal factorization $(\tau_1,\tau_2,\cdots,\tau_{n-1}) \in \mathfrak{M}_{n}$
and an integer $k \ge 1$ such that $\sigma=\tau_1 \cdots \tau_k$.
Geometrically, this means that $\sigma$ is on a geodesic path from the identity to the cycle
$(1,2,\ldots,n)$ in the Cayley graph of the symmetric group (using all transpositions as generators).
It was already mentioned in the Introduction that, if $\sigma$ is a geodesic permutation,
then the partition $\PPP(\sigma)$ of $[n]$ into cycles of $\sigma$ is a non-crossing partition.
In fact, $\PPP$ realizes a bijection between the set of all  geodesic permutations of size $n$
and the set of non-crossing partitions of size $n$ \cite[Theorem 1]{Biane1997}.
The pre-image of a non-crossing partition $P$ is the unique permutation $\sigma$
whose partitions in cycles is given by $P$ and such that each cycle is {\em increasing}, 
in the sense that it can be written 
$(c_1,\ldots,c_\ell)$ with $c_1<\cdots<c_\ell$.
As an example, both permutations $(1,3,4)(2)(5,6)$ and $(1,4,3)(2)(5,6)$ are associated
with the non-crossing set partition $\{\{1,3,4\},\{2\},\{5,6\}\}$,
but only the first one is geodesic, since all its cycles are increasing.

The Kreweras complement $ \mathcal{K}(P)$ of a non-crossing partition $P$ of $[n]$ will play an important role.
A first way of defining the Kreweras complement is the following (we refer to \cite[Sec.~9 and Sec.~18]{NS06} for details).
Observe that $\sigma \mapsto   (1,2, \ldots,n) \sigma^{-1} $
defines a bijective map from geodesic permutations to themselves.
Then the \emph{Kreweras complement} is the corresponding map on non-crossing partitions
using the bijection $\PPP$.
Namely, it is defined as follows: for any geodesic permutation $\sigma$ of size $n$,
\begin{equation}
\label{eq:K}\mathcal{K}\big(\PPP(\sigma)) \big)= \PPP \big( (1,2, \ldots,n) \sigma^{-1} \big).
\end{equation}
 The Kreweras complementation can alternatively be visualized as follows: consider the representation of $P \in \mathfrak{P}_n$ in the unit disk where blocks are colored in black; invert the colors and rotate the vertices of the regular $n$-gon by an angle $-\pi/n$; then the blocks of $ \mathcal{K}(P)$ are given by the vertices lying in the same ``colored'' component. See Figure~\ref{fig:complement_Kreweras} for an illustration,
 where $P=
\{1, 3, 5\}, \{2\}, \{4\}, \{6, 7, 11, 12\}, \{8\}, \{9, 10\}\}$ 
and $\mathcal{K}(P)=\{\{1, 2\}, \{3, 4\}, \{5, 12\}, \{6\}, \{7, 8, 10\}, \{11\}\}$.

%%%%%%%%%%%%%%%%%%%%%%%%%%%%%%
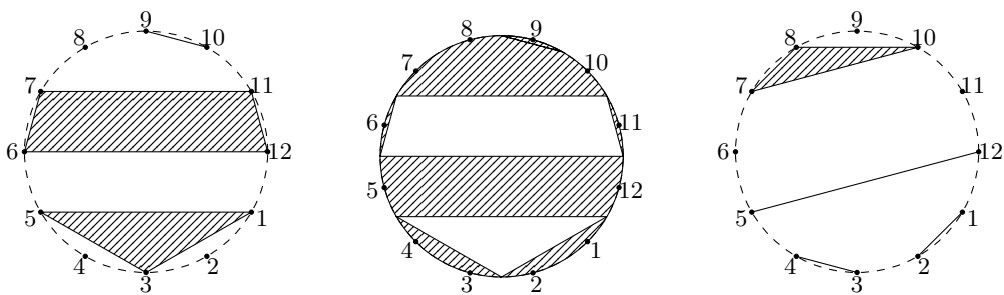
\begin{figure}[ht] \centering
%%% Partition
\begin{scriptsize}
\begin{tikzpicture}[scale=.8]
\draw[thin, dashed]	(0,0) circle (2);
\foreach \x in {1, 2, ..., 12}
	\coordinate (\x) at (-\x*360/12 : 2);
\foreach \x in {1, 2, ..., 12}
	\draw
	[fill=black]	(\x) circle (1pt)
	(-\x*360/12 : 2*1.1) node {\x}
;
\filldraw[pattern=north east lines]
	(1) -- (3) -- (5) -- cycle
	(6) -- (7) -- (11) -- (12) -- cycle
	(9) -- (10)
;
\end{tikzpicture}
\qquad
%
%%% Partition complement
\begin{tikzpicture}[scale=.8]
\draw[thin, dashed]	(0,0) circle (2);
\foreach \x in {1, 2, ..., 12}
	\coordinate (\x) at (-\x*360/12 : 2);
\foreach \x in {1, 2, ..., 12}
	\draw
[fill=black]	(-360/24-\x*360/12 : 2) circle (1pt)
	(-360/24-\x*360/12 : 2*1.1) node {\x}
;
\filldraw[pattern=north east lines]
	(0:2) arc (0:-360/12:2) -- (-5*360/12:2) arc (-5*360/12:-6*360/12:2) -- cycle
	(-360/12:2) arc (-360/12:-3*360/12:2) -- (1)
	(-3*360/12:2) arc (-3*360/12:-5*360/12:2) -- (3)
	(-6*360/12:2) arc (-6*360/12:-7*360/12:2) -- (6)
	(-9*360/12:2) arc (-9*360/12:-10*360/12:2) -- (9)
(-7*360/12:2) arc (-7*360/12:-9*360/12:2) -- (-10*360/12:2) arc (-10*360/12:-11*360/12:2) -- cycle
	(-11*360/12:2) arc (-11*360/12:-12*360/12:2) -- (11)
;
\draw 	(-9*360/12:2) arc (-9*360/12:-10*360/12:2);
\end{tikzpicture}
\qquad
%
%%% Partition complement bis
\begin{tikzpicture}[scale=.8]
\draw[thin, dashed]	(0,0) circle (2);
\foreach \x in {1, 2, ..., 12}
	\coordinate (\x) at (-\x*360/12 : 2);
\foreach \x in {1, 2, ..., 12}
	\draw
	[fill=black]	(-\x*360/12 : 2) circle (1pt)
	(-\x*360/12 : 2*1.1) node {\x}
;
\filldraw[pattern=north east lines]
	(1) -- (2)
(3) -- (4)
	(5) -- (12)
	(7) -- (8) -- (10) -- cycle
;
\end{tikzpicture}
\end{scriptsize}
\caption{Example of a non-crossing partition (left) and its Kreweras complement (middle and right).}
\label{fig:complement_Kreweras}
\end{figure}

 The notion of minimal factorization of a cycle is naturally
 extended to the notion of minimal factorization
 of a general permutation (we do not impose any transitivity condition here);
 the minimal number of transpositions needed to obtain $\sigma$
 is then $n-|\PPP(\sigma)|$, where $|\PPP(\sigma)|$ is the number of blocks of $\PPP(\sigma)$,
 that is the number of cycles in the disjoint cycle decomposition of $\sigma$.
 The following simple combinatorial lemma extends \eqref{eq:cardFn} by counting the number of minimal factorizations of a general permutation.
It is probably well-known but we could not locate it in the literature. We give its proof since it is short.

\begin{lemma}\label{lem:minfact}
Let $\sigma \in \mathfrak{S}_{n}$ be a permutation and let $P=\PPP(\sigma)$ be its associated partition. Let $k$ be the number of blocks of $P$. Then the number of minimal factorizations of $\sigma$ is
\[ \left( n- k \right) !  \cdot\prod_{B \in P}    \frac{|B|^{|B|-2}}{(|B|-1)!}.\]
\end{lemma}

\begin{proof}
Let $\sigma=C_1  \cdots C_k$ be the disjoint cycle decomposition of $\sigma$.
With each cycle $C_i$ is associated a block $B_i$ of $P=\PPP(\sigma)$.
By taking for each $i$ a minimal factorization of $C_i$ and then by shuffling their factors in any possible way, one gets a minimal factorization of $\sigma$. Conversely, it can be easily shown
that all minimal factorizations of $\sigma$ are obtained in this way.
 By \eqref{eq:cardFn}, the number of minimal factorizations of the cycle $C_i$  
 is $|C_i|^{|C_i|-2}=|B_i|^{|B_i|-2} $.
 The number of possible shuffles is 
$ \binom{|B_{1}|+\cdots+|B_{k}|-k}{|B_{1}|-1, \ldots,|B_{k}|-1 }$
(a minimal factorization of $|C_{i}|$ is a $|B_i|-1$ letter word).
The desired result follows by multiplying all these factors and by rearranging the terms.
\end{proof}

We use this lemma to establish Proposition~\ref{prop:lawproduct} (given in the Introduction),
that we restate here for the reader's convenience:
if $1 \leq k \leq n -1$ are positive integers and if
$\mathscr{F}^{(n)}= (\tb_{1}^{(n)}, \ldots, \tb_{n-1}^{(n)})$ is a random uniform element of $ \mathfrak{M}_{n}$
and $P$ a non-crossing partition of $[n]$ with $n-k$ blocks,
then
\[\Pr{\PPP(\tb_{1}^{(n)} \tb_{2}^{(n)} \cdots \tb_{k}^{(n)})=P}= \frac{k! (n-k-1)!}{n^{n-2}}  \cdot  \left( \prod_{B \in P} \frac{|B|^{|B|-2}}{(|B|-1)!} \right) \cdot \left(  \prod_{B \in  \mathcal{K}(P)} \frac{|B|^{|B|-2}}{(|B|-1)!} \right).\]

\begin{proof}[Proof of Proposition~\ref{prop:lawproduct}]
  We use the above notation and 
let in addition $\sigma \in \mathfrak{S}_{n}$ be the geodesic permutation such that $\PPP(\sigma)=P$.  Since $ | \mathfrak{M}_{n}|  =n^{n-2}$, it is enough to determine the cardinality of the set
\[\left\{ (\tau_{1}, \ldots, \tau_{n-1}) \in \mathfrak{T}_{n}^{n-1} : \tau_{1} \tau_{2} \cdots \tau_{n-1}= (1,2, \ldots,n) \textrm{ and } \tau_{1} \tau_{2} \cdots \tau_{k}= \sigma \right\}.\]
Elements of this set can be seen as pairs $\big( (\tau_1,\dots,\tau_k),(\tau_{k+1},\dots,\tau_{n-1}) \big)$
of minimal factorizations of $\sigma$
and of $(1,2, \ldots,n) \sigma^{-1}$.
The desired result then follows from Lemma~\ref{lem:minfact} 
since $ \PPP \big(  (1,2, \ldots,n) \sigma^{-1} \big) =\mathcal{K}(\PPP(\sigma))$.
\end{proof}

We are now in position to establish Corollary~\ref{cor:marginals}.

\begin{proof}[Proof of Corollary~\ref{cor:marginals}]
We start with (i). Consider $n-2$ transpositions $t_{1}, \ldots,t_{n-2} \in \mathfrak{T}_{n}$. To simplify notation, set $C=(1,2,\ldots,n)$. Then
\[\hspace{-3mm}\Pr{(\tb_{1}^{(n)}, \ldots, \tb_{n-2}^{(n)})=(t_{1}, \ldots,t_{n-2})}= \frac{\big| \{t \in \mathfrak{T}_{n} : t_{1}t_{2} \cdots t_{n-2} t= C\}\big| }{n^{n-2}}= \frac{\mathbbm{1}_{C^{-1} t_{1}t_{2} \cdots t_{n-2} \in \mathfrak{T}_{n}} }{n^{n-2}}.\]
Similarly,
\[\hspace{-3mm}\Pr{(\tb_{2}^{(n)}, \ldots, \tb_{n-1}^{(n)})=(t_{1}, \ldots,t_{n-2})}= \frac{\big| \{t \in \mathfrak{T}_{n} : t \, t_{1}t_{2} \cdots t_{n-2}= C\}\big| }{n^{n-2}}= \frac{\mathbbm{1}_{ t_{1}t_{2} \cdots t_{n-2}  C^{-1}\in \mathfrak{T}_{n}} }{n^{n-2}}.\]
Since $t_{1}t_{2} \cdots t_{n-2}  C^{-1}$ and $ C^{-1} t_{1}t_{2} \cdots t_{n-2}$ belong to the same conjugacy class, one is a transposition if and only if the other is. It follows that $(\tb_{1}^{(n)}, \ldots, \tb_{n-2}^{(n)})$ and $(\tb_{2}^{(n)}, \ldots, \tb^{(n)}_{n-1})$ have the same distribution.

Let us now determine the law of $ \tb_{1}^{(n)}$.   For $1 \leq a  \leq n-1$ and $ 1 \leq i \leq n-a$, if $\tau=(a,a+i)$, then all the blocks of $ \PPP(\tau)$ have size one, except one which has size two, and $ \mathcal{K}( \PPP(\tau))$ has two blocks of sizes $i$ and $n-i$. Therefore, by Proposition~\ref{prop:lawproduct},
\[\Pr{\tb_{1}^{(n)}=(a,a+i)}= \frac{1!(n-2)!}{n^{n-2}} \cdot 1 \cdot \frac{i^{i-2}}{(i-1)!} \cdot \frac{(n-i)^{(n-i-2)}}{(n-i-1)!}.\]
The identity \eqref{eq:lawT1} immediately follows. The other assertions  of (ii) are then easy consequences, and are left to the reader.
\end{proof}

\subsection{Partial proof of Theorem~\ref{thm:cvlam} (i)$_{c=0}$.}
\label{ssec:partialproof}

In this section, we prove the convergence of the first coordinate in Theorem~\ref{thm:cvlam} (i) when $c=0$,
namely that $\FKn$ tends to $\mathbf L_0$ when $K_n/\sqrt{n}$ tends to $0$.
We start with a simple observation, whose proof is straightforward.

\begin{lemma}
\label{lem:observation}
Let $(\tau_{1}, \ldots, \tau_{n-1}) \in \mathfrak{M}_{n}$ be a minimal factorization.
Fix a positive integer $k \leq n-1$, set  $\PPP_{k}= \PPP(\tau_{1} \tau_{2}\cdots \tau_{k})$ and $ \FFF_{k}=\FFF(\tau_{1} ,\tau_{2} ,\ldots ,\tau_{k})$. Then the blocks of $\PPP_{k}$ are the connected components of $\FFF_{k}$. In particular, $\S \cap \dot{\PPP}_{k}=\S \cap \dot{\FFF}_{k}$.
\end{lemma}

The reader can back look at Figure~\ref{fig:observation} for an example.
On this figure, the blocks of $\PPP_{6}$ are indeed the connected components of $\FFF_{6}$, 
and  \[\S \cap \dot{\PPP}_{6}=\S \cap \dot{\FFF}_{6} =  \{ e^{- 2 \textrm{i} \pi {k}/{12}}; k \in \{1,3,5,6,7,9,10,11,12\} \}.\]

Corollary~\ref{cor:marginals} and Proposition~\ref{prop:lawproduct} allow us to establish half of Theorem~\ref{thm:cvlam} (i)$_{c=0}$.

\begin{proof}[Partial proof of Theorem~\ref{thm:cvlam} (i)$_{c=0}$] We assume  that $ \frac{K_{n}}{\sqrt{n}} \rightarrow 0$ and $K_{n} \rightarrow \infty$ as $n \rightarrow \infty$, and we shall show that $\dFKn \rightarrow \mathbf{L}_{0}=\mathbb{S}$. In order to show that $\dFKn $ converges in distribution to $ \S$, it is enough to show that for every fixed $\epsilon>0$: (a) the probability that there exists a chord of Euclidean length at least $\epsilon$ tends to $0$ as $n \rightarrow \infty$ and (b) the convergence $\S \cap \dFKn \rightarrow \S$ holds in distribution as $n \rightarrow \infty$.

For (a), write $ \Phi( (u,v))=\min(v-u,n-v+u)$ for a transposition $(u,v) \in \mathfrak T_n$ with $u<v$.
It is enough to show that $\Prb{\max_{1 \leq i \leq K_{n}} \Phi(\tb_{i}^{(n)}) > \epsilon n} \rightarrow 0$ as $n \rightarrow \infty$. By Corollary~\ref{cor:marginals} (i), for every $1 \leq i<j \leq n-1$, $\tb_{i}^{(n)}$ and $\tb_{j}^{(n)}$ have the same distribution, so that
$$\Prb{\max_{1 \leq i \leq K_{n}} \Phi(\tb_{i}^{(n)}) > \epsilon n} \leq K_{n} \Prb{\Phi(\tb_{1}^{(n)})>\epsilon n}.$$
Therefore, combining Corollary~\ref{cor:marginals} (ii) with the fact that  $ \frac{n^{n-2}}{(n-1)!} \sim  \tfrac{e^{n}}{n^{3/2}}$ as $ n \rightarrow \infty$, there exists a constant $C>0$ such that
\begin{align*}
K_{n} \Pr{\Phi(\tb_{1}^{(n)})>\epsilon n} & \leq  n K_{n} \sum_{i=  \epsilon n}^{ (1-\epsilon)n}  \frac{(n-2)!}{n^{n-2}} \cdot \frac{i^{i-2}}{(i-1)!} \cdot \frac{(n-i)^{(n-i-2)}}{(n-i-1)!}\\
& \leq  C  K_{n}  \sum_{i=  \epsilon n}^{ (1-\epsilon)n}   \frac{n^{3/2}}{e^{n}} \cdot  \frac{e^{i}}{i^{3/2}} \cdot  \frac{e^{n-i}}{(n-i)^{3/2}} \\
& =  C\, \frac{K_{n}}{\sqrt{n}} \cdot \frac{1}{n}    \sum_{i=  \epsilon n}^{ (1-\epsilon)n}  \frac{1}{ \left( \frac{i}{n} \cdot \left( 1- \frac{i}{n} \right)   \right) ^{3/2}},
\end{align*}
which tends to $0$ as $ n \rightarrow \infty$ by recognizing a Riemann sum and by using the fact that $\tfrac{K_{n}}{\sqrt{n}} \rightarrow 0$.

We now turn to (b):
fix an arc of $\S$ of length $\eps$, we will show that the probability
that there is a chord of $\dFKn$ with one endpoint in this arc tends to $1$.
If we prove this, then any limit point of $\S \cap \dFKn$ should intersect any arc of $\S$
and thus has to be $\S$ itself (we work with the Hausdorff topology on compact subsets of $\overline{\D}$, so limit points are necessarily closed sets).
By compactness this entails the convergence $\S \cap \dFKn \to \S$.

By rotational invariance, it is enough to show that the probability that there is no chord of $\dFKn$ adjacent to an element of the set $ \{e^{-2 \i \pi {j}/{n}} : 1 \leq j \leq \lfloor \epsilon n \rfloor \}$ tends to $0$ as $n \rightarrow \infty$. By construction of $\PPP^{(n)}_{K_{n}}$, there is no chord of $\dFKn$ adjacent to an element of the set $ \{e^{-2 \i \pi {j}/{n}} : 1 \leq j \leq \lfloor \epsilon n \rfloor \}$ if and only if all the elements of  the set $ \{e^{-2 \i \pi {j}/{n}} : 1 \leq j \leq \lfloor \epsilon n \rfloor \}$ are singleton blocks in $\PPP^{(n)}_{K_{n}}$. In turn, it is therefore enough to check that
\[ \P \big( \forall 1 \leq j \leq \lfloor \epsilon n \rfloor, \{j\} \textrm{ is a block of } \PPP^{(n)}_{K_{n}}  \big)  \quad \mathop{\longrightarrow}_{n \rightarrow \infty} \quad 0.\]

To this end, observe that a non-crossing partition $P \in \mathfrak{P}_{n}$ with $n-K_{n}$ blocks such that $\{j\}$ is a block of $P$ for every $1 \leq j \leq \lfloor \epsilon n \rfloor$ can be seen as a non-crossing partition  $P' \in \mathfrak{P}_{n- \lfloor \epsilon n \rfloor}$  with $n- \lfloor \epsilon n \rfloor-K_{n}$ blocks simply by erasing the blocks $\{j\}$ for every $1 \leq j \leq \lfloor \epsilon n \rfloor$ and by subtracting $\lfloor \epsilon n \rfloor$ to everyone. Then the sizes of the blocks of size at least $2$ of $P'$ are those of $P$, and the the sizes of the blocks of size at least $2$ of $\mathcal{K}(P')$ are those of $\mathcal{K}(P)$, except that the size of the block of  $\mathcal{K}(P)$ containing $n$ is the size of the block of  $\mathcal{K}(P')$ containing $n- \lfloor \epsilon n \rfloor$ plus $\lfloor \epsilon n \rfloor$ (see Figure~\ref{fig:chirurgie}).
 \begin{figure}[t]
 \begin{center}
 \includegraphics[scale=0.15]{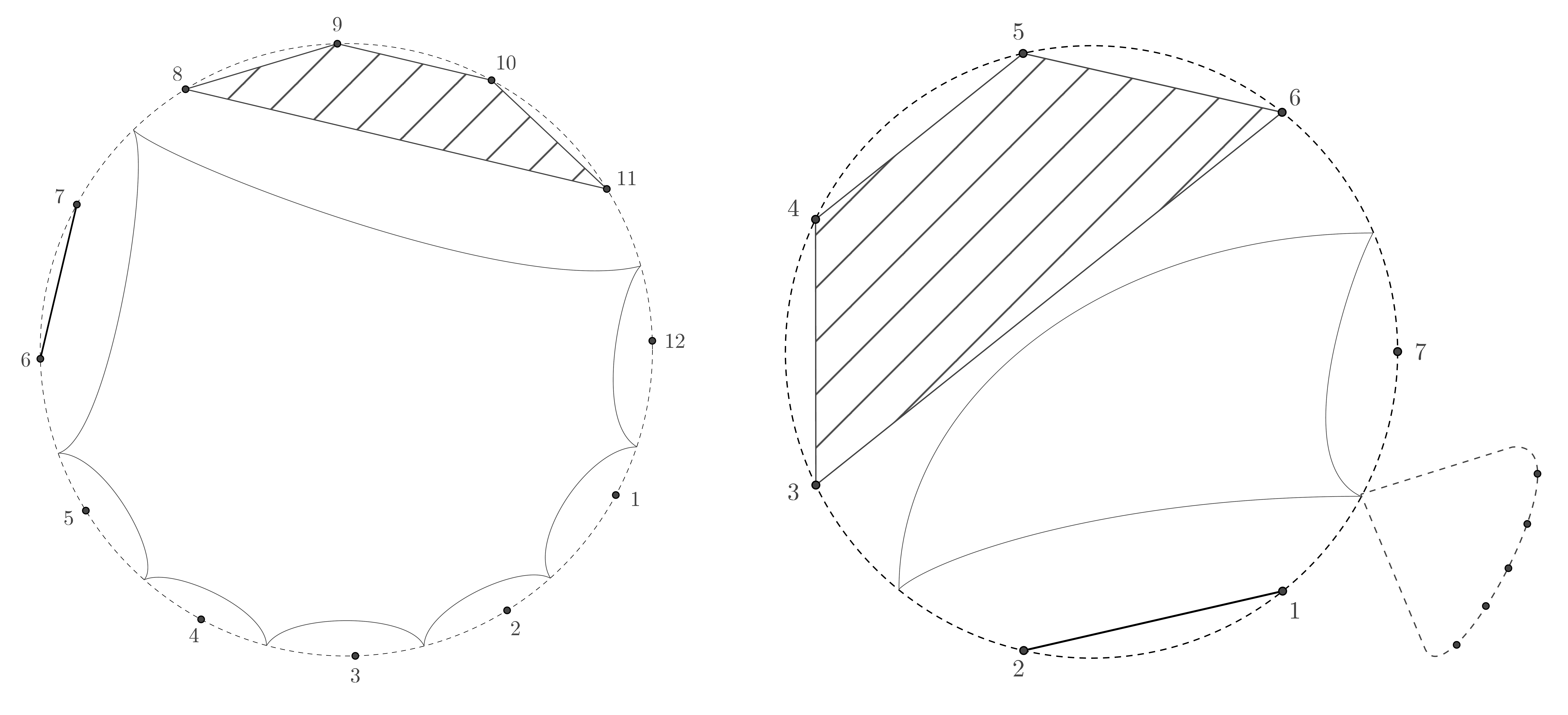}
 \caption{\label{fig:chirurgie}Illustration of the erasing procedure: with a non-crossing partition $P \in \mathfrak{P}_{12}$ such that  $\{j\}$ is a block of $P$ for every $1 \leq j \leq 5$ one associated a non-crossing partition $P' \in \mathfrak{P}_{12-5}$. The curved lines represent the block containing $12$ in $ \mathcal{K}(P)$ (which has size $8$) and $7$ in $ \mathcal{K}(P')$ (which has indeed size $8-5=3$).}
 \end{center}
 \end{figure}

By  Proposition~\ref{prop:lawproduct}, the quantity $\P \big( \forall 1 \leq j \leq \lfloor \epsilon n \rfloor, \{j\} \textrm{ is a block of } \PPP^{(n)}_{K_{n}}  \big)$ is therefore equal to
\begin{multline*}
  \frac{K_{n}! (n-K_{n}-1)!}{n^{n-2}}  \sum_{\substack{P' \in \mathfrak{P}_{n- \lfloor \epsilon n \rfloor} \\  \textrm{with } n- \lfloor \epsilon n \rfloor-K_{n} \textrm{ blocks}}}  \left( \prod_{B \in P'} \frac{|B|^{|B|-2}}{(|B|-1)!} \right) \cdot \left(  \prod_{B \in  \mathcal{K}(P')} \frac{|B|^{|B|-2}}{(|B|-1)!} \right) \\
  \cdot  \frac{(|\hat{B}|-1)!}{|\hat{B}|^{|\hat{B}|+2}} \cdot\frac{(|\hat{B}|+\lfloor \epsilon n \rfloor)^{|\hat{B}|+\lfloor \epsilon n \rfloor-2}}{(|\hat{B}|+\lfloor \epsilon n \rfloor-1)!}.
\end{multline*}
where $\hat{B}$ denotes the block of $\mathcal{K}(P')$ containing $n- \lfloor \epsilon n \rfloor$. Now, since $ \frac{n^{n-2}}{(n-1)!} \sim  \tfrac{e^{n}}{n^{3/2}}$ as $ n \rightarrow \infty$, there exists a constant $C>0$ (whose value will change from line to line) such that $ \tfrac{1}{C}  \tfrac{e^{n}}{n^{3/2}} \leq  \frac{n^{n-2}}{(n-1)!}  \leq C  \tfrac{e^{n}}{n^{3/2}}$ for every $n \geq 1$.  Therefore
\[
 \frac{(|\hat{B}|-1)!}{|\hat{B}|^{|\hat{B}|+2}} \cdot\frac{(|\hat{B}|+\lfloor \epsilon n \rfloor)^{|\hat{B}|+\lfloor \epsilon n \rfloor-2}}{(|\hat{B}|+\lfloor \epsilon n \rfloor-1)!} \leq   C \frac{|\hat{B}|^{3/2}}{e^{|\hat{B}|}} \cdot  \frac{e^{|\hat{B}|+\lfloor \epsilon n \rfloor}}{(|\hat{B}|+\lfloor \epsilon n \rfloor)^{3/2}} \leq C e^{\lfloor \epsilon n \rfloor}.
\]
Hence, using the fact that the probabilities in Proposition~\ref{prop:lawproduct} sum up to one, we get 
\begin{multline*}
  P \big( \forall 1 \leq j \leq \lfloor \epsilon n \rfloor, \{j\} \textrm{ is a block of } \PPP^{(n)}_{K_{n}}  \big) \\
  \leq  C  \, \frac{K_{n}! (n-K_{n}-1)!}{n^{n-2}} \cdot \frac{(n-\lfloor \epsilon n \rfloor)^{(n-\lfloor \epsilon n \rfloor-2)}}{K_{n}! (n-\lfloor \epsilon n \rfloor  -K_{n}-1)!} \cdot e^{\lfloor \epsilon n \rfloor}.
\end{multline*}

Again, it is a simple matter to check that there exists a constant $C>0$ such that we have $ \tfrac{1}{C}  \tfrac{ e^{n}}{n^{3/2-k}} \leq  \tfrac{n^{n-2}}{(n-k-1)!} \leq C \tfrac{ e^{n}}{n^{3/2-k}}$ for every $n \geq 1$ and $1 \leq k \leq  \sqrt{n}$. Thus, since $K_{n}/\sqrt{n} \rightarrow 0$, for $n$ sufficiently large we have
\begin{align*}
\P \big( \forall 1 \leq j \leq \lfloor \epsilon n \rfloor, \{j\} \textrm{ is a block of } \PPP^{(n)}_{K_{n}}  \big) &\leq C \frac{n^{3/2-K_{n}}}{e^{n}} \cdot  \frac{e^{n-\lfloor \epsilon n \rfloor}}{(n-\lfloor \epsilon n \rfloor)^{(3/2-K_{n})}} \cdot e^{\lfloor \epsilon n \rfloor}\\
&= C \left( 1- \frac{\lfloor \epsilon n \rfloor}{n} \right)^{K_{n}-3/2}
\end{align*}
which tends to $0$ since $K_{n} \rightarrow \infty$ as $n \rightarrow \infty$.
\end{proof}

We conclude this Section with a symmetry lemma, which will be useful later.

\begin{lemma}
\label{lem:symmetry}
Fix $1 \leq k \leq n-1$. The two random non-crossing partitions
$\PPP\big(\tb_{1}^{(n)} \tb_{2}^{(n)} \cdots \tb_{n-k-1}^{(n)}\big)$
and $\mathcal{K}\big( \PPP\big(\tb_{1}^{(n)} \tb_{2}^{(n)} \cdots \tb_{k}^{(n)}\big)\big)$ have the same distribution.
\end{lemma}

\begin{proof}
  First note that $P$ has $n-k$ blocks if and only if $\mathcal{K}(P)$ has $k+1$ blocks,
  so that both random non-crossing partitions have the same support (non-crossing partitions of $n$
  with $k+1$ blocks).
For every geodesic permutation $\sigma \in \mathfrak{S}_{n}$, we have the following equality of multisets:
\begin{multline*}
\{|B| ; B \in \mathcal{P}(\sigma) \} \cup  \{|B| ; B \in \mathcal{K}( \mathcal{P}(\sigma)) \}\\
  =\{|B| ; B \in \mathcal{P}((1,2, \ldots,n) \sigma^{-1} ) \} \cup  \{|B| ; B \in \mathcal{K}( \mathcal{P}((1,2, \ldots,n) \sigma^{-1} )) \}.
\end{multline*}
Indeed, by definition we have $\mathcal{P}( (1,2, \ldots,n) \sigma^{-1} )=\mathcal{K}( \mathcal{P}(\sigma))$,
while the non-crossing partition $\mathcal{K}( \mathcal{P}( (1,2, \ldots,n) \sigma^{-1} ))=\mathcal K(\mathcal K( \mathcal{P}(\sigma)))$ is obtained from $ \mathcal{P}(\sigma)$
by a rotation of angle $\tfrac{2\pi}{n}$, which keeps block sizes invariant.
The lemma then follows from Proposition~\ref{prop:lawproduct}.
\end{proof}

\subsection{Minimal factorizations and trees}
\label{sec:trees}

In order to study properties of large random minimal factorizations, it will be useful to code their associated non-crossing partitions with trees.

\subsubsection{Plane trees.} We use Neveu's formalism \cite{Nev86} to define plane trees: let $\N = \{1, 2, \dots\}$ be the set of all positive integers, set $\N^0 = \{\varnothing\}$ and consider the set of labels $\U = \bigcup_{n \ge 0} \N^n$. For $u = (u_1, \dots, u_n) \in \U$, we denote by $|u| = n$ the length of $u$; if $n \ge 1$, we define $pr(u) = (u_1, \dots, u_{n-1})$ and for $i \ge 1$, we let $ui = (u_1, \dots, u_n, i)$; more generally, for $v = (v_1, \dots, v_m) \in \U$, we let $uv = (u_1, \dots, u_n, v_1, \dots, v_m) \in \U$ be the concatenation of $u$ and $v$. We endow $\U$ with the lexicographical order: given $v,w \in \U$, if $z \in \U$ is their longest common prefix, so that $v = z(v_1, \dots, v_n)$, $w = z(w_1, \dots, w_m)$ with $v_1 \ne w_1$), then $v \prec w$ if $v_1 < w_1$.

A (locally finite) \emph{plane tree} is a nonempty  subset $\tau \subset \U$ such that (i) $\varnothing \in \tau$; (ii)~if $u \in \tau$ with $|u| \ge 1$, then $pr(u) \in \tau$; (iii)  if $u \in \tau$, then there exists an integer $k_u(\tau) \ge 0$ such that $ui \in \tau$ if and only if $1 \le i \le k_u(\tau)$.

We may view each vertex $u$ of a tree $\tau$ as an individual of a population for which $\tau$ is the genealogical tree. For $u,v \in \tau$, we let  $\llbracket u, v \rrbracket$ be the vertices belonging to the shortest path from $u$ to $v$.
Accordingly, we use $\llbracket u, v \llbracket$ for the same set, excluding $v$.
The vertex $\varnothing$ is called the \emph{root} of the tree and for every $u \in \tau$, $k_u(\tau)$ is the number of children of $u$, $|u|$ is its \emph{generation}, $pr(u)$ is its \emph{parent} and more generally, the vertices $u, pr(u), pr \circ pr (u), \dots, pr^{|u|}(u) = \varnothing$ belonging to $ \llbracket \varnothing, u \rrbracket$ are its \emph{ancestors}. 
The descendants of $u$ are all the vertices $v \neq u$ such that $u$ is an ancestor of $v$.
A vertex with no children (i.e. with $k_u(\tau) = 0$)
is called a \emph{leaf}, other vertices are called \emph{internal}.
To simplify, we will sometimes write $k_{u}$ instead of $k_{u}(\tau)$.

We denote by $\mathbb{A}$ the countable set of all finite plane trees and by $\mathbb{A}^{\infty}$   the  set of all (finite or infinite) plane trees
(equipped with the smallest $\sigma$-algebra such that 
the projections consisting in forgetting all vertices after generation $k$
are measurable for every $k \geq 1$).
For every  $n \geq 1$, we let $\mathbb{A}_n$ 
be the set of plane trees with $n$ vertices.

We will also consider bicolored trees:
except specific mention, our bicolored trees have a black root and alternating colors, 
i.e.\ the children of a black vertex are white and vice-versa.
Given a plane tree, such a coloring is of course unique;
if $\tau \in \mathbb{A}$ is a tree, we denote by $\bullet_{\tau}$ the set of all vertices of $\tau$ at even generation (called \emph{black vertices}) and by $\circ_{\tau}$ the set of all vertices of $\tau$ at odd generation (called \emph{white vertices}).

\subsubsection{Coding non-crossing partitions with trees.} 
\label{sec:coding}
Given a non-crossing partition $P$, we associate with $P$ a plane rooted tree $\TTT(P)$ as follows.
The tree $\TTT(P)$ has one black vertex for each block of $P$, and one white vertex for each block of $\mathcal K(P)$.
Recall that blocks of $\mathcal K(P)$ correspond to the white regions between the blocks of $P$
in the graphical representation of $P$.
We then connect vertices corresponding to neighbour regions.
This gives a plane tree which, using the terminology of planar graphs,
could be called the {\em dual tree} of $P$.
This tree is however not rooted and by convention we  root it at the black vertex
corresponding to the block of $P$ containing $n$, at the corner corresponding to $n$
(note that in order to root a planar tree  one needs not only to choose a root vertex but also a root corner).
This construction is best understood on an example, see Figure~\ref{fig:nc}.

%%%%%%%%%%%%%%%%%%%%%
%%% PARTITION and dual tree %%%
%%%%%%%%%%%%%%%%%%%%%%
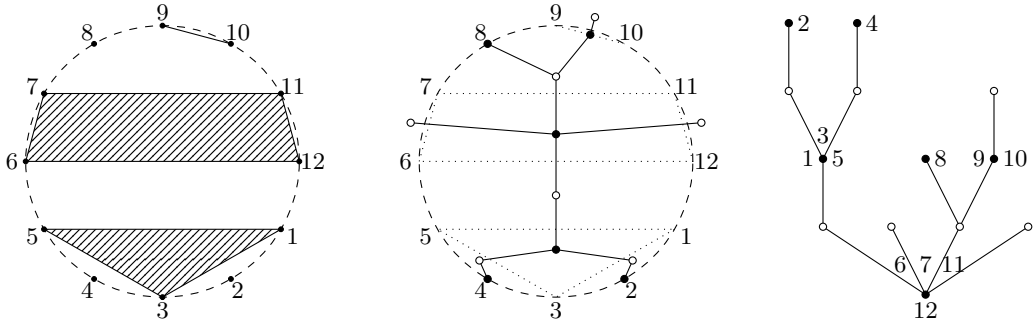
\begin{figure}[t] \centering
%%% Partition
\begin{scriptsize}
\begin{tikzpicture}[scale=.9]
\draw[thin, dashed]	(0,0) circle (2);
\foreach \x in {1, 2, ..., 12}
	\coordinate (\x) at (-\x*360/12 : 2);
\foreach \x in {1, 2, ..., 12}
	\draw
	[fill=black]	(\x) circle (1pt)
	(-\x*360/12 : 2*1.1) node {\x}
;
\filldraw[pattern=north east lines]
	(1) -- (3) -- (5) -- cycle
	(6) -- (7) -- (11) -- (12) -- cycle
	(9) -- (10)
;
\end{tikzpicture}
\qquad
	%
%%% Partition and dual tree
\begin{tikzpicture}[scale=.9]
\draw[thin, dashed]	(0,0) circle (2);
\foreach \x in {1, 2, ..., 12}
	\coordinate (\x) at (-\x*360/12 : 2);
\foreach \x in {1, 2, ..., 12}
	\draw	(-\x*360/12 : 2*1.1) node {\x}
;
\draw[dotted]
	(1) -- (3) -- (5) -- cycle
	(6) -- (7) -- (11) -- (12) -- cycle
	(9) -- (10)
;
\coordinate (A) at (0, .4);
\coordinate (B) at (0, -1.3);
\coordinate (C) at (2);
\coordinate (D) at (4);
\coordinate (E) at (8);
\coordinate (F) at ($(9)!0.5!(10)$);
\coordinate (A') at (0, -.5);
\coordinate (B') at (360/48-2*360/12 : 2*.92);
\coordinate (C') at (3*360/48-5*360/12 : 2*.92);
\coordinate (D') at (360/24-7*360/12 : 2*1.1);
\coordinate (E') at (0, 1.25);
\coordinate (F') at (360/24-10*360/12 : 2*1.1);
\coordinate (G') at (360/24-12*360/12 : 2*1.1);
\draw
		(A) -- (A') -- (B)
		(B) -- (B') -- (C)
		(B) -- (C') -- (D)
		(A) -- (D')
		(A) -- (E')
		(E') -- (E)
		(E') -- (F) -- (F')
		(A) -- (G')
;
\foreach \x in {A, B, ..., F}
		\draw[fill=black]	(\x) circle (1.5pt)
;
\foreach \x in {A, B, ..., G}
		\draw[fill=white]	(\x') circle (1.5pt)
;
\end{tikzpicture}
  \qquad 
  \begin{tikzpicture}[scale=.9]
\coordinate (0) at (0,0);
	\coordinate (1) at (-1.5,1);
		\coordinate (11) at (-1.5,2);
			\coordinate (111) at (-2,3);
						\coordinate (1111) at (-2,4);
			\coordinate (112) at (-1,3);
					\coordinate (1121) at (-1,4);
	\coordinate (2) at (-.5,1);
	\coordinate (3) at (.5,1);
		\coordinate (31) at (0,2);
		\coordinate (32) at (1,2);
			\coordinate (321) at (1,3);
	\coordinate (4) at (1.5,1);
\draw
	(0) -- (1)	(0) -- (2)	(0) -- (3)	(0) -- (4)
	(1) -- (11)
	(11) -- (111) -- (1111)	(11) -- (112) -- (1121)
	(3) -- (31)	(3) -- (32) -- (321)
;
\draw[fill=black]
	(0) circle (1.5pt)
	(11) circle (1.5pt)
	(31) circle (1.5pt)
	(32) circle (1.5pt)
	(1111) circle (1.5pt)
	(1121) circle (1.5pt)
;
\draw[fill=white]
	(1) circle (1.5pt)
	(2) circle (1.5pt)
	(3) circle (1.5pt)
	(4) circle (1.5pt)
	(111) circle (1.5pt)
	(112) circle (1.5pt)
	(321) circle (1.5pt)
;

	% Labels
\draw
	(11) node[left] {  $1$}
	(1111) node[right] {$2$}
	(11)++(0,.1) node[above] {$3$}
	(1121) node[right] {$4$}
	(11) node[right] {$5$}
	(0)++(-.15,.2) node[above left] {$6$}
	(0)++(0,.2) node[above] {$7$}
	(31) node[right] {$8$}
	(32) node[left] {$9$}
	(32) node[right] {$10$}
	(0)++(.1,.2) node[above right] {$11$}
	(0) node[below] {$12$}
;
\end{tikzpicture}
\end{scriptsize}
\caption{The partition $\{\{1, 3, 5\}, \{2\}, \{4\}, \{6, 7, 11, 12\}, \{8\}, \{9, 10\}\}$, its dual plane two-type tree and its associated plane rooted tree (with the black corners listed in the contour order). }
\label{fig:nc}
\end{figure}

It is easy to see that the resulting graph is a bicolored tree with black root and alternating colors.
Moreover, the degree of a vertex corresponds to the number of elements in the corresponding block,
either of $P$ or $\mathcal K(P)$. (The degree of the root is its number of children, while degree
of other vertices are their number of children plus one; this subtlety will be of importance later.)

The map $\TTT$ defines a bijection between non-crossing partitions of size $n$
and rooted plane trees with $n+1$ vertices (taken with their canonical bicoloration as above), see, {\em e.g.}, \cite{KM17}. 
In view of future use, let us  explain how to recover the non-crossing partition from the associated plane rooted tree $T$. Define the contour sequence $(u_0, u_1, \dots ,u_{2n})$ of a tree $T$  with $n+1$ vertices as follows: $u_0 = \varnothing$ and for each $i \in \{0, \dots, 2n-1\}$, $u_{i+1}$ is either the first child of $u_i$ which does not appear in the sequence $(u_0, \dots, u_i)$, or the parent of $u_i$ if all its children already appear in this sequence. Now, a corner of a vertex $v \in T$ is a sector around $v$ delimited by two consecutive edges in the contour sequence (with the convention that the first and last edge in the contour sequence are adjacent). We index from $1$ to $n$ the corners of (only) the black vertices of $T$, following the contour sequence and ending at the root corner (see Figure~\ref{fig:nc}). Then the non-crossing partition $P$ whose blocks are the labels adjacent to corners of black vertices satisfies $\mathcal{T}(P)=T$.

Observe that if $P$ is a non-crossing partition of size $n$ with $n-k$ blocks (with $1 \leq k \leq n-1$), then $ \mathcal{T}(P)$ is a tree with $n+1$ vertices,  $n-k$ black vertices and $k+1$ white vertices.

We now state a simple lemma, which allows to approximately reconstruct  the non-crossing partition from its associated tree. Denote by $\abs{\bullet_{T}}$  (resp.~$\abs{\circ_{T}}$)  the number of black (resp.~white) vertices of a plane rooted tree $T$.

\begin{lemma}
  \label{lem:OnBlockLabels}
Let $T$ be a rooted plane tree, taken with its canonical bicoloration.  Let $ (v^{\bullet}_{i})_{0 \leq i \leq \abs{\bullet_{T}}-1 }$ be the list of the black vertices of $T$ in lexicographical order.
  Fix $ 0 < i \leq \abs{\bullet_{T}}-1$. Let $x^{\bullet}_{i}$ (resp.~$y^{\bullet}_{i}$) be the index of the first (resp.~last) black corner of $v^{\bullet}_i$ visited by the contour sequence. Denote by $n_{i}$ the total number of  descendants of $v^{\bullet}_{i}$.
 \begin{enumerate}
 \item[(i)] We have   $y^{\bullet}_{i}=x^{\bullet}_{i}+n_{i}$.
  \item[(ii)] Let $\ell_{i}^{\bullet}$ denote the number of black corners  branching on the right of  $\llbracket \emptyset, v^{\bullet}_{i} \llbracket$ (that is corners  adjacent to black vertices of $\llbracket \emptyset, v^{\bullet}_{i} \llbracket$  which are visited after $v^{\bullet}_{i}$ in the contour sequence). Then
   \[ x^{\bullet}_{i}= i+\sum_{j=0}^{i-1} k_{j} - \ell_{i}^{\bullet},\]
   where $k_{j}$ is the number of children of $v^{\bullet}_{j}$.
  \item[(iii)] We have $ i \leq x^{\bullet}_{i} \leq i+ \abs{\circ_{T}}$ and $ i +n^{\bullet}_{i}\leq y^{\bullet}_{i} \leq i+n^{\bullet}_{i}+ 2 \abs{\circ_{T}}$, where  $n^{\bullet}_{i}$ is the total number of black descendants of $v^{\bullet}_{i}$.
 \end{enumerate}
  \end{lemma}

\begin{proof} The first assertion is a simple combinatorial fact that can be checked by induction on the size of a tree.
  For (ii),  observe that $\sum_{j=0}^{i-1}(k_{j}+1)-1$ is the total number of black corners adjacent to $v^{\bullet}_{0}, v^{\bullet}_{1}, \ldots, v^{\bullet}_{i-1}$ (the root $v^{\bullet}_0$ has $k_0$ corners, while the vertex $v^{\bullet}_j$ for $j>1$ has $k_j+1$ corners).
  When visiting $v^{\bullet}_{i}$ for the first time, among these black corners, the contour sequence has not yet visited those branching on the right of  $\llbracket \emptyset, v^{\bullet}_{i} \llbracket$. For (iii), if $u \in  \llbracket \emptyset, v^{\bullet}_{i} \llbracket$ is a black vertex,  the number of black corners  branching on the right of $\llbracket \emptyset, v^{\bullet}_{i} \llbracket$ and adjacent to $u$ is less than or equal to $k_{u}$. Therefore $\sum_{j=0}^{i-1} k_{j} \geq  \ell_{i}^{\bullet}$, which shows that $i \leq x^{\bullet}_{i}$. The inequality $x^{\bullet}_{i} \leq i+ \abs{\circ_{T}}$ comes from the fact that that $\sum_{j=0}^{i-1} k_{j}$  is the total number of (white) children of the  (black) vertices $v^{\bullet}_0$,\dots,$v^{\bullet}_{i-1}$, which is therefore less than or equal to $ \abs{\circ_{T}}$. The inequalities concerning $y^{\bullet}_{i}$ readily follow from (i) and the fact that $ n^{\bullet}_{i} \leq n_{i} \leq n_{i}^{\bullet}+\abs{\circ_{T}}$. This completes the proof.
    \end{proof}

\subsubsection{Random BGW trees.} Let $\mu$ be a  probability measure on $\Z_+$ (called the \emph{offspring distribution}) such that $\mu(0) > 0$, $\mu(0)+\mu(1)<1$ (to avoid trivial cases). The {\BGW} measure with offspring distribution $\mu$ is a probability measure $\mathrm{BGW}^\mu$ on $\mathbb{A}^{\infty}$ such that  for every $\tau \in \mathbb{A}$,
\begin{equation}\label{eq:def_GW}
\mathrm{BGW}^\mu(\tau) = \prod_{u \in \tau} \mu(k_u),
\end{equation}
see e.g.~\cite[Prop.~1.4]{LG05}.

It is well-known that $\mathrm{BGW}^\mu(\mathbb{A})=1$ if and only if $\sum_{i \geq 0} i \mu (i) \leq 1$.
Now, given a subset $\mathbb{B}$ of $\mathbb A$ of positive probability with respect to  $\mathrm{BGW}^\mu$,
 the conditional probability measure  $\mathrm{BGW}^\mu$  given $\mathbb{B}$ is the probability measure $\mathrm{BGW}^\mu_{\mathbb B}$ 
defined by
\[ \mathrm{BGW}^\mu_{\mathbb B}(\tau)= \frac{\mathrm{BGW}^\mu( \tau)}{\mathrm{BGW}^\mu(\mathbb{B})}, 
\qquad \tau \in \mathbb{B}.\]

This framework is readily adapted to  multi-type BGW trees.
In this case, we consider trees where each vertex has a type (belonging to some finite set).
The probability of a finite tree $\tau$ is still given by product over its vertices as in \eqref{eq:def_GW},
but the factors depend on the type of the vertex and on its number of children of each type.
In the sequel we focus on  a very special case of such multi-type BGW trees, which we call  {\em two-type alternating BGW trees}. In this special case, given two offspring distributions $\mub$ and $\muc$,
the $\mathrm{BGW}$ measure $\BGW^{\mub,\muc}$ is a probability measure on $\mathbb{A} ^{\infty}$ satisfying
\begin{equation}
\label{eq:def_GW2}
\BGW^{\mub,\muc}(\tau) = \left(  \prod_{u \in \bullet_{\tau}} \mub(k_u) \right) \cdot \left(  \prod_{u \in \circ_{\tau}} \muc(k_u) \right), \qquad \tau \in \mathbb{A}.
\end{equation}
In words, black vertices (i.e.\ vertices at even generation) have offspring distribution $\mub$,
while white vertices (i.e.\ vertices at odd generation) have offspring distribution $\muc$.
As before, for $\mathbb B \subseteq \mathbb A$ with $\BGW^{\mub,\muc}(\mathbb B)>0$, we can define
the conditional probability  measure $\BGW^{\mub,\muc}$ given  $\mathbb{B}$.

We will need a slight variation of the previous definition:
if $\tmub$, $\mub$ and $\muc$ are three offspring distributions,
we will consider a probability measure ${\BGW}^{\tmub,\mub,\muc}$ on $\mathbb{A} ^{\infty}$ 
such that
\begin{equation}
  \label{eq:BGWtilde}
  {\BGW}^{\tmub,\mub,\muc}(\tau)=\tmub(k_\emptyset) \, \left(  \prod_{u \in \bullet_{\tau},u\neq \emptyset} \mub(k_u) \right) \cdot \left(  \prod_{u \in \circ_{\tau}} \muc(k_u) \right), \qquad \tau \in \mathbb{A}.
\end{equation}
In words, the only modification with the previous case is that the
root has offspring distribution $\tmub$ and not $\mub$ as the other black vertices.

\subsubsection{Representation of $\PKn$ as a $\BGW$ tree.}
As before, let $(\tb_{1}^{(n)}, \ldots, \tb_{n-1}^{(n)})$ denote a random uniform minimal factorization of length $n$
and we are interested in the behaviour of the non-crossing partition
$\PPP(\tb_{1}^{(n)} \tb_{2}^{(n)} \cdots \tb_{K_{n}}^{(n)})$ as $n \rightarrow \infty$.
By applying the bijection $\mathcal T$, we get a random tree.
A key idea in this work is to identify this random tree 
with a conditioned two-type alternating BGW tree.

\begin{proposition}\label{prop:bitype}
Fix $n \geq 1$ and $1 \leq K_{n} \leq n-1$ . Let $\mub$ and $\muc$ be the two following offspring distributions:
\[\mub(i)= \ab \cdot  \bb^{i} \cdot \frac{(i+1)^{i-1}}{i!}  \quad (i \geq 0), \qquad \muc(i)= \ac \cdot  (\bc)^{i} \cdot \frac{(i+1)^{i-1}}{i!}  \quad (i \geq 0),\]
where $\ab,\bb,\ac,\bc$ are positive parameters (which may depend on $n$)
such that $\mub$ and $\muc$ are probability distributions.
Define a third probability distribution $\widetilde{\mu}_{\bullet}$
 by $\widetilde{\mu}_{\bullet}(i)=\mub(i-1)$ for $i \geq 1$
 and let $ \tTs$ be a random tree distributed as ${\BGW}^{\tmub,\mub,\muc}$, 
 where ${\BGW}^{\tmub,\mub,\muc}$ is defined in \eqref{eq:BGWtilde}.

Then $ \mathcal{T}( \PPP( \tb_{1}^{(n)} \tb_{2}^{(n)} \cdots \tb_{K_{n}}^{(n)}))$ has the same distribution as $ \tilde{\mathscr{T}}$, conditioned on having $n-K_{n}$ black vertices and $K_{n}+1$ white vertices. 
\end{proposition}

\begin{proof}
To simplify notation, set $\mathscr{S}_{K_n}= \tb_{1}^{(n)} \tb_{2}^{(n)} \cdots \tb_{K_{n}}^{(n)}$.
First note that $\mathscr{S}_{K_n}$ has $n-K_{n}$ cycles in its cycle decomposition, so that $ \PPP( \mathscr{S}_{K_n})$ is a non-crossing partition of size $n$ with $n-K_{n}$ blocks, implying that
$ \mathcal{T}( \PPP( \mathscr{S}_{K_n}))$ has $n-K_{n}$ black vertices and $K_{n}+1$ white vertices.
Now, let $\tau$ be a tree with $n-K_{n}$ black vertices and $K_{n}+1$ white vertices,
and let $P$ be the non-crossing partition such that $\tau= \mathcal{T}(P)$.
From Proposition~\ref{prop:lawproduct}, we have
\begin{multline*}
\Pr{\mathcal{T}( \PPP(\mathscr{S}_{K_n}))=\tau}= \Pr{ \PPP(\mathscr{S}_{K_n})= P} \\
=  \frac{K_{n}! (n-K_{n}-1)!}{n^{n-2}}  \cdot  \left( \prod_{B \in P} \frac{|B|^{|B|-2}}{(|B|-1)!} \right) \cdot \left(  \prod_{B \in  \mathcal{K}(P)} \frac{|B|^{|B|-2}}{(|B|-1)!} \right).
\end{multline*}
Using the identities $|P| =n-K_n$,
\[ \ |\mathcal K(P)|=K_n+1, \quad  \sum_{B \in P} (|B|-1)=n-(n-K_{n})=K_{n},  \quad 
\sum_{B \in  \mathcal{K}(P)} (|B|-1)=n-K_{n}-1,\] 
we can rewrite the above probability as 
\begin{align*}
  \hspace{-1mm}\Pr{\mathcal{T}( \PPP(\mathscr{S}_{K_n}))=\tau}
&= \frac{1}{C_{n,K_n}} \!\cdot\! \left( \prod_{B \in P} \ab \bb^{|B|-1}\frac{|B|^{|B|-2}}{(|B|-1)!} \right)\! \cdot\! \left(  \prod_{B \in  \mathcal{K}(P)} \ac \bc^{|B|-1} \frac{|B|^{|B|-2}}{(|B|-1)!} \right),\\
\text{with }
C_{n,K_n} & =\frac{n^{n-2} \cdot \ab^{n-K_{n}} \bb^{K_{n}} \ac^{K_{n}+1} \bc^{n-K_{n}-1}}{K_{n}! (n-K_{n}-1)!}.
\end{align*}
 Note that black vertices of $\tau$ with $i$ children are in bijection with blocks of $P$ of size $i+1$ (except the root, which is in bijection with a block of size $i$ if it has $i$ children) and white vertices of $\tau$ with $i$ children are in bijection with blocks of $ \mathcal{K}(P)$ of size $i+1$. Therefore, the above product correspond to the right-hand side
 of \eqref{eq:BGWtilde} for the measures $\tmub$, $\mub$ and $\muc$ defined in the proposition,
 and we have
 \[\Pr{\mathcal{T}( \PPP(\mathscr{S}_{K_n}))=\tau}=\frac{1}{C_{n,K_n}} \cdot \P( \tilde{ \mathscr{T} }=\tau).\]
 We sum this equality over all trees $\tau$  that have $n-K_{n}$ black vertices and $K_{n}+1$ white vertices.
 The left-hand side sums to $1$ and we get that $C_{n,K_n}$
 is the probability that $\tilde{ \mathscr{T} }$
 has $n-K_n$ black vertices and $K_n+1$ black vertices.
 The desired result follows.
\end{proof}

\section{Limit theorems for bi-conditioned bi-type BGW trees}
\label{sec:limitheorems}

\begin{table}[htbp]\caption{Table of the main notation and symbols appearing in Section~\ref{sec:limitheorems}.}
\centering
\begin{tabular}{c c p{0.75 \linewidth} }
\toprule
$X, \Xbr,\Xexc$ && The Lévy process characterized by \eqref{eq:Levyc},  its associated  bridge process going from $0$ to $0$ in a unit time and its associated excursion process.\\
$\mubn$, $\mucn$&& The offspring distributions defined by \eqref{eq:mubn_mucn}.\\
$\mathscr{T}_{n}$ & & An alternating two-type BGW tree, with offspring distributions $\mubn$ and $\mucn$ conditioned on having $n-K_{n}$ black vertices and $K_{n}+1$ white vertices. 
 \\
 $\mathscr{T}^{n}$ & & An alternating two-type BGW tree, with offspring distributions $\mubn$, $\mucn$.\\ 
 $\mbn,(\sbn)^{2}$ & & Resp.~the mean and variance of $\mubn$.\\
 $ \mcn,(\scn)^{2}$ & & Resp.~the mean and variance of $\mucn$.\\ 
 $\Scn_{k}$ , $\Sbn_{k}$ && Resp.~the sum of $k$ i.i.d.\ random variables distributed according to $\mucn,\mubn$.\\
\bottomrule
\end{tabular}
\label{tab:seclim}
\end{table}

In this Section, we establish functional limit theorems for two-type alternating BGW trees conditioned on having a fixed number of vertices of both types.
These limit theorems are a key step in the proof of Theorem~\ref{thm:cvlam} (i)$_{c>0}$.
Let us mention that even though we focus on  specific offspring distributions, 
we believe that our methods are robust and may allow to tackle the general study of such BGW trees under this new conditioning.

The section is organized as follows.
In Section~\ref{ssec:Levy} we start by defining the continuous-time processes
that appear in the limiting objects.
Then, in Section~\ref{ss:coding}, we discuss encodings of general two-type BGW trees.
We then establish functional scaling limit results for encodings 
of particular two-type BGW trees Section~\ref{ss:inv}.
This is based on technical uniform local limit lemmas,
that are proved in Section~\ref{ss:llt}.
In the mid-time, in Section~\ref{ss:descendants},
we present some useful results, as a preparation for the next section.

\subsection{Lévy processes, bridge and excursions}
\label{ssec:Levy}
We first start by presenting some material about Lévy processes, bridge and excursions that will be involved in the limiting processes (this is mostly standard, see e.g.~\cite{Ber96,Kyp06} for details).
These processes will also be used to construct the lamination $\mathbf{L}_{c}$. We shall work with the space ${\D}([0, 1], \R)$ of real-valued càdlàg functions on $[0,1]$ equipped with the Skorokhod $J_{1}$ topology (see Chapter VI in \cite{JS03} for background).

A \emph{Lévy process}  $(Z_t)_{t \ge 0}$ is a random càdlàg process with independent, stationary increments. For $t \geq 0$, the characteristic function of $Z_{t}$ has the form
\[\Es{e^{\textrm{i}\theta Z_{t}}} = \exp\left[ t \left( a\textrm{i} \theta -\tfrac{1}{2}\sigma^2 \theta^2 + \int_{\R}
(e^{\textrm{i}\theta x} -1 +\textrm{i}\theta x \mathbbm{1}_{[|x| <1]}) \Pi({\d}x)\right) \right], \qquad \theta \in \R,\]
where $a \in \R$ and $\sigma^2$ is the variance of the Brownian part.
The sigma-finite measure $\Pi({\d}x)$ on $\R^{*}$ satisfies $\int_{\R} (1 \wedge x^{2}) \Pi({\d}x) < \infty$ and is called the {Lévy measure} of $Z$. Roughly speaking, it encodes the rate of arrival of jumps. It is well known (see \cite[Theorem 3.6]{Kyp06}) that if $\Pi$ is supported on $\R_{+}$, then $\Es{e^{-\lambda X_{t}}}<\infty$ for every $t, \lambda \geq 0$.

\paragraph*{A particular Lévy process.} In the following, we are interested in a particular Lévy process $(X_t)$,
whose Laplace transform at time $t$ is given by
\begin{equation}
  \Es{e^{-\lambda X_t}}=e ^{-t \Phi(\lambda)}=e^{t \, c^{2}   \left( 1- \sqrt{1+ \frac{2 \lambda}{c} } \right) +t\, \lambda c }, \qquad \lambda \geq 0.
    \label{eq:Levyc}
\end{equation}
The fact that this indeed defines a Lévy process is not clear a priori, but this readily follows from  \cite[Sec.1.2.5]{Kyp06}
(see the first item below).
From the latter, we also infer the following properties:
\begin{itemize}
  \item The process $Y^c_t \coloneqq X_t +ct$ coincides with the so-called \emph{inverse Gaussian process},
    defined by the following distributional equality 
    \[(Y^c_t)_{t \ge 0} \mathop{=}^{(d)} \left(  \inf \{u>0 : W_{u}+\sqrt{c} \cdot u > c^{3/2} t \}  \right) _{t \geq 0},\]
    where $(W_{t})_{t \geq 0}$ is a standard Brownian motion.
    (The comparison of \eqref{eq:Levyc} with the formula for the characteristic exponent
    of inverse Gaussian processes given in \cite[Sec.1.2.5]{Kyp06} shows the equality in distribution for any given $t$;
    since both are Lévy processes, this is enough to conclude that the processes
    $(X_t +ct)_{t \ge 0}$ and $\left(  \inf \{u>0 : W_{u}+\sqrt{c} \cdot u > c^{3/2} t \}  \right) _{t \geq 0}$
    indeed coincide in distribution.)
    This implies that $(X_t +ct)_{t \geq 0}$ is almost surely increasing, i.e.\ is a {\em subordinator}.
    The process $(X_t)$ has therefore paths of bounded variation.

We also note for future reference that the subordinator $Y^{c}_t$ has zero drift
(see \cite[Eq.~(2.21)]{Kyp06} for the definition of the drift of a process with bounded variation,
and \cite[Exercise 2.11]{Kyp06} for a means to compute this drift),
so that $(X_t)$ has negative drift $-c$.
  \item The Lévy measure $\Pi$ and the density $d_t$ of $X_t$ are given by (see \cite[Exercice 1.6]{Kyp06}):
  \begin{equation}
  \label{eq:densiteXt}
  \Pi({\d}x)= \mathbbm{1}_{x>0} \cdot \frac{c^{3/2}}{\sqrt{2 \pi x^{3}}} \cdot e^{- \frac{c }{2} \cdot x} \ \mathrm{d}x, \qquad   d_{t}(x) =  \sqrt{ \frac{c^3 t^2}{2 \pi (x+ct)^{3}}} \cdot e^{- \frac{c x^{2}}{2(x+ct)}}  \mathbbm{1}_{x \geq -ct}.
  \end{equation}

   There is no Gaussian component. Note that $\Pi({\d}x)$ is supported on $\R_+$, so that $(X_t)$ is {\em spectrally positive} (i.e.\ makes only positive jumps).
    On the other hand $\Pi(\R_+)=\infty$, so that $(X_t)$ makes a.s. infinitely many jumps in every non-empty interval.    
    \item  Since $(X_{t})$ has paths of bounded variation, by \cite[Section 8.1, \emph{Regularity}]{Kyp06} (applied with the spectrally negative process $-X$),  $0$ is regular for $(-\infty,0)$ but irregular for $(0,\infty)$, meaning that
\begin{equation}
\label{eq:regularity}\Pr{\inf \{t \geq 0 : X_{t}<0 \}=0}=1, \qquad \Pr{\inf \{t \geq 0 : X_{t}>0 \}=0}=0.
\end{equation}
\end{itemize}

In the following, $(X_t)$ will always refer to this specific Lévy process
(we drop the dependence in $c$ to simplify notation).
A simulation of $(X_t)$ when $c=5$ is given in Figure~\ref{fig:X} (left image).
\begin{figure}[t]
\begin{center}
 \includegraphics[height=2.6cm]{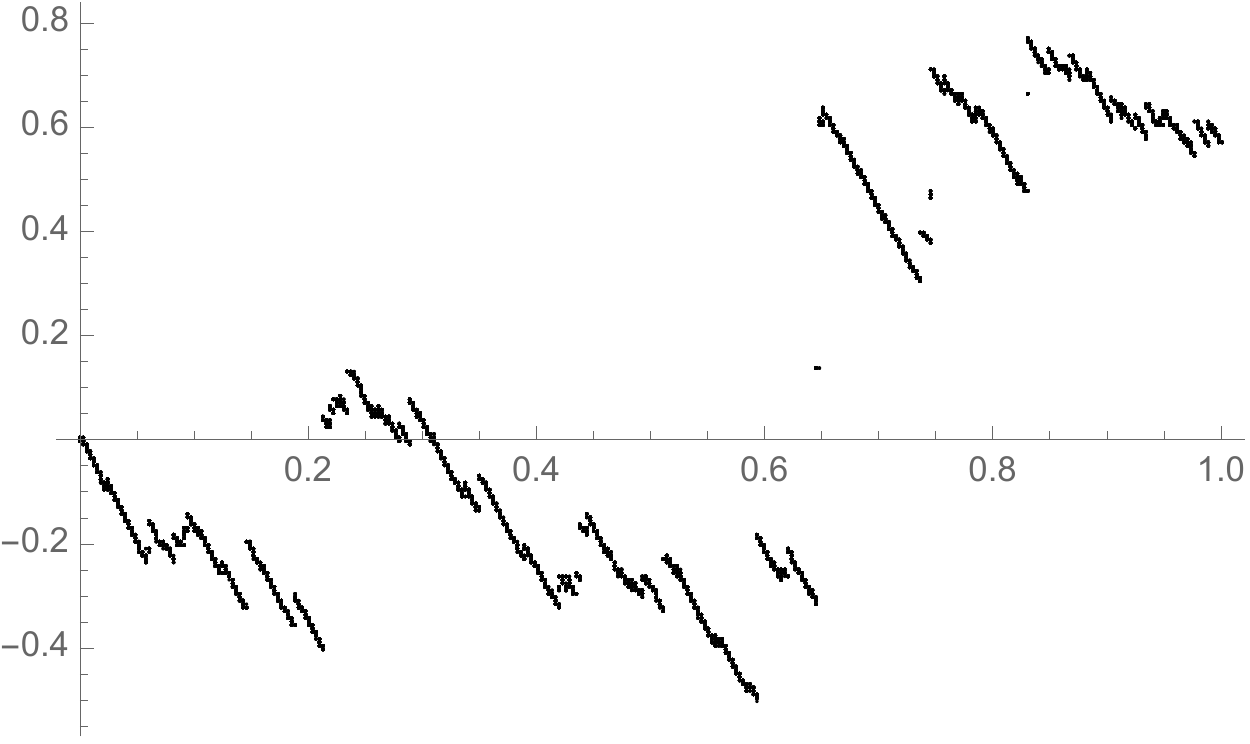}\qquad 
 \includegraphics[height=2.6cm]{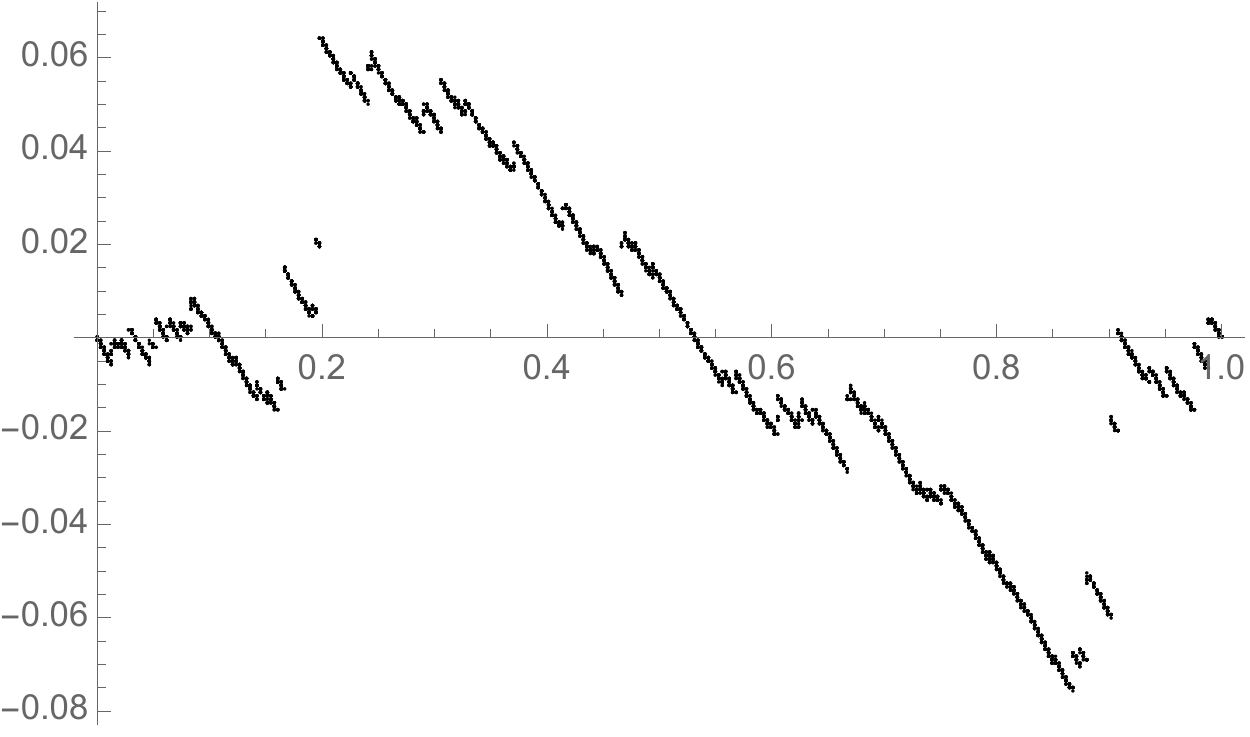}\qquad 
 \includegraphics[height=2.6cm]{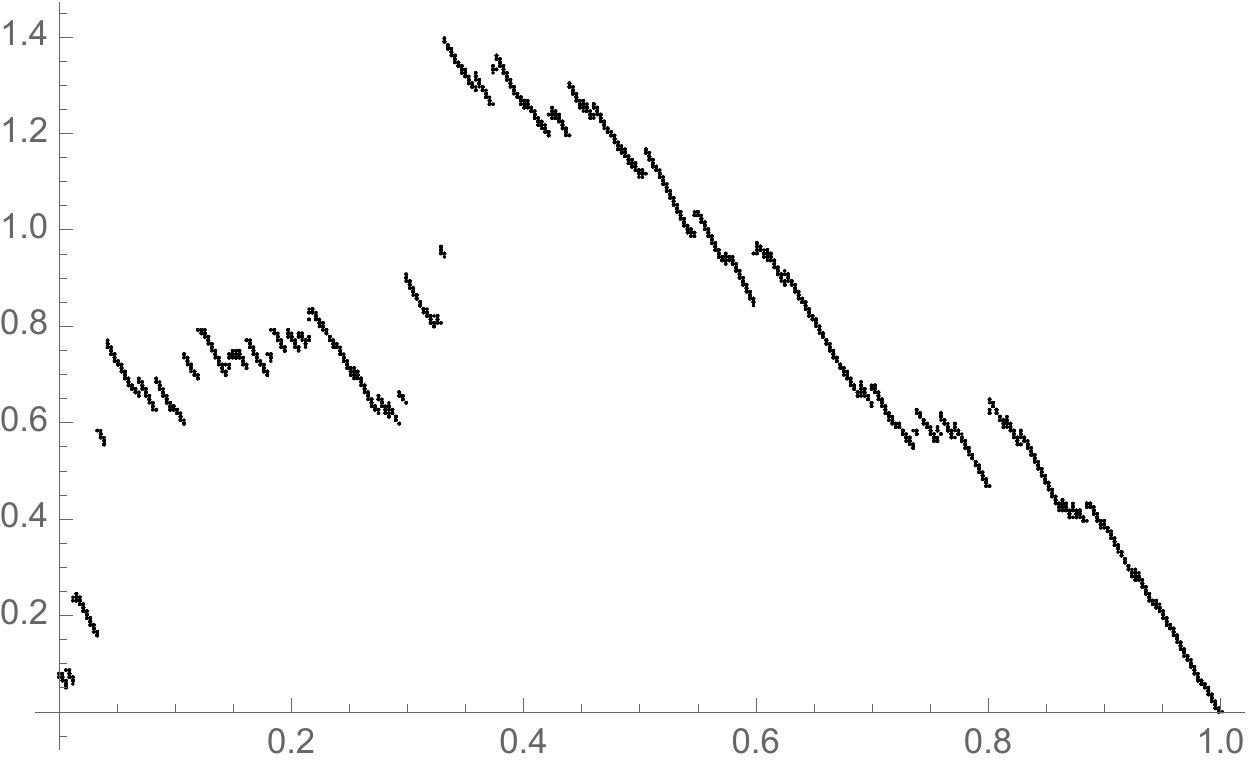}
 \caption{\label{fig:X}For $c=5$, from left to right simulations of :  $(X_{t})_{0 \leq t \leq 1}$, $\Xbr$ and $\Xexc$ (obtained as the Vervaat transform of $\Xbr$).}
\end{center}
\end{figure}

\paragraph*{Bridges and excursions of Lévy processes.}  We can construct the associated Lévy bridge process $\Xbr$
going from $0$ to $0$ in a unit time.
Specifically, by  \cite[Theorem 5]{CUB11} (see also the note added in proof in \cite{UB14}, and \cite{Kal81} for less general results), we  can construct the associated \emph{Lévy bridge} process $\Xbr$
going from $0$ to $0$ in a unit time (indeed, in the notation of  \cite{UB14}, the hypotheses (H1), (H2) and (H3) are satisfied, see Section 2.1 in \cite{CUB11}, since denoting $\Psi(x)=\Phi(-ix)$ the characteristic exponent of $X$, $\exp(-t \Psi)$ is integrable for every $t>0$).
Roughly speaking, $\Xbr$ has the law of $(X_t)_{0 \leq t \leq 1}$ conditioned to return to $0$ at time $1$.
More formally, its law is characterized by the fact that,
for every $0<u<1$ and every bounded continuous functional $F: \mathbb{D}([0,u]) \rightarrow \R$,
\begin{equation}
\label{eq:lawbridge}\Es{F \left( \Xbr_{s} : 0 \leq s \leq u  \right) } = \Es{F \left( X_{s} : 0 \leq s \leq u  \right) \frac{d_{1-u}(-X_{u})}{d_{1}(0)} },
\end{equation}
where we recall that $d_{t}$ is the density of $X_{t}$ (see \eqref{eq:densiteXt} for its explicit value). 
We refer to Figure~\ref{fig:X} (middle image) for a simulation of $\Xbr$  when $c=5$.

 Finally, we will also consider the associated \emph{excursion process} $\Xexc$.
Roughly speaking, $\Xexc$   is obtained by conditioning on the event
\[\{X_{1}=0, X_{t} \geq 0 \textrm{ for every } 0 \leq t \leq 1\}.\] 
Formally, following Miermont \cite[Definition~1]{Mie01}, we define the Vervaat transform $\mathcal V\, f$ of a bridge $f$
(i.e.\ a function $f \in \D([0,1])$ with $f(0)=f(1-)=f(1)$) as follows:
let $t_{\min}$ be the location of the right-most minimum of $f$, that is, the largest $t$ such
that $f(t-) \wedge f (t) = \inf\, f$, then, for $t \in [0,1)$ set 
\[\mathcal V\, f (t) =f(t + t_{\min} [\mathrm{mod}\ 1]) - \inf_{[0,1]} f, \]
and $\mathcal V\, f (1) = \lim_{t \to 1^-} \mathcal V\, f (t)$.
We then set $\Xexc= \mathcal V\, \Xbr$.
See Figure~\ref{fig:X} (right image) for a simulation  when $c=5$.

It is important to observe that $\Xexc_{0}>0$ almost surely. Indeed, since $0$ is irregular for $(0,\infty$) (see \eqref{eq:regularity}), by \cite[Theorem 3.1 (b)]{Mil77}, the Lévy process $X$ jumps when it reaches its infimum on any fixed interval, and this property also holds for $\Xbr$ by absolute continuity \eqref{eq:lawbridge}.  The fact that $\Xexc_{0}>0$ then follows from the construction $\Xexc= \mathcal V\, \Xbr$ (see also Proposition 4 in \cite{Mie01}).

\subsection{Coding bi-conditioned bi-type BGW trees}
\label{ss:coding}

In this section, inspired by \cite{CL16}, we present general results concerning  bi-conditioned bi-type BGW trees, which extend the usual coding of one-type BGW trees using a random walk (see in particular \cite[Sec.~1.1]{LG05} for the definition of the  \emph{{\L}ukasiewicz path}, or equivalently the \emph{depth-first search}, of a plane tree, which we will use).  If $\tau \in \mathbb{A}$ is a plane rooted  tree,  recall that $\bullet_{\tau}$ denotes the set  of all black vertices of $\tau$ (at even generation) and by  $\circ_{\tau}$ the set of all white vertices of $\tau$ (at odd generation). Denote by $(v^{\bullet}_{i}(\tau))_{0 \leq i < \abs{\bullet_{\tau}}} $ the black vertices of $\tau$ listed in lexicographical order.  For $1 \leq i \leq \abs{\bullet_{\tau}}$, let $H_{i}(\tau)$ be the number of white children of $v^{\bullet}_{i-1}(\tau)$ and let $B_{i}(\tau)$ be the number of black grandchildren of $v^{\bullet}_{i-1}(\tau)$. Finally,  set $\oH_{0}(\tau)=\oB_{0}(\tau)=0$, and for $1 \leq i \leq \abs{\bullet_{\tau}}$ 
\[\oH_{i}(\tau)=H_{1}(\tau)+ \cdots+H_{i}(\tau), \qquad  \oB_{i}(\tau)=B_{1}(\tau)+ \cdots+B_{i}(\tau)-i.\]
Alternatively, consider the \emph{reduced black subtree} of $\tau$, that is the tree obtained by keeping only black vertices
with the same ancestor/descendant relations. 
Then $(\overline{B}_i(\tau))_{0 \le i \le |\bullet_\tau|}$ is its {\L}ukasiewicz path.
This explains the notation $B$, referring to {\em black} vertices;
the notation $H$ comes from the word \emph{hidden},
since white vertices disappear in the reduction operation.

 We consider two non-degenerate offspring distributions $\mub$ and $\muc$ on $\Z_{+}$ and we denote by $\Ts$ a random alternating two-type BGW tree with distribution given by \eqref{eq:def_GW2}. We let $S^{\circ}_{k}$ and $S^{\bullet}_{k}$ respectively denote the sum of $k$ i.i.d.\ random variables distributed according to $\muc$ and $\mub$.
 Let $(H,B)$ be a random variable such that 
 \begin{equation}
    \Pr{H=i,B=j}=\mub(i) \, \Pr{S^{\circ}_{i}=j}, \qquad  i, j \geq 0.
    \label{eq:HB}
  \end{equation}
Let $(H_{i},B_{i})_{i \geq 1}$ be a sequence of i.i.d.\ random variables distributed as $(H,B)$. Set $\oH_{i}=H_{1}+H_{2}+ \cdots+H_{i}$ and  $\oB_{i}=B_{1}+B_{2}+ \cdots+B_{i}-i$ for $i \geq 1$.

\begin{lemma}\label{lem:codeRW} Fix $\nb \geq 1$ and $\nc \geq 0$. The following assertions hold.
\begin{enumerate}
\item[(i)]We have \[\Pr{\ \abs{\bullet_{\Ts}} =\nb,\abs{\circ_{\Ts}}=\nc}=\Pr{\oH_{\nb}=\nc,\oB_{\nb}=-1 \textrm{ and } \oB_{i} \geq 0 \textrm{ for every } 1 \leq i < \nb}.\]
\item[(ii)] If $\Prb{ \abs{\bullet_{\Ts}} =\nb,\abs{\circ_{\Ts}}=\nc}>0$, let $\Ts_{\nb,\nc}$ be a random tree distributed as $\Ts$ conditioned on the event $\{ \abs{\bullet_{\Ts}} =\nb,\abs{\circ_{\Ts}}=\nc \}$. Then the pair of paths $(H_{i}(\Ts_{\nb,\nc}) ,B_{i}(\Ts_{\nb,\nc}))_{1 \leq i \leq \nb}$ has the same distribution as $(H_{i},B_{i})_{1 \leq i \leq \nb}$ conditioned on the event 
  \[\{\oH_{\nb}=\nc,\oB_{\nb}=-1 \textrm{ and } \oB_{i} \geq 0 \textrm{ for every } 1 \leq i < \nb\}.\]
\end{enumerate}
\end{lemma}

\begin{proof}
Let $\tau$ be a finite tree. For $u \in \tau$, recall that $pr(u)$ is the parent of $u$ (when $u \neq \emptyset$). Denote by $\prec_{\bullet}$ the lexicographical order on the black vertices of $\tau$. We define a total order on the white vertices of $\tau$ as follows. For $u,v \in \circ_{\tau}$, if $pr(u) \neq pr(v)$, we set $u \prec_{\circ} v$ if and only if $pr(u) \prec _{\bullet} pr(v)$; if $pr(u)=pr(v)$, we set  $u \prec_{\circ} v$ if and only if $u$ is less than $v$ for the lexicographical order.
Then denote by $(v^{\circ}_{i}(\tau))_{0 \leq i < \abs{\circ_{\tau}}} $ the white vertices of $\tau$ listed in increasing $\prec_{\circ}$ order (we warn the reader that in this proof the vertices 
$(v^{\circ}_{i}(\tau))_{0 \leq i < \abs{\circ_{\tau}}}$ are \emph{not} ordered in lexicographical order,  unlike in all other places in the article).
An example is given on Figure~\ref{fig:bijection}.

\begin{figure}[ht] 
\begin{center}
\begin{scriptsize} 
\begin{tikzpicture}
\coordinate (0) at (0,0);
	\coordinate (1) at (-1.5,1);
		\coordinate (11) at (-1.5,2);
			\coordinate (111) at (-2,3);
				\coordinate (1111) at (-2,4);
					\coordinate (11111) at (-2,5);
			\coordinate (112) at (-1,3);
				\coordinate (1121) at (-1,4);
	\coordinate (2) at (-.5,1);
	\coordinate (3) at (.5,1);
		\coordinate (31) at (0,2);
		\coordinate (32) at (1,2);
			\coordinate (321) at (1,3);
	\coordinate (4) at (1.5,1);
\draw
	(0) -- (1)	(0) -- (2)	(0) -- (3)	(0) -- (4)
	(1) -- (11)
	(11) -- (111) -- (1111) -- (11111)	(11) -- (112) -- (1121)
	(3) -- (31)	(3) -- (32) -- (321)
;
\draw[fill=black]
	(0) circle (1.5pt)
	(11) circle (1.5pt)
	(31) circle (1.5pt)
	(32) circle (1.5pt)
	(1111) circle (1.5pt)
	(1121) circle (1.5pt)
;
\draw[fill=white]
	(1) circle (1.5pt)
	(2) circle (1.5pt)
	(3) circle (1.5pt)
	(4) circle (1.5pt)
	(111) circle (1.5pt)
	(112) circle (1.5pt)
	(321) circle (1.5pt)
	(11111) circle (1.5pt)
;
	% Labels
\draw
	(11) node[left] {  $v^{\bullet}_{1}$}
	(1111) node[left] {$v^{\bullet}_{2}$}
	(1121) node[left] {$v^{\bullet}_{3}$}
	(31) node[left] {$v^{\bullet}_{4}$}
	(32) node[left] {$v^{\bullet}_{5}$}
	(0) node[below] {$v^{\bullet}_{0}$}
	(1) node[left] {$v^{\circ}_{0}$}
	(2) node[left] {$v^{\circ}_{1}$}
	(3) node[left] {$v^{\circ}_{2}$}
	(4) node[left] {$v^{\circ}_{3}$}
	(111) node[left] {$v^{\circ}_{4}$}
	(112) node[left] {$v^{\circ}_{5}$}
	(11111) node[left] {$v^{\circ}_{6}$}
	(321) node[left] {$v^{\circ}_{7}$}
;
\end{tikzpicture}
\end{scriptsize}
\end{center}
\caption{An example of a tree $\tau$, with its $6$ black vertices listed according to $\prec_{\bullet}$ and its $8$ white vertices according to $\prec_{\circ}$. Here, the associated paths are $\Phi(\tau)=\big((4,2,1,0,0,1), (1,0,2,0,1,0,1,0) \big)$.
}
\label{fig:bijection} \end{figure}
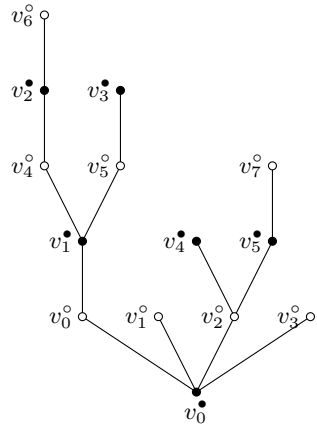

For $1 \leq i \leq \abs{\circ_{\tau}}$, we let $W_{i}(\tau)$ be the number of children of $v^{\circ}_{i-1}(\tau)$. 
It is a simple matter to check that the mapping
\[\begin{array}{ccccc}
\Phi& : & \mathbb{A} & \longrightarrow & \mathbb{V} \\
 & & \tau & \longmapsto &  \left(  (H_{i}(\tau))_{1 \leq i \leq \abs{\bullet_{\tau}}} ,   (W_{i}(\tau))_{1 \leq i \leq \abs{\circ_{\tau}}} \right) \\
\end{array}\]
is a bijection, where
\begin{equation*}
\! \!\mathbb{V}= \bigcup_{\nb \geq 1, \nc \geq 0} \left\{ \begin{split} &
  \left(  (h_{1},h_{2}, \ldots,h_{\nb}), (w_{1}, w_{2}, \ldots, w_{\nc}) \right) \in \Z_{+}^{\nb} \times \Z_{+}^{\nc} :  \\
  &  \qquad\qquad\qquad h_{1}+ \cdots+h_{\nb}=\nc, w_{1}+ \cdots+w_{\nc}=\nb-1 \\
     & \qquad\qquad \textrm{and } w_{1}+ w_{2}+ \cdots+w_{h_{1}+ \cdots+h_{i}} \geq i \textrm{ for every } 1 \leq i \leq \nb-1 \end{split}\right\}.
\end{equation*}
See again Figure~\ref{fig:bijection} for an example.
In addition, this bijection has the following property: if $(\bm{H},\bm{W})= \left(  (h_{1},h_{2}, \ldots,h_{\nb}), (w_{1}, w_{2}, \ldots, w_{\nc}) \right)$ and if $\tau$ is a tree such that $\phi(\tau)=(\bm{H},\bm{W})$, then $\tau$ has $\nc$ white vertices and $\nb$ black vertices and $\Pr{ \Ts=\tau}= \prod_{i=1} ^{\nb} \mu_{\bullet}(h_{i}) \cdot \prod_{j=1}^{\nc}\mu_{\circ}(w_{j})$.
Also, setting $h_{0}=0$ and by using the change of variables $b_{i}=w_{h_{1}+ \cdots+h_{i-1}}+ w_{h_{1}+ \cdots+h_{i-1}+1}+ \cdots+w_{h_{1}+ \cdots+h_{i}}$, $\overline{b}_{i}= b_1+\dots+b_i-i$ and $\overline{h}_{i}=h_{1}+ \cdots+h_{i}$, we see that
\[
\Pr{\ \abs{\bullet_{\Ts}} =\nb,\abs{\circ_{\Ts}}=\nc}= \sum_{\substack{ h_{1}, \cdots, h_{\nb} \geq 0\\ b_{1}, \cdots, b_{\nb} \geq 0 :  \\ \overline{h}_{\nb}=\nc, \overline{b}_{\nb}=-1 \\ \overline{b}_{i} \geq 0 \textrm{ for every } 1 \leq i  < \nb}}  \prod_{i=1} ^{\nb} \mu_{\bullet}(h_{i}) \cdot \sum_{x_{1}+\cdots+x_{h_{i}}= b_{i}} \prod_{j=1}^{h_{i}}\mu_{\circ}(x_{j}).
\]
But \[ \mu_{\bullet}(h_{i}) \cdot \sum_{x_{1}+\cdots+x_{h_{i}}= b_{i}} \prod_{j=1}^{h_{i}}\mu_{\circ}(x_{j})= \Pr{H=h_{i},B=b_{i}}.\]
This readily implies (i). The second assertion is established by similar arguments, and we leave the details to the reader.
\end{proof}

As in the monotype case, the probability in Lemma~\ref{lem:codeRW} (i)
can be computed through a variant of the cyclic lemma.
For this, we need to introduce the notion of Vervaat transform of a bivariate sequence (with respect to the second variable).
First, if $\mathbf{x}=(a_{k},b_{k})_{1 \leq k \leq m}$ is a sequence of integer couples and $i \in \Z/m\Z$, we define the cyclic shift $\mathbf{x}^{(i)}$ by $x^{(i)}_k=(a_{k+i \mod m},b_{k+i \mod m})$ for $1 \leq k \leq m$ (where representatives modulo $m$ are chosen in $ \{1,2, \ldots,m\}$).  We then define the (discrete) Vervaat transform $ \Vcd(\mathbf{x})$ as follows (we introduce the superscript $\mathsf{d}$ to underline the fact that this transformation acts on discrete sequences). Let $i_{*}(\mathbf{x})$ be defined by
\[i_{*}(\mathbf{x}) = \min \left\{ j  \in \{1,2, \ldots,m\} ; b_{1}+b_{2}+\cdots+b_{j}= \min_{1 \leq i \leq m} (b_{1}+b_{2}+\cdots+b_{i}) \right\}.\]
Then $\Vcd(\mathbf{x}) \coloneqq \mathbf{x}^{(i_{*}(\mathbf{x}))}$.
  
\begin{lemma}\label{lem:vervaat}
The following assertions hold for every integers $\nb \geq 1$ and $\nc \geq 0$.
\begin{enumerate}
\item[(i)] We have
  \[\hspace{-5mm}\Pr{\oH_{\nb}=\nc, \oB_{\nb}=-1 \textrm{ and } \oB_{i} \geq 0 \textrm{ for every } 1 \leq i <\nb }= \frac{1}{\nb} \Pr{\oH_{\nb}=\nc, \oB_{\nb}=-1 }.\]
\item[(ii)] The Vervaat transform of $(H_{i},B_{i})_{1 \leq i \leq \nb}$ conditionally given the event
\[ \{\oH_{\nb}=\nc, \oB_{\nb}=-1 \}\]
has the same distribution of $(H_{i},B_{i})_{1 \leq i \leq \nb}$ conditionally given the event
\[ \{\oH_{\nb}=\nc, \oB_{\nb}=-1, \oB_{i} \geq 0 \textrm{ for every } 1 \leq i <\nb \}.\]
\end{enumerate}
\end{lemma}
\begin{proof}This follows from a simple extension of the so-called cyclic lemma (see e.g.~\cite[Lemma 6.1]{Pit06}). Indeed, if $\mathbf{x}=(a_{i},b_{i})_{1 \leq i \leq m}$ are pairs of integers such that $b_{1}+b_{2}+\cdots+b_{m}=-1$, then there is a unique $j \in \Z/m\Z$ such that the cyclic shift $\mathbf{x}^{(j)}=(a^{(j)}_{i},b^{(j)}_{i})_{1 \leq i \leq m}$ fulfills $b^{(j)}_{1}+ \cdots +b^{(j)}_{i} \geq 0$ for every $1 \leq i<m$. It is then standard to obtain the desired results by an exchangeability argument; we leave details to the reader. 
\end{proof}

The following lemma will allow, roughly speaking, to decouple the dependence between the random variables $(H_{i})$ and $(B_{i})$. It is a particular case of \cite[Theorem 1.2]{CL16}. However, in our setting it is elementary, so we provide a proof.

\begin{lemma}\label{lem:magique}
We have, for every integers $n_{1},n_{2},n_{3} \geq 0$,
\[\Pr{\oH_{n_{1}}=n_{2},\oB_{n_{1}}=n_{3}}=\Pr{S^{\bullet}_{n_{1}}=n_{2}} \Pr{S^{\circ}_{n_{2}}=n_{1}+n_{3}}.\]
\end{lemma}

\begin{proof}
We have
\begin{align*}
\Pr{\oH_{n_{1}}=n_{2},\oB_{n_{1}}=n_{3}}&=  \sum_{\substack{h_{1}+ \cdots+h_{n_{1}} =n_{2}  \\ b_{1}+ \cdots +b_{n_{1}} =n_{1}+n_{3} }} \prod_{i=1}^{n_{1}} \mub(h_{i}) \, \Pr{S^{\circ}_{h_{i}} =b_{i}} \\
&=  \sum_{h_{1}+ \cdots+h_{n_{1}} =n_{2}  }    \left( \prod_{i=1}^{n_{1}} \mub(h_{i})  \right) \left( \sum_{ b_{1}+ \cdots b_{n_{1}} =n_{1}+n_{3}}  \prod_{i=1}^{n_{1}}\Pr{S^{\circ}_{h_{i}} =b_{i}}  \right)\\
&= \sum_{h_{1}+ \cdots+h_{n_{1}} =n_{2}}  \left( \prod_{i=1}^{n_{1}} \mub(h_{i}) \right) \Pr{S^{\circ}_{n_{2}}=n_{1}+n_{3}}\\
&=\Pr{S^{\bullet}_{n_{1}}=n_{2}} \Pr{S^{\circ}_{n_{2}}=n_{1}+n_{3}}.
\end{align*}
This completes the proof.
\end{proof}

In particular, note that 
\[\Pr{\oH_{\nb}=\nc, \oB_{\nb}=-1}=\Pr{S^{\bullet}_{\nb}=\nc} \Pr{S^{\circ}_{\nc}=\nb-1}.\]

By combining  Lemmas~\ref{lem:codeRW} and \ref{lem:magique}, we can express the probability that a two-type alternating random BGW tree has a given number of vertices of both types.
In view of future use, we also consider two-type alternating random BGW trees
with {\em with white root} (and offspring distributions $\muc,\mub$).
We keep the condition that colors should be alternating so that,
in this setting, all vertices at even heights are white and have offspring distribution $\muc$,
while vertices at odd heights are black and have offspring distribution $\mub$.
Such a two-type alternating random BGW tree with white root will be denoted $\Tsc$.
In contrast, we denote by  $ \Tsb$ the two-type alternating random BGW tree where the root is black,
which was simply denoted by $\Ts$ before.

\begin{corollary}
  \label{cor:probaGivenNumber}
For every $\nb \geq 1$ and $\nc \geq 0$, we have
\begin{equation}
\Pr{
\abs{\bullet_{\Tsb}}=n_\bullet,\, \abs{\circ_{\Tsb}}=n_\circ } = \frac{1}{n_\bullet} \Pr{S^{\bullet}_{n_\bullet}=n_\circ} \Pr{S^{\circ}_{n_\circ}=n_\bullet-1}
\label{eq:probaGivenNumberVertexRootBlack}
\end{equation}
and
\begin{equation}
\Pr{
\abs{\bullet_{\Tsc}}=n_\bullet,\, \abs{\circ_{\Tsc}}=n_\circ } = 
\frac{1}{n_\circ} \Pr{S^{\bullet}_{n_\bullet}=n_\circ-1} \Pr{S^{\circ}_{n_\circ}=n_\bullet}.
\label{eq:probaGivenNumberVertexRootWhite}
\end{equation}
\end{corollary}

The identity \eqref{eq:probaGivenNumberVertexRootBlack} is an immediate consequence of  Lemmas~\ref{lem:codeRW} and \ref{lem:magique},  and \eqref{eq:probaGivenNumberVertexRootWhite}  follows by symmetry. We will need (in the proof of the forthcoming Lemma~\ref{lem:root}) an extension to forests with a fixed number of components: if $ \mathscr{F}^{\circ}_{j}$ denotes a collection of $j$ i.i.d.\ BGW trees distributed as  $ \Tsc$ (with a white root), then
\begin{equation}
\label{eq:forestc}
\Pr{\abs{\bullet_{\mathscr{F}^{\circ}_{j}}}=n_\bullet,\, \abs{\circ_{\mathscr{F}^{\circ}_{j}}}=n_\circ } = 
\frac{j}{n_\circ} \Pr{S^{\bullet}_{n_\bullet}=n_\circ-j} \Pr{S^{\circ}_{n_\circ}=n_\bullet}.
\end{equation}
Also,  if $ \mathscr{F}^{\bullet}_{j}$ denotes a collection of $j$ i.i.d.\ BGW trees distributed as  $ \Tsb$ (with a black root), then
\begin{equation}
\label{eq:forestb}
\Pr{\abs{\bullet_{\mathscr{F}^{\bullet}_{j}}}=n_\bullet,\, \abs{\circ_{\mathscr{F}^{\bullet}_{j}}}=n_\circ } = 
\frac{j}{n_\bullet} \Pr{S^{\bullet}_{n_\bullet}=n_\circ} \Pr{S^{\circ}_{n_\circ}=n_\bullet-j}.
\end{equation}
The proofs are similar, and we leave the details to the reader.

\subsection{A functional invariance principle for alternating BGW trees}
\label{ss:inv}

It is now time to consider specific offspring distributions.
These distributions will be chosen of the form
\begin{equation}
  \mubn(i)= \abn \cdot  (\bbn)^{i} \cdot \frac{(i+1)^{i-1}}{i!}  \quad (i \geq 0), \qquad \mucn(i)= \acn \cdot  (\bcn)^{i} \cdot \frac{(i+1)^{i-1}}{i!}  \quad (i \geq 0).
  \label{eq:mubn_mucn}
\end{equation}
Indeed, according to Proposition~\ref{prop:bitype}, 
alternating BGW trees with these distributions are connected to minimal factorizations
(for the moment, let us forget that, in Proposition~\ref{prop:bitype}, the root has a different offspring distribution).
We start by specifying the choices of the parameters
$\abn$, $\bbn$, $\acn$ and $\bcn$.
(Note that the chosen parameters and hence the offspring distributions depend on $n$.)

\begin{lemma}\label{lem:exist} For every $n \ge 1$,  fix $K_{n} \in \{1,\dots,n-1\}$.
\begin{enumerate}
  \item[(i)] We may choose positive parameters  $\abn,\bbn,\acn,\bcn$ in a unique way  such that \eqref{eq:mubn_mucn}
defines probability distributions with respective means
\[m^{n}_{\bullet}  \, \coloneqq  \, \sum_{i=0}^{\infty} i \cdot \mubn(i)= \frac{K_{n}+1}{n-K_{n}}, \qquad m^{n}_{\circ} \, \coloneqq \, \sum_{i=0}^{\infty} i \cdot \mucn(i)= \frac{n-K_{n}}{K_{n}+1}.\]
\item[(ii)] Assume that $K_{n} \rightarrow \infty$. If $\tfrac{K_{n}}{n} \rightarrow 0$ as $n \rightarrow \infty$, then
\[(\sbn)^{2} \sim \frac{K_{n}}{n}, \quad  \bbn \sim \frac{K_{n}}{n}  \qquad \text{ and } \qquad  (\scn)^{2} \sim   \big( \frac{n}{K_{n}} \big)^{3},\]
where $(\sbn)^{2}$ and $(\scn)^{2}$ denote respectively the variance of $\mubn$ and $\mucn$.
In particular, if $c>0$ is fixed and $\tfrac{K_{n}}{\sqrt{n}} \rightarrow c$ as $n \rightarrow \infty$, then   $ (\sbn)^{2} \sim \tfrac{c}{\sqrt{n}}$ and $(\scn)^{2} \sim  \frac{n^{3/2}}{c^{3}}$.
\end{enumerate}
\end{lemma}

\begin{proof}Let $F$ be the power series defined by
\begin{equation}
\label{eq:F}F(z)= \sum_{k \geq 0} \frac{(k+1)^{k-1}}{k!} z^{k}.
\end{equation}
It is elementary to check that $G(z)= \frac{z F'(z)}{F(z)}$ defines a continuous increasing function on $[0,1/e)$ and that $\lim_{z \rightarrow 0+} G(z)=0$, $\lim_{z \rightarrow 1/e-} G(z)=\infty$. One then chooses $\bbn, \bcn$ such that respectively $G(\bbn)=\frac{K_{n}+1}{n-K_{n}}$, $G(\bcn)= \frac{n-K_{n}}{K_{n}+1}$, and then $\abn=1/F(\bbn)$, $\acn=1/F(\bcn)$.
This proves (i).

For (ii), since $G(z)=z+o(z)$ as $z \rightarrow 0$ and since $G(\bbn)= \frac{K_{n}+1}{n-K_{n}}$, we get that $\bbn \sim \frac{K_{n}}{n}$ as $n \rightarrow \infty$.
The estimate for the variance $(\sbn)^{2}$ is then obtained by using the expression
\hbox{$(\sbn)^{2}= \frac{(\bbn)^{2} F''(\bbn)}{F(\bbn)}+\mbn-(\mbn)^{2}$:}
the dominant term when $n  \rightarrow \infty$ is $\mbn$
and we have \hbox{$(\sbn)^{2}\sim \frac{K_{n}}{n}$}.

The estimate concerning $(\scn)^{2}$ is more subtle, as it involves the behavior of $F$ near its radius of convergence $1/e$.
We can analytically extend $F$
on a complex split neighborhood on $1/e$ (i.e.\ on a complex neighborhood of $1/e$, without the real half-line $[1/e,+\infty)$).
Indeed we have $F(z)=-W(-z)/z$, where $W$ is the Lambert function (see \cite[Eq~3.1]{CGHJK96}).
Using \cite[Eq~4.22]{CGHJK96}, as $z \rightarrow 0$, we have
\begin{equation}
\label{eq:devF}F \left( \frac{1}{e}-z \right)=e- \sqrt{2} e^{3/2} \sqrt{z}+ o(\sqrt{z}),
\end{equation}
where $z$ is a complex number avoiding the negative real line
(throughout the paper, we use the principal determination of 
$\sqrt{z}$ when $z$ is in $\CC \backslash \R_{<0}$).
By singular differentiation \cite[Theorem VI.8 p. 419]{FS09},
we get expansions for $F' \big( \tfrac{1}{e}-z \big)$ and $F'' \big( \tfrac{1}{e}-z \big)$
by differentiating the right-hand side of \eqref{eq:devF}.
Hence $G(1/e-z) \sim \frac{1}{\sqrt{2e}} \cdot \frac{1}{\sqrt{z}}$ as $z \rightarrow 0$.
By definition $\bcn$ is the solution of $G(\bcn)=\tfrac{n-K_n}{K_n+1}$, so that 
\begin{equation}
  1/e-\bcn  \quad \mathop{\sim}_{n \rightarrow \infty} \quad  \frac{1}{2e} \left( \frac{K_n}{n-K_n}  \right)^{2}  \quad \mathop{\sim}_{n \rightarrow \infty} \quad  \frac{1}{2e} \left( \frac{K_n}{n}  \right)^{2}. 
  \label{eq:est_bcirc}
\end{equation}
We again use the expression \hbox{$(\scn)^{2}= \frac{(\bcn)^{2} F''(\bcn)}{F(\bcn)}+m_{\circ}-m_{\circ}^{2}$}
to estimate the variance.
An easy computation gives $(\scn)^{2}\sim ( {n}/{K_{n}} )^{3}$.
(This time, the dominant term is $\frac{(\bcn)^{2} F''(\bcn)}{F(\bcn)}$.)
\end{proof}

We denote by $ \Ts_{n}$ an alternating two-type BGW tree (with black root), with offspring distributions $\mubn$ and $\mucn$ conditioned on having $n-K_{n}$ black vertices and $K_{n}+1$ white vertices.  Recall from the beginning of Section~\ref{ss:coding} the definition of the path $(B_{i}(\Ts_{n}))_{1 \leq i \leq n-K_{n}}$.
We are aiming at a functional invariance theorem for a renormalized version of $(B_{i}(\Ts_{n}))_{1 \leq i \leq n-K_{n}}$.
In the sequel, $X$, $\Xbr$ and $\Xexc$ are as in Section~\ref{ssec:Levy}:
$X$ is the Lévy process with the specific characteristic exponent given in \eqref{eq:Levyc},
and $\Xbr$ and $\Xexc$ are the associated bridge and excursion processes.

Instead of working as usual with $\D([0,1],\R)$, we will work with $\D([-1,1],\R)$
by extending our function with value $0$ on $[-1,0)$.
The reason for that is that our limiting process $\Xexc$ almost surely
takes a positive value in $0$ (it ``starts with a jump''),
while $(\oB_{i}(\Ts_{n}))_{1 \leq i \leq n-K_{n}}$ stays small for a small amount of time.
Formally, we set $\Xexc_{t}=0$ for $t<0$, 
and $\overline{B}^{n}_{i}(\Ts_n)=0$ for $i \leq 0$ or $i>n-K_{n}$.

\begin{theorem}\label{thm:cvXexc} Assume that $ \tfrac{K_{n}}{\sqrt{n}} \rightarrow c>0$ as $n \rightarrow \infty$. The following assertions hold.
\begin{enumerate}
\item[(i)] The convergence \[  \left(  c \frac{\oB_{\lfloor u (n-K_{n}) \rfloor}(\Ts_{n})}{ n} : -1 \leq u \leq 1\right)  \quad \mathop{\longrightarrow}^{(d)}_{n \rightarrow \infty} \quad (\Xexc_{u}: -1 \leq u \leq 1)\]
holds in distribution.
\item[(ii)]
  Let $\widetilde{\mu}^{n}_{\bullet}$ be the probability distribution defined by
  $\widetilde{\mu}^{n}_{\bullet}(i)=\mubn(i-1)$ for $i \geq 1$
  and consider a random tree $\tilde{\Ts}_{n}$ with distribution ${\BGW}^{\tmub,\mub,\muc}$,
  conditioned on having $n-K_{n}$ black vertices and $K_{n}+1$ white vertices
  (the distribution ${\BGW}^{\tmub,\mub,\muc}$ is defined in \eqref{eq:BGWtilde}; informally this is an alternating two-type BGW tree with black root, offspring distributions $\mubn$ and $\mucn$ except the root which has offspring distribution $\widetilde{\mu}^{n}_{\bullet}$).
  Then the convergence 
  \[  \left(  c \frac{\oB_{\lfloor u (n-K_{n}) \rfloor}(\tilde{\Ts}_{n})}{ n} : -1 \leq u \leq 1\right)  \quad \mathop{\longrightarrow}^{(d)}_{n \rightarrow \infty} \quad (\Xexc_{u}: -1 \leq u \leq 1)\]
holds in distribution.
\end{enumerate}
\end{theorem}

The main difficulty is (i). Indeed, in the forthcoming Lemma~\ref{lem:root} (iii),
implies $\Ts_{n}$ and $\tTs _{n}$ can be coupled so that $\Ts_{n}=\tTs_{n}$
with probability tending to $1$ as $ n \rightarrow \infty$ (see e.g.~\cite{Lin92}).
Therefore, (ii) follows from (i).

From Lemma~\ref{lem:codeRW}, $(\oB_i(\Ts_{n}))_{i \ge 1}$ has the distribution of a conditioned random walk,
so that Theorem~\ref{thm:cvXexc} (i) is in fact an invariance principle for a conditioned random walk.
Thanks to Lemma~\ref{lem:vervaat}, 
we can consider a simpler conditioning of the form ``bridge'' instead of ``excursion''.

To state the invariance principle with this simpler conditioning, let us
recall some notation from Section~\ref{ss:coding}, adding an exponent $n$ to keep in mind that the
chosen offspring distributions do depend on $n$.
First, we  denote by $\Sbn$ and $\Scn$ the random walks with respective jump distributions given by $\mubn$ and $\mucn$.
Then we let $(H^{n}_{k},B^{n}_{k})_{k \geq 1}$ be a sequence of i.i.d.~random variables with distribution 
given by: for $i, j \geq 0$,  \[ \Prb{H^{n}=i,B^{n}=j}=\mubn(i) \, \Prb{\Scn_{i}=j}.\]
Finally, for $i \ge 1$, we set $\overline{H}^{n}_{i}=H^{n}_{1}+H^{n}_{2}+ \cdots+H^{n}_{i}$ and  $\oB^{n}_{i}=B^{n}_{1}+B^{n}_{2}+ \cdots+B^{n}_{i}-i$.

\begin{proposition}
\label{prop:cvbridge}
Assume that $ \tfrac{K_{n}}{\sqrt{n}} \rightarrow c>0$ as $n \rightarrow \infty$. Conditionally given the event
$ \{\overline{H}^{n}_{n-K_{n}}=K_{n}+1, \oB^{n}_{n-K_{n}}=-1 \}$, the convergence
\[  \left(   c \cdot  \frac{\oB^{n}_{\lfloor u (n-K_{n}) \rfloor}}{ n} \right)_{0 \leq u \leq 1}  \quad \mathop{\longrightarrow}^{(d)}_{n \rightarrow \infty} \quad \Xbr\]
holds in distribution.
\end{proposition}

Before proving Proposition~\ref{prop:cvbridge}, we explain how it implies Theorem~\ref{thm:cvXexc}. 

\begin{proof}[Proof of Theorem~\ref{thm:cvXexc} (i) using Proposition~\ref{prop:cvbridge}] Recall that $\Vc$ denotes the Vervaat transform that has been introduced in the end of Section~\ref{ssec:Levy}. From Lemma~\ref{lem:vervaat} (ii) and Lemma~\ref{lem:codeRW} (ii), it is elementary to get that
\[ \Vc \left(  \left(  \frac{\oB^{n}_{\lfloor u (n-K_{n}) \rfloor}}{ \sqrt{ (\scn)^{2} K_{n} }}  \right) _{{0 \leq u \leq 1}} \right) \quad \mathop{=}^{(d)}  \quad  \left( \frac{\oB_{\lfloor u (n-K_{n}) \rfloor}(\Ts_{n})}{ \sqrt{ (\scn)^{2} K_{n} }}  \right)_{ 0 \leq u \leq 1}.\]
Note that $\Xbr$ reaches its infimum at a unique time almost surely:
indeed, this is true for the unconditioned process $X$ on every fixed interval, and transfers to $\Xbr$ by the absolute continuity relation \eqref{eq:lawbridge}.
Therefore
$\Vc$ is almost surely continuous at $\Xbr$, and it follows that
\begin{equation}
\left( \frac{\oB_{\lfloor u (n-K_{n}) \rfloor}(\Ts_{n})}{ \sqrt{ (\scn)^{2} K_{n} }}  \right)_{ -1 \leq u \leq 1} \quad \mathop{\longrightarrow}_{n \rightarrow \infty}^{(d)} \quad (\Xexc_{u})_{-1 \leq u \leq 1}.
   \label{eq:cvHB}
 \end{equation}
 (Since $\Xbr$ jumps a.s. at its minimum, $\Vc$ is not a.s. continuous at $\Xbr$ when seen as a functional 
 $\D([0,1]) \to \D([0,1])$, but only as a functional $\D([0,1]) \to \D([-1,1])$;
 details are left to the reader.)
This completes the proof.
\end{proof}

\begin{remark} It is possible to strengthen the previous results by considering a bivariate Vervaat transform and to establish that
\[   \left( \frac{\overline{H}_{\lfloor u (n-K_{n}) \rfloor}(\Ts_{n})-u K_{n}}{ \sqrt{ (\sbn)^{2} n} }, \frac{\oB_{\lfloor u (n-K_{n}) \rfloor}(\Ts_{n})}{ \sqrt{ (\scn)^{2} K_{n} }}  \right)_{ -1 \leq u \leq 1}   \quad \mathop{\longrightarrow}_{n \rightarrow \infty}^{(d)} \quad  (W^{\mathrm{br}},\Xexc),\]
where $ W^{\mathrm{br}}$ is a Brownian bridge independent of $\Xexc$ (with $W^{\mathrm{br}}_{u}=0$ for $u<0$).
However, since we do not need this bivariate convergence for our initial goal (proving Theorem~\ref{thm:cvlam} (i)$_{c>0}$),
we focus here on the convergence of the second component only.
\end{remark}

The strategy of the proof of Proposition~\ref{prop:cvbridge} consists of two steps: first we establish a convergence of the unconditioned processes (Lemma~\ref{lem:cvnoncond}), and then we show Proposition~\ref{prop:cvbridge} by writing an analogue of the absolute continuity relation \eqref{eq:lawbridge} in the discrete setting and by passing to the limit. We will  heavily rely on the following uniform local limit estimates concerning the random walks $\Sbn$ and $\Scn$ (the proofs, very technical, are postponed to Section~\ref{ss:llt}).

\begin{lemma} \label{lem:ll1}
  Fix $0<u \leq 1$. For $n,N \geq 1$,  set  $\DbnN= \sqrt{ (\sbn)^{2}  N}$. Assume that $ \tfrac{K_{n}}{n} \rightarrow 0$ and that $K_{n} \rightarrow \infty$ as $n \rightarrow \infty$. 
  We have     
  \[  \sup_{un \leq N  \leq n} \sup_{k \in \Z  } \left|  \DbnN  \cdot \Pr{\Sbn_N=k} 
  -   p \left( \frac{k- N\tfrac{K_n+1}{n-K_n}}
  {  \DbnN  }\right)\right|  \quad \mathop{\longrightarrow}_{n \rightarrow \infty} \quad 0,\] 
  where $p(x)= \frac{1}{\sqrt{2 \pi}} e^{- \frac{x^{2}}{2}}$ is the standard Gaussian density.
  \label{lem:ll1Uniforme}
\end{lemma}

\begin{lemma}\label{lem:ll2}
 Fix $u>0$. Assume that $ \tfrac{K_{n}}{\sqrt{n}} \rightarrow c>0$ as $n \rightarrow \infty$.  We set $\Dcn=\sqrt{ (\scn)^{2} K_{n} } \sim n /c$. 
 Then it holds that
\[ \sup_{|j| \leq n^{3/8}} \sup_{k \in \Z} \left|  \Dcn \cdot \Pr{\Scn_{ uK_{n}  + j }=k} -  q_{u} \left( \frac{k}{  \Dcn }\right)\right|  \quad \mathop{\longrightarrow}_{n \rightarrow \infty} \quad 0,\]
where
\begin{equation}
\label{eq:defqu}q_{u}(x)=\left( \frac{u^{2} c^{3}}{2 \pi x^{3}} \right)^{1/2} \exp \left(  - \frac{c (x-uc)^{2} }{2x} \right) \mathbbm{1}_{x>0}.
\end{equation}\end{lemma}

We have chosen the exponent $3/8 \in (1/4,1/2)$ in view of specific future use,  but as the proof will show we can replace $3/8$ by  any positive value less than $1/2$.

   \begin{lemma}\label{lem:cvnoncond}
  Assume that $ \tfrac{K_{n}}{\sqrt{n}} \rightarrow c>0$ as $n \rightarrow \infty$.
  Consider sequences $(\oH^{n}_i)_{i \ge 1}$ and $(\oB^{n}_i)_{i \ge 1}$ as defined before Proposition~\ref{prop:cvbridge}.
  Without conditioning, the following convergence holds in distribution in $\mathbb{D}(\R_{+},\R^{2})$:
  \[ \left(  \frac{\oH^{n}_{ \lfloor u (n-K_{n}) \rfloor } - u K_{n} }{\sqrt{ (\sbn)^{2} n} },  \frac{\oB^{n}_{ \lfloor u (n-K_{n}) \rfloor }}{\sqrt{ (\scn)^{2} K_{n}}} \right)_{u \geq 0}  \quad \mathop{\longrightarrow}_{n \rightarrow \infty} \quad  (W_{u},X_{u})_{u \geq 0},\]
 where $W$ is a standard Brownian motion, independent from the Lévy process $X$.
  \end{lemma}
  
 Aiming for Proposition~\ref{prop:cvbridge}, the convergence of the first component seems to be superfluous.
 However,  its behavior has to be controlled in order to obtain a limit theorem for the second component
 in conditioned setting.

\begin{proof}
In virtue of \cite[Theorem 16.14]{Kal02}, it is enough to check that the one-dimensional convergence holds for $u=1$. Recall that $p$ is the density of the standard Gaussian distribution and that
\[q_{1}(x)=\left( \frac{ c^{3}}{2 \pi x^{3}} \right)^{1/2} e ^{  - \frac{c (x-c)^{2} }{2x} } \mathbbm{1}_{x>0}, \qquad d_{1}(x) =  \sqrt{ \frac{c^3}{2 \pi (x+c)^{3}}} \cdot e^{- \frac{c x^{2}}{2(x+c)}}  \mathbbm{1}_{x \geq -c}\]
are respectively the density appearing in \eqref{eq:defqu} and the density \eqref{eq:densiteXt} of $X_{1}$. In particular, note that $q_{1}(x+c)=d_{1}(x)$ for $x \in \R$.

For fixed $x,y \in \R$, by Lemma~\ref{lem:magique}, we have
\begin{multline*}
  \Prb{\overline{H}^{n}_{  n-K_{n} }= \lfloor x\sqrt{ (\sbn)^{2} n} +K_{n} \rfloor ,\oB^{n}_{  n-K_{n}}= \lfloor y \sqrt{ (\scn)^{2} K_{n}} \rfloor} \\
=\Pr{ \Sbn_{ n-K_{n} }=  \lfloor x\sqrt{ (\sbn)^{2} n} +K_{n} \rfloor} \Pr{\Scn_{ \lfloor x\sqrt{ (\sbn)^{2} n} +K_{n} \rfloor }= n-K_{n}  + \lfloor y \sqrt{ (\scn)^{2} K_{n}} \rfloor}.
\end{multline*}
Note that $\sqrt{ (\sbn)^{2} n} \sim \sqrt{c} \, n^{1/4}$. By the local limit  Lemmas~\ref{lem:ll1} and \ref{lem:ll2}, as $n \rightarrow \infty$, using the fact that $\sqrt{ (\sbn)^{2} n} \sim \sqrt{c}\cdot n^{1/4}=o(n^{3/8})$, this quantity is asymptotic to
\[ \frac{1}{ \sqrt{ (\sbn)^{2} n}  } \cdot   p \left(x\right)  \cdot  \frac{1}{\sqrt{ (\scn)^{2} K_{n}}} q_{1} \left( y+c \right)=\frac{1}{ \sqrt{ (\sbn)^{2} n}  } \cdot   p \left(x\right)  \cdot  \frac{1}{\sqrt{ (\scn)^{2} K_{n}}} d_{1}(y).\]
It is then standard (see e.g.~\cite[Theorem 7.8]{Bil68}) that this implies that the convergence
  \[ \left(  \frac{\oH^{n}_{   n-K_{n} } -  K_{n} }{\sqrt{ (\sbn)^{2} n} },  \frac{\oB^{n}_{   n-K_{n}  }}{\sqrt{ (\scn)^{2} K_{n}}} \right)  \quad \mathop{\longrightarrow}^{(d)}_{n \rightarrow \infty} \quad  (W_{1},X_{1})\]
  holds in distribution. This completes the proof.\end{proof}

We are now in position to prove Proposition~\ref{prop:cvbridge}.

\begin{proof}[Proof of Proposition~\ref{prop:cvbridge}] To simplify notation, set, for $0 \leq u \leq 1$,
\[\hH^{(n)}_{u}= \frac{\oH^{n}_{ \lfloor u (n-K_{n}) \rfloor } - u K_{n} }{\sqrt{ (\sbn)^{2} n} }, \qquad \hB^{(n)}_{u}= \frac{\oB^{n}_{\lfloor u (n-K_{n}) \rfloor}}{ \sqrt{ (\scn)^{2} K_{n} }} .\]
Recall that $\sqrt{ (\scn)^{2} K_{n}} \sim {n}/{c}$ and $\sqrt{ (\sbn)^{2} n} \sim \sqrt{c}\cdot n^{1/4}$. 
It is therefore enough to show that conditionally given the event
$ \mathcal{C}_{n} \coloneqq \{\oH^{n}_{n-K_{n}}=K_{n}+1, \oB^{n}_{n-K_{n}}=-1 \}$, 
the following (stronger) convergence in distribution holds jointly
\[  \left(  \hH^{(n)}_{u},  \hB^{(n)}_{u} \right)_{0 \leq u \leq 1}  \quad \mathop{\longrightarrow}^{(d)}_{n \rightarrow \infty} \quad (W^{\mathrm{br}},\Xbr),\]
where $ W^{\mathrm{br}}$ is a Brownian bridge independent of $\Xbr$.  Recall that $d_{u}$ denotes the density of $X_{u}$ and that $p_{u}$ denotes the density of standard Brownian motion at time $u$.

Fix $0<u<1$ and $F : \D([0,u],\R^{2}) \rightarrow \R$ a bounded continuous functional. For $\delta \geq 0$, we define $F_{\delta}$ by $F_{\delta}(f,g)=F(f,g) \mathbbm{1}_{|f(u)|< \frac{1}{\delta}}$. We fix $\delta>0$ and consider the quantity
\[A^{\delta}_{n}= \Es{ \left. F_{\delta} \left(    \hH^{(n)}_{s},  \hB^{(n)}_{s}: 0 \leq s \leq u \right)
\right| \mathcal{C}_{n}}.\]
We define the following quantity:
\begin{equation}
  \hspace{-3mm} \psi \left( \oH^{n}_{ \lfloor u (n-K_{n}) \rfloor },\oB^{n}_{ \lfloor u (n-K_{n}) \rfloor } \right) \coloneqq \frac{\phi_{n-K_{n}- \lfloor u (n-K_{n}) \rfloor} (K_{n}+1-\oH^{n}_{ \lfloor u (n-K_{n}) \rfloor } ,-1-\oB^{n}_{\lfloor u (n-K_{n}) \rfloor})}{\phi_{n-K_{n}} (K_{n}+1,-1) }
  \label{eq:def_psi}
\end{equation}
with $\phi_{i}(a,b)=\Prb{\oH^{n}_{i}=a,\oB^{n}_{i}=b}$.
By applying the Markov property at time $\lfloor u (n-K_{n}) \rfloor$, we have
\begin{equation}
  A^{\delta}_{n}= \Es{  F_{\delta} \left(   \hH^{(n)}_{s},  \hB^{(n)}_{s}: 0 \leq s \leq u \right) \psi \left( \oH^{n}_{ \lfloor u (n-K_{n}) \rfloor },\oB^{n}_{ \lfloor u (n-K_{n}) \rfloor } \right)  } .
  \label{eq:Adn}
\end{equation}
We observe that $n-K_{n}- \lfloor u (n-K_{n}) \rfloor= (1-u)n+\mathcal O(\sqrt{n})$ (with a deterministic $\mathcal O(\sqrt{n})$). 
Besides, on the event $ \big| \hH^{(n)}_{u} \big| < \tfrac{1}{\delta}$, 
we have $K_{n}-\oH^{n}_{ \lfloor u (n-K_{n}) \rfloor}=(1-u)K_{n}+J_{n}$ with $|J_{n}|= \mathcal{O}(n^{1/4})$.
Using Lemma~\ref{lem:magique}, we get that
\begin{multline}
 \phi_{n-K_{n}- \lfloor u (n-K_{n}) \rfloor} (K_{n}+1-\oH^{n}_{ \lfloor u (n-K_{n}) \rfloor } ,-1-\oB^{n}_{\lfloor u (n-K_{n}) \rfloor})\\
   = \tPr{\Sbn_{(1-u)n+\mathcal O(\sqrt{n})}=K_{n}+1-\oH^{n}_{ \lfloor u (n-K_{n}) \rfloor}}  \\
   \cdot\tPr{\Scn_{(1-u) K_{n}+ J_{n}}= (1-u)n+\mathcal O(\sqrt{n})-\oB^{n}_{\lfloor u (n-K_{n}) \rfloor})},
   \label{eq:Tech4}
\end{multline}
where $\tilde{\mathbb{P}}$ denotes the conditional probability given $\oH^{n}_{ \lfloor u (n-K_{n} \rfloor}$ and $\oB^{n}_{\lfloor u (n-K_{n})\rfloor}$
(the randomness in Lemma~\ref{lem:magique} comes from $\Sbn$ and $\Scn$).

Conditionally given $\oH^{n}_{ \lfloor u (n-K_{n}) \rfloor }$ and $\oB^{n}_{ \lfloor u (n-K_{n}) \rfloor }$, 
Lemma~\ref{lem:ll1} gives us
\begin{multline*}
\tPr{\Sbn_{(1-u)n+\mathcal O(\sqrt{n})}=K_{n}-\oH^{n}_{ \lfloor u (n-K_{n}) \rfloor}}\\
  = \frac{1}{\sqrt{ (\sbn)^{2}  n} } \frac{1}{\sqrt{1-u}} p \left(  -   \frac{\oH^{n}_{ \lfloor u (n-K_{n}) \rfloor } - u K_{n}}{ \sqrt{1-u} \cdot \sqrt{ (\sbn)^{2} n} }  + \eta^{\bullet}_{n}\right) +  \frac{\epsilon^{\bullet}_{n}}{\sqrt{ (\sbn)^{2}  n}}.
\end{multline*}

Here, $\eta^{\bullet}_{n}$ and $\epsilon^{\bullet}_{n}$ are random sequences 
(they depend on $\oH^{n}_{ \lfloor u (n-K_{n}) \rfloor}$),
and tend in $L^\infty$ norm to $0$
on the event $ \big| \hH^{(n)}_{u} \big| < \tfrac{1}{\delta}$
(for $\epsilon^{\bullet}_{n}$, this is because of the uniformity in $k$ in Lemma~\ref{lem:ll1}).
Since $p$ is uniformly continuous
and since $\tfrac{1}{\sqrt{1-u}}p(\tfrac{-x}{\sqrt{1-u}})=p_{1-u}(-x)$,
we have
\begin{multline}
  \tPr{\Sbn_{(1-u)n+\mathcal O(\sqrt{n})}=K_{n}-\oH^{n}_{ \lfloor u (n-K_{n}) \rfloor}}\\
=\frac{1}{\sqrt{ (\sbn)^{2}  n} }p_{1-u}\left( -\frac{\oH^{n}_{ \lfloor u (n-K_{n}) \rfloor } 
- u K_{n}}{ \sqrt{ (\sbn)^{2} n} }\right) + \frac{\epsilon^{'\bullet}_{n}}{\sqrt{ (\sbn)^{2}  n}},
\label{eq:PSbn}
\end{multline}
for some other random sequence $\epsilon^{'\bullet}_{n}$ tending to $0$ in $L^\infty$ norm.

Similarly, using Lemma~\ref{lem:ll2} and the fact that  $q_{1-u}(x+c(1-u))=d_{1-u}(x)$, 
on the event $ \big| \hH^{(n)}_{u} \big| < \tfrac{1}{\delta}$,
we have
\begin{multline}
\tPr{\Scn_{(1-u) K_{n}+ J_{n}}= (1-u)n+\mathcal O(\sqrt{n})-\oB^{n}_{\lfloor u (n-K_{n}) \rfloor})}\\
   =  \frac{1}{ \sqrt{ (\scn)^{2} K_{n} } } d_{1-u} \left( \frac{\oB^{n}_{\lfloor u (n-K_{n}) \rfloor}}{\sqrt{ (\scn)^{2} K_{n} }}+ \eta^{\circ}_{n}\right)+  \frac{\epsilon^{\circ}_{n}}{\sqrt{ (\scn)^{2}  K_n}},
   \label{eq:PScn}
\end{multline}
conditionally given $\oH^{n}_{ \lfloor u (n-K_{n}) \rfloor }$ and $\oB^{n}_{ \lfloor u (n-K_{n}) \rfloor }$. 
Here $\eta^{\circ}_{n}$ is a deterministic sequence tending to $0$,
while $\epsilon^{\circ}_{n}$ is random and tends in $L^\infty$ norm to $0$
(because of the uniformity in $j$ in Lemma~\ref{lem:ll2}).
Since $d_{1-u}$ is uniformly continuous, the error $\eta^{\circ}_{n}$ in the argument
can be absorbed in the $\epsilon^{\circ}_{n}$ error.

In the proof of Lemma~\ref{lem:cvnoncond}, we already saw that
\begin{multline}
  \phi_{n-K_{n}} (K_{n}+1,-1) =\Pr{\oH^{n}_{n-K_{n}}=K_{n}+1, \oB^{n}_{n-K_{n}}=-1} \\
  \mathop{\sim}_{n \rightarrow \infty}   \frac{1}{ \sqrt{ (\sbn)^{2} n}  } \cdot   p_{1}  (0) \cdot  \frac{1}{\sqrt{ (\scn)^{2} K_{n}}} d_{1}(0).
  \label{eq:Cond_Proba}
\end{multline}
By combining Eqs.~\eqref{eq:Cond_Proba}, \eqref{eq:PSbn}, \eqref{eq:PScn}, \eqref{eq:Tech4} and \eqref{eq:def_psi}, we get
\begin{equation}
   \psi \left( \oH^{n}_{ \lfloor u (n-K_{n}) \rfloor },\oB^{n}_{ \lfloor u (n-K_{n}) \rfloor }\right)
=  \frac{p_{1-u}(\hH^{(n)}_{u})}{p_{1}(0)} \frac{d_{1-u}(\hB^{(n)}_{u})}{d_{1}(0)} +\, \text{error},
\label{eq:limit_psi}
\end{equation}
where the (random) error tends to $0$ in $L^\infty$ norm as $n \to \infty$ on the event $ \big| \hH^{(n)}_{u} \big| < \tfrac{1}{\delta}$.
Therefore, using \eqref{eq:Adn}, the dominated convergence theorem and Lemma~\ref{lem:cvnoncond}, we have:
\begin{equation}
\label{eq:cvu0}A^{\delta}_{n}  \quad \mathop{\longrightarrow}_{n \rightarrow \infty} \quad  \Es{F \left(   W_{s},X_{s} : 0 \leq s \leq u \right)  \mathbbm{1}_{|W_{u}|<1/\delta} \frac{p_{1-u}(-W_{u})}{p_{1}(0)} \frac{d_{1-u}(-X_{u})}{d_{1}(0)} }
\end{equation}
This last quantity is equal to $\Esb{F \big(  W^{\textrm{br}}_{s},X^{\textrm{br}}_{s} : 0 \leq s \leq u \big)  \mathbbm{1}_{|W^{\textrm{br}}_{u}|<1/\delta}}$ by \eqref{eq:lawbridge}.

In remains to take $\delta \rightarrow 0$. Specifically, 
taking for $F$ the constant functional always equal to $1$ in \eqref{eq:cvu0}, we get
\begin{equation}
\label{eq:cvE}\Pr{ |\hH^{(n)}_{u}| \geq 1/\delta \big|  \mathcal{C}_{n} }  \quad \mathop{\longrightarrow}_{n \rightarrow \infty} \quad  \Pr{|W^{br}_{u}| \geq 1/\delta}.
\end{equation}
Now fix $\epsilon>0$ and choose $\delta>0$ such that $  \Prb{|W^{br}_{u}| \geq 1/\delta} \leq \epsilon/ \| F \|_{\infty}$. Then, by \eqref{eq:cvE},  for every $n$ sufficiently large  $\Prb{  |\hH^{(n)}_{u}| \geq  1/\delta \big| \mathcal{C}_{n}} \leq 2\epsilon / \| F \|_{\infty}$, so that $\big| A_{n}^{\delta} - A_{n}^{0} \big| \leq 2 \epsilon$. Hence, for every $n$ sufficiently large,
\begin{align*}
  & \hspace{-3mm} \left| \Es{F \left(   W^{\textrm{br}}_{s},X^{\textrm{br}}_{s} : 0 \leq s \leq u \right) }- \Es{ \left. F \left(  \hH^{n}_{s},\hB^{n}_{s} : 0 \leq s \leq u \right) \right| \mathcal{C}_{n}}  \right|\\
  & \hspace{-2mm} \leq \| F \|_{\infty} \Pr{|W^{br}_{u}| \geq 1/\delta} + 
\left| \Es{F \left(   W^{\textrm{br}}_{s},X^{\textrm{br}}_{s} : 0 \leq s \leq u \right)\mathbbm{1}_{|W^{br}_{u}|<1/\delta}}  -A_{n}^{\delta}  \right| +  \left| A_{n}^{\delta} - A_{n}^{0} \right | \\
& \hspace{-1mm}\leq 4 \epsilon.
\end{align*}
Hence
\begin{equation}
\label{eq:cvu}\Es{ \left. F \left(  \hH^{(n)}_{s},  \hB^{(n)}_{s} : 0 \leq s \leq u \right) \right|  \mathcal{C}_{n} }   \mathop{\longrightarrow}_{n \rightarrow \infty} \quad \Es{F \left(  W^{\textrm{br}}_{s},X^{\textrm{br}}_{s} : 0 \leq s \leq u \right) }.
\end{equation}

A standard time-reversal argument, based on the fact that conditionally given $ \mathcal{C}_{n}$, the two paths  $(\oH^{n}_{i},\oB^{n}_{i})_{1 \leq i \leq n-K_{n}}$ and $(\oH^{n}_{n-K_{n}}-\oH^{n}_{n-K_{n}-i},\oB^{n}_{n-K_{n}}-\oB^{n}_{n-K_{n}-i})_{1 \leq i \leq n-K_{n}}$ have the same distribution, shows that the sequence of random variables $  (  \hH^{n}_{u}, \hB^{n}_{u} )_{0 \leq u \leq 1}$ is tight in $\D([0,1],\R^{2})$. The convergence \eqref{eq:cvu} then shows that there is only one possible distributional limit, and this completes the proof.
\end{proof}

\subsection{On black vertices with many descendants}
\label{ss:descendants}
We conclude Section~\ref{sec:limitheorems} by several results, concerning black vertices with many descendants, which build upon the machinery that has just been developed. The first main result (Lemma~\ref{lem:root}) concerns the modified alternating two-type BGW root which appears in Theorem~\ref{thm:cvXexc} (ii), while the two others (Lemmas~\ref{lem:AtMost2} and \ref{lem:two}) will be useful in the proof of Theorem~\ref{thm:cvlam} (i)$_{c>0}$ (convergence to $\mathbf{L}_{c}$).

Throughout Section~\ref{ss:descendants}, we assume that $ \tfrac{K_{n}}{\sqrt{n}} \rightarrow c>0$ as $n \rightarrow \infty$
and the offspring distributions $\mubn$ and $\mucn$ are the ones given by Lemma~\ref{lem:exist}.
We use the notation of Theorem~\ref{thm:cvXexc}, namely
\begin{itemize}
  \item $\Ts_{n}$ is  an alternating two-type BGW tree (with black root), with offspring distributions $\mub^{n}$ and $\muc^{n}$,
    conditioned on having $n-K_{n}$ black vertices and $K_{n}+1$ white vertices;
  \item  and  $\tilde{\Ts}_{n}$ is defined as $\Ts_{n}$,
     except that the root has an offspring distribution $\widetilde{\mu}^{n}_{\bullet}$
     given by $\widetilde{\mu}^{n}_{\bullet}(i)=\mubn(i-1)$ for $i \geq 1$
     (the offspring distribution of other vertices and the conditioning are unchanged).
\end{itemize} 

\paragraph*{Children of the root.}
We start by comparing the number of children of the root in these two models:
\begin{lemma}\label{lem:root}
The following properties are satisfied:
\begin{enumerate}
\item[(i)] with probability tending to $1$ as $n \rightarrow \infty$, the root of $\Ts_{n}$ has only one child;
\item[(ii)] with probability tending to $1$ as $n \rightarrow \infty$, the root of $\tilde{\Ts}_{n}$ has only one child;
\item[(iii)] we have $d_{\mathrm{TV}}(\Ts_{n},\tilde{\Ts}_{n}) \rightarrow 0$ as $n \rightarrow \infty$, where $d_{\mathrm{TV}}$ denotes the total variation distance.
\end{enumerate}
\end{lemma}

\begin{proof} 
Denote by $\Ts^{n}$ (resp.~$\tTs^{n}$) the unconditioned version of $\Ts_{n}$ (resp.~$\tTs_{n}$). We keep the superscript $n$ to emphasize that the offspring distributions depend on $n$. To simplify notation, for $k \geq 0$, set 
$$P^{n}_{k}= \Prb{\textrm{the root of }\Ts^{n} \textrm{ has } k \textrm{ children}, \abs{\bullet_{\Ts^{n}}}=n-K_{n}, \abs{\circ_{\Ts^{n}}}=K_{n}+1}$$
and  $P^{n}= \Prb{ \abs{\bullet_{\Ts^{n}}}=n-K_{n}, \abs{\circ_{\Ts^{n}}}=K_{n}+1}$.
Similarly, define $\tilde{P}^{n}_{k}$ and $\tilde{P}^{n}$ when  $\Ts^{n}$ is replaced with $\tTs^{n}$. Clearly, $P^{n}_{0}= \tilde{P}^{n}_{0}=0$.
Finally, set
\[A_{n,k}= \frac{1}{K_{n}+1} \Pr{\Sbn_{n-K_{n}-1}=K_{n}+1-k} \Pr{\Scn_{K_{n}+1}=n-K_{n}-1}.\]
By decomposing a tree with a black root having $k$ children into a forest of $k$ trees with white roots and using \eqref{eq:forestc}, we have
$P^{n}_{k}=  k\, \mubn(k) \cdot  A_{n,k}$ and $\tilde{P}^{n}_{k}= k\, \mubn(k-1) \cdot A_{n,k}$.

In the following, $C$ will denote a universal constant whose value might change from line to line.
By Lemmas~\ref{lem:exist}, \ref{lem:ll1} and \ref{lem:ll2}, we have the following two asymptotic estimates for $A_{n,k}$:
\begin{itemize}
  \item for fixed $k \geq 1$, $ A_{n,k} \sim { p(0) \, q_{1}(c)}\cdot \tfrac{1}{\sqrt{K_{n}} n^{3/2}}$ as $n \rightarrow \infty$;
  \item we have $A_{n,k} \leq  \frac{C}{\sqrt{K_{n}} n^{3/2}}$ uniformly for all $n \geq 1$ and $k \geq 1$.
\end{itemize} 
When $K_{n} \rightarrow \infty$ and $\tfrac{K_{n}}{n} \rightarrow 0$, we saw in the proof of Lemma~\ref{lem:exist} that $\abn  \rightarrow 1$ and that $\bbn \sim \frac{K_{n}}{n}$ as $n \rightarrow \infty$. Hence,
\[\mubn(0)  \quad \mathop{\longrightarrow}_{n \rightarrow \infty} \quad 1, \qquad  \mubn(1)  \quad \mathop{\sim}_{n \rightarrow \infty} \quad  \frac{K_{n}}{n},\]
and for $n$ large enough,
\[\mubn(k) \leq  \frac{K_{n}}{{n}} \frac{(k+1)^{k-1}}{k!} \, \textrm{ for } k \geq 1, \quad \mubn(k) \leq  \left(  \frac{K_{n}}{{n}}  \right)^{2} \frac{(k+1)^{k-1}}{k!} \, \textrm{ for } k \geq 2.\]
Consequently
\[ P^{n}_{1}  \quad \mathop{\sim}_{n \rightarrow \infty} \quad   p(0) \,q_{1}(c) \cdot \frac{\sqrt{K_{n}}}{n^{5/2}}, \qquad \sum_{k=2}^{\infty} P^{n}_{k} \leq  C \frac{K_{n}}{n} \sum_{k=2}^{\infty}  \frac{\sqrt{K_{n}}}{n^{5/2}}  \frac{(k+1)^{k-1}}{k!}= o\big( P^{n}_{1}\big)\]
since $\tfrac{K_{n}}{n} \rightarrow 0$.
Thus $P^{n} \sim P^{n}_{1}$ as $n \rightarrow \infty$, and (i) follows.

Similarly, for (ii), we now have
\[ \tilde{P}^{n}_{1}  \ \mathop{\sim}_{n \rightarrow \infty} \
{ p(0)\, q_{1}(c)}\ \frac{1}{\sqrt{K_{n}} n^{3/2}}, \qquad \sum_{k=2}^{\infty} \tilde{P}^{n}_{k} \leq C  \frac{K_{n}}{n}\sum_{k=1}^{\infty}\frac{1}{\sqrt{K_{n}} n^{3/2}}  \frac{(k+1)^{k-1}}{k!}= o\big( \tilde{P}^{n}_{1} \big).\]  
Thus $\tilde{P}^{n} \sim \tilde{P}^{n}_{1}$ as $n \rightarrow \infty$, and (ii) follows.

For (iii), recalling that $\mathbb{A}$ denotes the set of all finite plane rooted trees, we show that  
\begin{equation}
\label{eq:dtv}\sup_{A \subset \mathbb{A}} \left| \Pr{\Ts_{n} \in A} -  \Pr{\tilde{\Ts}_{n} \in A}\right|   \quad \mathop{\longrightarrow}_{n \rightarrow \infty} \quad  0.
\end{equation}
By (i) and (ii), it is enough to establish this convergence for subsets $A \subset \mathbb{A}$ such that the roots of all trees in $A$ have $1$ child.
If $\tau \in A$, denote by $[\tau]^{1}$ the subtree (with white root) grafted on the child of the root of $\tau$.
Denote by $\Ts^{\circ,n}$  an (unconditioned) alternating BGW tree with white root (and offspring distributions
$\mubn$ and $\mucn$). Then we have
\begin{align*}
  \Pr{ \Ts_{n}=\tau}&= \frac{1}{P^{n}} \Pr{ \Ts^{n}=\tau} = \frac{1}{P^{n}} \cdot \mubn(1) \cdot \Pr{ \Ts^{\circ,n}=[\tau]^{1}},\\
  \Pr{ \tTs_{n}=\tau}&= \frac{1}{\tilde{P}^{n}} \Pr{ \tTs^{n}=\tau} = \frac{1}{\tilde{P}^{n}} \cdot \mubn(0) \cdot \Pr{ \Ts^{\circ,n}=[\tau]^{1} }.
\end{align*}
Thus, setting $[A]^{1}= \{[\tau]^{1} : \tau \in A\}$, we get  
\begin{align*}
&\left| \Pr{\Ts_{n} \in A} -  \Pr{\tilde{\Ts}_{n} \in A}\right| \\
& \qquad  \qquad = \left| \frac{1}{P^{n}} \cdot \mubn(1)- \frac{1}{\tilde{P}^{n}} \cdot \mubn(0) \right| \Pr{ \Ts^{\circ,n}  \in [A]^{1}}\\
& \qquad \qquad \leq \left| \frac{1}{P^{n}} \cdot \mubn(1)- \frac{1}{\tilde{P}^{n}} \cdot \mubn(0) \right|  \Pr{\abs{\bullet_{\Ts^{\circ,n}}}=n-K_{n}-1, \abs{\circ_{\Ts^{\circ,n}}}=K_{n}+1} \\
& \qquad \qquad = \left| \frac{1}{P^{n}} \cdot \mubn(1)- \frac{1}{\tilde{P}^{n}} \cdot \mubn(0) \right|  A_{n,1}.
\end{align*}
Since $\tfrac{\mubn(1) A_{n,1}}{P^{n}} \rightarrow 1$ and $\tfrac{\mubn(0) A_{n,1}}{\tilde{P}^{n}} \rightarrow 1$ as $n \rightarrow \infty$, this completes the proof.
\end{proof}

\paragraph*{Descendants of black vertices.}
We now focus on the model $\Ts_n$ and discuss degrees of black vertices.

\begin{lemma}
  With probability tending to $1$ as $n \rightarrow \infty$, the tree $\Ts_{n}$ does not contain any black vertex with three or more children.
  \label{lem:AtMost2}
\end{lemma}

\begin{proof}
Denote by $\Delta^{\bullet}(\Ts_{n})$ the maximum number of children of a black vertex in $\Ts_{n}$.
The idea is to reformulate the assertion in terms of random walks. 
We keep the notation  $(H^{n}_k,B^n_k)_{k \geq 1}$, $\overline{H}^{n}$ and $\overline{B}^{n}$ which has been introduced in Section~\ref{ss:coding}. 
To simplify notation, let us use $ \mathcal{C}_n$ and  $ \mathcal{C}^+_{n}$ for the events
\begin{align*}
  \mathcal{C}_{n}&= \{\oH^{n}_{n-K_{n}}=K_{n}+1,\oB^{n}_{n-K_{n}}=-1\};  \\
  \mathcal{C}^+_{n}&= \{\oH^{n}_{n-K_{n}}=K_{n}+1,\oB^{n}_{n-K_{n}}=-1, \oB^{n}_{i} \geq 0 \textrm{ for every } 1 \leq i <n-K_{n}\}.
\end{align*}

Since the maximum element of a vector is invariant under cyclic permutations, we may combine  Lemmas~\ref{lem:vervaat} and \ref{lem:codeRW} to get that 
\begin{align*}
  \Delta^{\bullet}(\Ts_{n}) \ & \mathop{=}^{(d)} \  \max(H^{n}_{1}, \ldots, H^{n}_{n-K_{n}}) \textrm{ under } \Pr{\ \cdot  \ \big|  \mathcal{C}^+_{n} } \\
  & \mathop{=}^{(d)} \  \max(H^{n}_{1}, \ldots, H^{n}_{n-K_{n}}) \textrm{ under } \Pr{\ \cdot  \ \big|  \mathcal{C}_{n} }.
\end{align*}
Therefore, the lemma asserts that
\[\lim_{n \to \infty} \Pr{\max(H^{n}_{1}, \ldots, H^{n}_{n-K_{n}}) \geq 3 \big|  \mathcal{C}_{n} } = 0.\]

The proof strategy consists in showing the same statement without conditioning 
and then using an absolute continuity argument as in the proof of Proposition~\ref{prop:cvbridge}.

The unconditioned statement relies on a simple estimate on the tail of the black vertex offspring distribution $\mubn$.
 The proof of Lemma~\ref{lem:exist} shows that we have $|a_\bullet^n| \le 2$ and $b_\bullet^n \le \tfrac{2 c}{\sqrt{n}} \le 1/2$ for $n$ large enough. Therefore
  \begin{equation}
    \mubn([3,\infty))\le 2 \left(\frac{2 c}{\sqrt{n}} \right)^3  \sum_{i \ge 3} \frac{(i+1)^{i-1}}{i!} \, \left(\frac{1}{2}\right)^{i-3}=\mathcal O(n^{-3/2}).
    \label{eq:tail_mub}
  \end{equation}
  As a consequence,
  \begin{equation}
    \Pr{\max(H^{n}_{1}, \ldots, H^{n}_{n-K_{n}}) \geq 3} \leq n  \mub\big([3,\infty)\big)  \quad \mathop{\longrightarrow}_{n \rightarrow \infty} \quad 0.
    \label{eq:noDeg3_Unconditioned}
  \end{equation}

  In order to use absolute continuity, 
  we first note that, by cyclic exchangeability,
  \begin{align*}
    \Pr{\max(H^{n}_{1}, \ldots, H^{n}_{n-K_{n}})  \geq 3 \big| \mathcal{C}_{n} } &\leq
\Prb{\max(H^{n}_{1}, \ldots, H^{n}_{\lfloor n/2 \rfloor})  \geq 3 \big| \mathcal{C}_{n} }\\
&\qquad + \Prb{\max(H^{n}_{\lfloor n/2 \rfloor+1}, \ldots, H^{n}_{n-K_{n}})  \geq 3 \big| \mathcal{C}_{n} }\\
&\le 2\Pr{\max(H^{n}_{1}, \ldots, H^{n}_{\lfloor n/2 \rfloor})  \geq 3 \big| \mathcal{C}_{n} }.
\end{align*}
To show that this last probability tends to $0$,
we use the same arguments as in the proof  of Proposition~\ref{prop:cvbridge}.
If $\delta >0$ and $G^{\delta}_{n}$  is the event $  \big\{\big| \tfrac{\oH^{n}_{ \lfloor  n /2 \rfloor } -  K_{n}/2 }{\sqrt{ (\sbn)^{2} n} } \big| < \tfrac{1}{\delta} \big\}$, then by \eqref{eq:cvE},
we have $\Prb{G^{\delta}_{n} | \mathcal{C}_{n}  } \rightarrow \Prb{|W^{\mathrm{br}}_{1/2}| <1/\delta}$.
For fixed $\epsilon>0$, we can therefore choose $\delta>0$ such that $\Prb{G^{\delta}_{n} \big| \mathcal{C}_{n} } \geq 1- \epsilon$ for $n$ sufficiently large. If $ \EEE_{n}$ is the event $\{\max(H^{n}_{1}, \ldots, H^{n}_{\lfloor n/2 \rfloor})  \geq 3\}$,  then
\[\Pr{ \EEE_{n} \big | \mathcal{C}_{n}  } \le \epsilon + \Pr{ \EEE_{n} \cap G^\delta_n \big| \mathcal{C}_{n} }
\leq \epsilon+ \Es{ \mathbbm{1}_{\EEE_n}\, \mathbbm{1}_{G^\delta_n} \,  \psi \left( \oH^{n}_{ \lfloor n/2 \rfloor },\oB^{n}_{ \lfloor    n/2 \rfloor } \right)},\]
with $\psi$ is defined in \eqref{eq:def_psi}.
But, from \eqref{eq:limit_psi}, there exists a constant $C>0$ (which may depend on $\epsilon$) such that
$\mathbbm{1}_{G^\delta_n} \psi \left( \oH^{n}_{ \lfloor n/2 \rfloor },\oB^{n}_{ \lfloor    n/2 \rfloor } \right) \leq C$ for $n$ sufficiently large. 
On the other hand, from \eqref{eq:noDeg3_Unconditioned}, 
$\mathbb{P}(\EEE_n) \le \eps/C$ for large $n$.
Hence, 
\[\Pr{\max(H_{1}, \ldots, H_{n-K_{n}})  \geq 3 \big| \mathcal{C}_{n} } \leq  2\epsilon + 2C \Pr{\max(H_{1}, \ldots, H_{\lfloor n/2 \rfloor}) \geq 3} \leq 3\epsilon.\]
This completes the proof.
\end{proof}

The tree $\Ts_{n}$ may, however, contain black vertices with $2$ children.
For such vertices, we will need a further control on their number of descendants.
Before stating and proving such a result, we give a concentration inequality for the lower tails sums of i.i.d. random variables distributed as $\mubn$.

\begin{lemma}
  Let $u, \delta$ be positive constants such that $0<u<1$ and $0<\delta<u c$. There exists a constant $A>0$ (depending only on $u$, $\delta$ and $c$)
  such that, for $n$ sufficiently large,
  \[\Pr{\Sbn_{\lfloor u\, n \rfloor} \le \delta \sqrt{n}}
  \le e^{-A\sqrt{n}}.\]
  \label{lem:Concentration}
\end{lemma}
\begin{proof}
This is actually a straightforward application of the one-sided version of  Hoeffding's inequality.
Specifically, let $(Y_{i})_{i \geq 1}$ be i.i.d.~random variable with distribution $\mubn$. Observe that
\[\Pr{\Sbn_{\lfloor u\, n \rfloor} \le \delta \sqrt{n}}= \Pr{ \frac{1}{\lfloor u\, n \rfloor} \sum_{i=1}^{\lfloor u\, n \rfloor} \left(  \frac{K_{n}}{n}-Y_{i} \right)  \geq \tfrac{K_n}{n}-\tfrac{\delta \, \sqrt{n}}{\lfloor u\, n\rfloor}}.\]
We then apply \cite[Theorem 3]{Hoeffding}, and more precisely \cite[Eq.~(2.12)]{Hoeffding} (taking, in the notation of \cite{Hoeffding}, $\lfloor u\, n \rfloor$ instead of $n$, $X_{i}= \tfrac{K_{n}}{n}-Y_{i}$, $b= \tfrac{K_{n}}{n}$, $t=\tfrac{K_n}{n}-\tfrac{\delta \, \sqrt{n}}{\lfloor u\, n\rfloor}$), and obtain
\[\Pr{\Sbn_{\lfloor u\, n \rfloor} \le \delta \sqrt{n}}
  \leq e^{-\tau_{n} h_{n}(\lambda_{n})}\]
  with $h_{n}(\lambda)= \big( \tfrac{1}{\lambda} +1\big)  \ln(1+\lambda)-1$, $\tau_{n}= \tfrac{\lfloor u\, n \rfloor t}{b}$ and $\lambda_{n}= \tfrac{bt}{(\sbn)^{2}}$. The following asymptotic estimates are readily derived
  (recall from Lemma~\ref{lem:exist} that $(\sbn)^{2} \sim \tfrac{K_n}{n}$):
  \[t \sim \frac{u \, c-\delta}{u} n^{-1/2},\ \tau_{n} \sim  \frac{u \, c-\delta}{c} n,\  \lambda_{n} \sim \frac{uc-\delta}{u} n^{-1/2}\text{ and }h_{n}(\lambda_{n}) \sim \frac{\lambda_{n}}{2}.\]
  Thus, for $n$ sufficiently large,
\[\Pr{\Sbn_{\lfloor u\, n \rfloor} \le \delta \sqrt{n}} \leq  e^{- \frac{(uc-\delta)^{2}}{3uc} \sqrt{n}}.\]
This completes the proof.
\end{proof}

\begin{lemma}\label{lem:two}
Fix $\epsilon>0$. With probability tending to $1$ as $n \rightarrow \infty$, the tree $\Ts_{n}$ does not contain a black 
  vertex $v$ that has two children, both with at least $\eps \, n$ descendants (counting both white and black vertices).  \label{lem:PasDeuxGrandsSousArbres}
\end{lemma}
\begin{proof}
As before, we consider the alternating two-type $\BGW$ tree $\Ts^{n}$ (with black root) and offspring distribution
  $\mubn$ and $\mucn$, so that $\Ts_{n}$ is distributed according to $\Ts^{n}$ conditioned on 
  \hbox{$\{ \abs{\bullet_{\Ts^{n}}}=n-K_n, \abs{\circ_{\Ts^{n}}} =K_n+1\}$}.   
  If $\tau$ is a finite tree and $v \in \bullet_{\tau}$ is a black vertex, we denote by 
  $Q_\eps(\tau,v)$ the property that $v$ has two children both with at least $\eps \, |\tau|$ descendants. Then
  
  \[    \Pr{\exists v \in \bullet_{\Ts_{n}}: \, Q_\eps(\Ts_{n},v)} \le \frac{A_{n}}{\Pr{\abs{\bullet_{\Ts}}=n-K_n, \abs{\circ_{\Ts}} =K_n}}\]
  with \[A_{n} \coloneqq \sum_{\substack{\tau: \abs{\bullet_{\tau}}=n-K_n, \abs{\circ_{\tau}} =K_n+1  \\v \in \bullet_{\tau} \textrm{ with } Q_\eps(\tau,v)}} \Pr{\Ts^{n}=\tau}.\]
 
  A couple $(\tau,v)$ with $\abs{\bullet_{\tau}}=n-K_n, \abs{\circ_{\tau}} =K_n+1$ and $v \in \bullet_{\tau} $ with $Q_\eps(\tau,v)$
 can be decomposed into three trees $\tau_{1}$, $\tau_{2}$ and $\tau_{3}$ such that:  
  \begin{itemize}
    \item $\tau_{1}$ and $\tau_{2}$ are respectively the subtrees (with white root) grafted to the white children of $v$;
    \item $(\tau_{3},v)$ is the tree 
      obtained from $\tau$ by replacing the whole subtree grafted in $v$ with a marked leaf.
  \end{itemize}

 \begin{figure}[t]
 \begin{center}
 \includegraphics[scale=0.5]{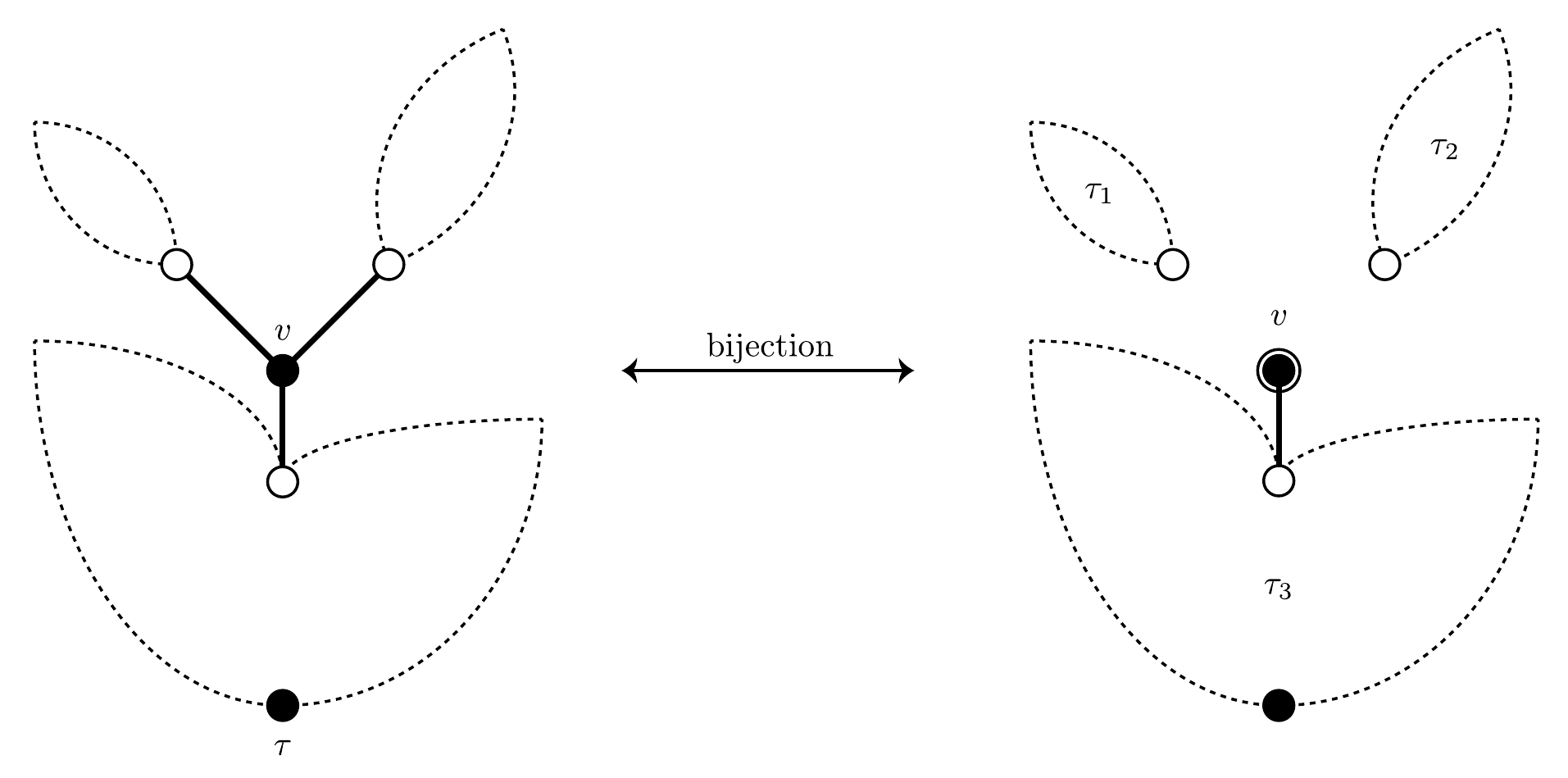}
 \caption{\label{fig:decomposition}Illustration of the decomposition $(\tau,v) \leftrightarrow  (\tau_{1},\tau_{2},(\tau_{3},v))$.}
 \end{center}
 \end{figure}
  This decomposition defines a bijection (see Figure~\ref{fig:decomposition} for an illustration):
  \begin{multline}
    \label{eq:BijTTi}
    \left\{ (\tau,v): Q_\eps(\tau),v), \abs{\bullet_{\tau}}=n-K_n, \abs{\circ_{\tau}} =K_n+1 \right\} \\
    \simeq \left\{ (\tau_{1},\tau_{2},(\tau_{3},v)):\begin{array}{l} \text{ for $i=1,2$, we have }
    \abs{\bullet_{\tau_{i}}}+ \abs{\circ_{\tau_{i}}} \ge \eps n,\\
    \text{ and }
    \abs{\bullet_{\tau_{1}}}+\abs{\bullet_{\tau_{2}}}+\abs{\bullet_{\tau_{3}}}=n-K_n,\\
    \qquad \ \abs{\circ_{\tau_{1}}}+\abs{\circ_{\tau_{2}}}+\abs{\circ_{\tau_{3}}}=K_n+1.
  \end{array}\right\}.
  \end{multline}
  In the right-hand side,
  $\tau_{1}$ and $\tau_{2}$ have white roots, while $(\tau_{3},v)$ has a black root and a marked black  leaf.
  Moreover, 
  if $(\tau,v)$ is mapped to $(\tau_{1},\tau_{2},(\tau_{3},v))$, we have
  \begin{equation}
  \Pr{\Ts^{n}=\tau}=\frac{\mubn(2)}{\mubn(0)} \, \Pr{\Ts^{\circ,n}_{1}=\tau_{1}} \, \Pr{\Ts^{\circ,n}_{2}=\tau_{2}} \, \Pr{\Ts^{\bullet,n}_{3}=\tau_{3}},
  \label{eq:ProbTProbTi}
\end{equation}
where as in Corollary~\ref{cor:probaGivenNumber}, $ \Ts^{\circ,n}_{1}$ and $ \Ts^{\circ,n}_{2}$ are two-type alternating random BGW trees  with offspring distributions $\mucn,\mubn$ with white roots, and $ \Ts^{\bullet,n}_{3}$ is the two-type alternating random BGW tree with black root (which was simply denoted by $\Ts$ before); we also take  $\Ts^{\circ,n}_{1}$, $\Ts^{\circ,n}_{2}$  and $\Ts^{\bullet,n}_{3}$ to be independent.
  The factor $\tfrac{\mubn(2)}{\mubn(0)}$ comes from the fact that when we make the decomposition
  of $\tau$ into three pieces $(\tau_{1},\tau_{2},(\tau_{3},v))$, the black vertex $v$  with two children
  is replaced with a vertex having no children.

  Combining \eqref{eq:BijTTi} and \eqref{eq:ProbTProbTi}, we get
 \[ A_{n}= \sum_{\substack{(\tau_{1},\tau_{2},(\tau_{3},v)):\\ (\ast) \textrm{ holds}} } \frac{\mubn(2)}{\mubn(0)} \, \Pr{\Ts^{\circ,n}_{1}=\tau_{1}} \, \Pr{\Ts^{\circ,n}_{2}=\tau_{2}} \, \Pr{\Ts^{\bullet,n}_{3}=\tau_{3}},\] 
 where $(\ast)$ are the conditions appearing in \eqref{eq:BijTTi}.

  We now re-index  the sum according to the numbers $n^\circ_1$, $n^\bullet_1$, $n^\circ_2$, $n^\bullet_2$
  of white and black vertices in $\tau_{1}$ and $\tau_{2}$, respectively.
  We also denote $n^\circ_3\coloneqq K_n+1-n^\circ_1-n^\circ_2$
  and $n^\bullet_3\coloneqq n-K_n-n^\bullet_1-n^\bullet_2$, which are the numbers of
  white and black vertices in $\tau_{3}$. By bounding from above the number of black leaves of $\tau_{3}$ by $\abs{\bullet_{\tau_{3}}}=n^\bullet_3$, we get that 
  \begin{multline*}
A_{n} \le \frac{\mubn(2)}{\mubn(0)} 
   \sum_{n^\circ_1, n^\bullet_1 \atop n^\circ_1 + n^\bullet_1 \geq \eps\, n} 
   \sum_{n^\circ_2, n^\bullet_2 \atop n^\circ_2 + n^\bullet_2 \geq \eps\, n} 
 \Bigg[ n^\bullet_3 \, \Pr{\abs{\bullet_{\Ts^{\circ,n}_{1}}}= n^\bullet_1, \abs{\circ_{\Ts^{\circ,n}_{1}}}= n^\circ_1} \\
  \qquad \cdot \Pr{\abs{\bullet_{\Ts^{\circ,n}_{2}}}= n^\bullet_2, \abs{\circ_{\Ts^{\circ,n}_{2}}}= n^\circ_2} \,
  \Pr{\abs{\bullet_{\Ts^{\bullet,n}_{3}}}= n^\bullet_3, \abs{\circ_{\Ts^{\bullet,n}_{3}}}= n^\circ_3} \Bigg].
  \end{multline*}
  We now use Eqs.~\eqref{eq:probaGivenNumberVertexRootBlack} and \eqref{eq:probaGivenNumberVertexRootWhite}
to express these probabilities using random walks 
(recall that $\Sbn$ and $\Scn$ are random walks with respective jump distributions given by $\mubn$ and $\mucn$):
  \begin{multline*}
   A_{n} \le \frac{\mubn(2)}{\mubn(0)} 
   \sum_{n^\circ_1, n^\bullet_1 \atop n^\circ_1 + n^\bullet_1 \geq \eps\, n} 
   \sum_{n^\circ_2, n^\bullet_2 \atop n^\circ_2 + n^\bullet_2 \geq \eps\, n} 
   \Bigg[ \frac{1}{n^\circ_1} \Pr{\Sbn_{n^\bullet_1}=n^\circ_1-1} \Pr{\Scn_{n^\circ_1}=n^\bullet_1}  \\
   \cdot \frac{1}{n^\circ_2} \Pr{\Sbn_{n^\bullet_2}=n^\circ_2-1} \Pr{\Scn_{n^\circ_2}=n^\bullet_2}  \,
   \Pr{\Sbn_{n^\bullet_3}=n^\circ_3} \Pr{\Scn_{n^\circ_3}=n^\bullet_3-1} \Bigg].
  \end{multline*}                                                                                                     
  Fix a positive constant $\delta \in (0,c \eps)$.
  We split the double sum in the right-hand side in two parts $\varSigma_1$ and $\varSigma_2$:
  $\varSigma_1$ is the sum of terms for which both $n^\circ_1 \ge \delta \sqrt{n}$ {\em and}
  $n^\circ_2 \ge \delta \sqrt{n}$,
  while $\varSigma_2$ contains all other terms.

\paragraph*{Bounding $\varSigma_1$.} Write
  \begin{multline*}   \varSigma_1 \le \frac{1}{\delta^2 n}
  \sum_{n^\circ_1, n^\bullet_1 \atop n^\circ_1 + n^\bullet_1 \geq \eps\, n} 
     \sum_{n^\circ_2, n^\bullet_2 \atop n^\circ_2 + n^\bullet_2 \geq \eps\, n}
     \Bigg[ \Pr{\Sbn_{n^\bullet_1}=n^\circ_1-1} \Pr{\Scn_{n^\circ_1}=n^\bullet_1} \\
     \cdot \Pr{\Sbn_{n^\bullet_2}=n^\circ_2-1} \Pr{\Scn_{n^\circ_2}=n^\bullet_2}
     \Pr{\Sbn_{n^\bullet_3}=n^\circ_3} \Pr{\Scn_{n^\circ_3}=n^\bullet_3-1} \Bigg].
   \end{multline*}
  Since $n^\circ_1 \le K_n+1 \sim c \sqrt{n}$, the condition $n^\circ_1 + n^\bullet_1 \geq \eps\, n$
  implies that we can find  $u \in (\tfrac{\delta}{c},\eps)$ such that  $n^\bullet_1 \geq u\, n$ for $n$ large enough.
  Similarly, we have $n^\bullet_2 \geq u\, n$ for $n$ large enough.
  Then Lemma~\ref{lem:ll1Uniforme} implies that
  $\Prb{\Sbn_{n^\bullet_1}=n^\circ_1-1} \le C_{1} n^{-1/4}$
  for $n$ large enough and for some constant $C_1>0$,
  uniformly for $n^\bullet_1 \in \{ \lceil un \rceil, \dots, n\}$ and for $n^\circ_1 \geq 1$.
  The same obviously holds when replacing indices $1$ with $2$.
By using the trivial bound $\Prb{\Sbn_{n^\bullet_3}=n^\circ_3} \le 1$, we get
  \[\varSigma_1 \le \frac{1}{\delta^2 n} \, C_{1}^2 n^{-1/2}
  \sum_{n^\circ_1, n^\circ_2}
  \left[ \sum_{n^\bullet_1, n^\bullet_2} \Pr{\Scn_{n^\circ_1}=n^\bullet_1} 
  \Pr{\Scn_{n^\circ_2}=n^\bullet_2}
  \Pr{\Scn_{n^\circ_3}=n^\bullet_3-1} \right].
  \]
  We claim that
  \[
   \sum_{n^\bullet_1, n^\bullet_2} \Pr{\Scn_{n^\circ_1}=n^\bullet_1} 
  \Pr{\Scn_{n^\circ_2}=n^\bullet_2}
  \Pr{\Scn_{n^\circ_3}=n^\bullet_3-1} =\Pr{\Scn_{K_n+1}=n-K_n-1}.\]
To see this, note that $n^\circ_1+n^\circ_2+n^\circ_3= K_n+1$ and $n^\bullet_1+n^\bullet_2+n^{\bullet}_{3}-1=n-K_{n}-1$,
and simply split the event $\{\Scn_{K_n+1}=n-K_n-1\}$,
  according to the intermediate values $\Scn_{n^\circ_1}$ and $\Scn_{n^\circ_1+n^\circ_2}$.

  By Lemma~\ref{lem:ll2}, there exists a constant $C_{2}>0$ such that  $\Prb{\Scn_{K_n+1}=n-K_n-1}   \leq C_{2}/n$ for $n$ sufficiently large. Also, the number of possible values for $n^\circ_1$ is at most $K_n+1$; and similarly for $n^\circ_2$.
  A straightforward analogue of \eqref{eq:tail_mub} gives $\mubn(2)=\mathcal O(n^{-1})$,
  while $\mubn(0)$ tends to a constant.
  Bringing everything together, there exists a constant $C_{3}>0$ such that
  \[ \frac{\mubn(2)}{\mubn(0)} \varSigma_1 \le  \frac{\mubn(2)}{\mubn(0)} \, \frac{1}{\delta^2 n} \, C_{1}^2 n^{-1/2}
  \sum_{n^\circ_1, n^\circ_2}
   \frac{C_{2}}{n}
 \le
 C_{3} \, \frac{1}{n} \, \frac{1}{\delta^2 n}  n^{-1/2}\,  n  \, \frac{1}{n}=\mathcal O(n^{-5/2}).\]

\paragraph*{Bounding  $\varSigma_2$.}
  As noticed above, we can find  $u \in (\tfrac{\delta}{c},\eps)$  such that for $n$ sufficiently large, in all summands, we have $n_1^\bullet \ge u\, n$ and $n_2^\bullet \ge u\, n$.
  Therefore, if $n_1^\circ < \delta\, \sqrt{n}$, Lemma~\ref{lem:Concentration} implies
  \[\Pr{\Sbn_{n^\bullet_1}=n^\circ_1-1} \le
  \Pr{\Sbn_{\lfloor u\, n \rfloor} \le \delta \sqrt{n}} \le \exp(-A \sqrt{n}).\]
  A similar bound holds if $n_2^\circ < \delta\, \sqrt{n}$.
  All the summands in $\varSigma_2$ are therefore bounded by $\exp(-A \sqrt{n})$.
  Since the number of terms is polynomial, we get that
  $\varSigma_2$ is exponentially small.

  To conclude, $A_{n}$ is $\mathcal O(n^{-5/2})$.
  Using \eqref{eq:probaGivenNumberVertexRootBlack} and the local limit theorems (Lemmas~\ref{lem:ll1} and \ref{lem:ll2}),
  we know that $\Prb{\abs{\bullet_{\Ts^{n}}}=n-K_n, \abs{\circ_{\Ts^{n}}} =K_n+1} \sim\ C n^{-9/4}$ as $n \rightarrow \infty$, for a certain constant $C>0$. 
  We have therefore proved that
  \[  \Pr{\exists v \in \bullet_{\Ts_{n}}: \, Q_\eps(\Ts_{n},v)}= \mathcal O(n^{-1/4}),\]
  which tends to $0$ as $n \rightarrow \infty$.  This completes the proof.
\end{proof}

\subsection{Local limit theorems}
\label{ss:llt}

We now establish two local limit estimates concerning the random walks $\Sbn$ and $\Scn$
(Lemmas~\ref{lem:ll1} and \ref{lem:ll2}).

\paragraph*{Local limit theorem for $\Sbn$.}

Let $\phibn(t)=\Esb{e^{it\Sbn_{1}}}= \tfrac{F( \bbn e^{it})}{F(\bbn)}$  be the characteristic function of $\Sbn_{1}$,
i.e.\ of the distribution $\mubn$.
(Recall that $F(z)=\sum_{k \ge 0} \tfrac{(k+1)^{k-1}}{k!} z^k$ was introduced in the proof of Lemma~\ref{lem:exist}.)

\begin{lemma}\label{lem:phib}
Assume that $\tfrac{K_{n}}{n} \rightarrow 0$ and $K_{n} \rightarrow \infty$ as $n \rightarrow \infty$. For every $n$ sufficiently large, for every $ 0 \leq |t| \leq \pi$  we have
\begin{equation}
\label{eq:phib}  \ln | \phibn (t)| \leq -  \frac{K_{n}}{ 8n} t^{2}.
\end{equation}
\end{lemma}

\begin{proof}
We take two different approaches depending on the distance of $|t|$ to $0$.

First assume that $|t| \leq 1$. We have (see e.g.~\cite[Lemma 3.3.7]{Dur10})
\[\phibn(t)=1+\textrm{i}\E[\Sbn_{1}]t- \E[(\Sbn_{1})^{2}] \frac{t^{2}}{2}+r^{n}_{1}(t),\]
with $|r^{n}_{1}(t)| \leq |t|^{3} \E[(\Sbn_{1})^{3}]/6$ for every $n \geq 1$ and $t \in \R$. By using the expansion of $F$ around $0$, we see that $ \E[\Sbn_{1}] \sim \E[(\Sbn_{1})^{2}] \sim \E[(\Sbn_{1})^{3}] \sim \frac{K_{n}}{n} \rightarrow 0$ as $n \rightarrow \infty$, so that
\begin{align*}
  \hspace{ -2mm} |\phibn(t)|^{2} &= \left( 1+\textrm{i}\E[\Sbn_{1}]t- \E[(\Sbn_{1})^{2}] \frac{t^{2}}{2}+r^{n}_{1}(t) \right) \left( 1-\textrm{i}\E[\Sbn_{1}]t- \E[(\Sbn_{1})^{2}] \frac{t^{2}}{2}+\overline{r^{n}_{1}(t)} \right)\\
&= 1 - (\sbn)^{2} t^{2}+r^{n}_{2}(t),
\end{align*}
where $|r^{n}_{2}(t)| \leq  |t|^{3} K_{n}/(2n)$ for every $n$ sufficiently large, uniformly for $t \in [-1,1]$.
Therefore
\[ 2 \ln |\phibn(t)| \leq - (\sbn)^{2} t^{2}+ |t|^{3} K_{n}/(2n)= - \frac{K_{n}}{n} t^{2} \left(  \frac{(\sbn)^{2} n}{K_{n}}-|t|/2 \right)
\]
for every $n$ sufficiently large, uniformly for $t \in [-1,1]$.
Since $(\sbn)^{2} \sim K_{n}/n$ as $n \rightarrow \infty$ by Lemma~\ref{lem:exist} (ii), it follows that for every $n$ sufficiently large and $0 \leq |t| \leq 1$ we have $ \ln | \phibn (t)| \leq -  \frac{K_{n}}{ 8n} t^{2}$.

Now assume that $1 \leq |t| \leq \pi$. Since $F(z)=1+z+o(z)$ as $z \rightarrow 0$ and, by definition, $\phibn(t)=\tfrac{F( \bbn e^{it})}{F(\bbn)}$,
we have
\[|\phibn(t)|^{2}=(1+\bbn(e^{it}-1)+o(\bbn))(1+\bbn(e^{-it}-1)+o(\bbn))=1+2\bbn(\cos(t)-1)+o(\bbn),\]
where the $o(\bbn)$ is uniform in $1\leq |t| \leq \pi$.
Therefore, for every $1 \leq |t| \leq \pi$, for every $n$ sufficiently large,
\[ 2 \ln |\phibn(t)| \leq - \bbn \frac{t^{2}}{3}+o(\bbn) \leq - \frac{K_{n}}{4n} t^{2},\]
where for the last inequality we have used the fact that $\bbn \sim \frac{K_{n}}{n}$ as $n \rightarrow \infty$ by Lemma~\ref{lem:exist} (ii) and that $t$ is not too close to $0$.
\end{proof}

\begin{proof}[Proof of Lemma~\ref{lem:ll1}]
We first check that the convergence
\begin{equation}
\label{eq:cv1b}\E\bigg[e^{i t \frac{\Sbn_{N}-N \frac{K_{n}+1}{n-K_{n}} }{ \DbnN }} \bigg]  \quad \mathop{\longrightarrow}_{n \rightarrow \infty} \quad e^{- \frac{t^{2}}{2}}
\end{equation}
holds uniformly for $un \leq N \leq n$ and $t$ in compact subsets of $\R$.  As in the proof of Lemma~\ref{lem:phib}, since $\tfrac{\E[(\Sbn_{1})^{3}]}{(\DbnN)^{3}}=\mathcal O(\tfrac{1}{n \sqrt{K_{n}}})= o(\frac{1}{n})$, we have
\[\phibn\left( \frac{t}{ \DbnN } \right)=1+\textrm{i}\E[\Sbn_{1}] \frac{t}{ \DbnN } - \E[(\Sbn_{1})^{2}] \frac{t^{2}}{2(\DbnN)^{2}}+ o \left( \frac{1}{n} \right),\]
where the $o$ is uniform when $un \leq N \leq n$ and $t$ belongs to a compact subset of $\R$, so that
\[\ln \phibn \left( \frac{t}{ \DbnN } \right)=it  \frac{K_{n}+1}{(n-K_{n})\DbnN} - \frac{t^{2}}{2} \cdot \frac{(\sbn)^{2}}{(\DbnN )^{2}}+ o \left( \frac{1}{n} \right).\]
Thus
\[\Es{e^{i t \frac{\Sbn_{N}-N \frac{K_{n}+1}{n-K_{n}} }{ \DbnN }}}= \exp \left(- \frac{t^{2}}{2} \cdot  \frac{ N (\sbn)^{2} }{(\DbnN)^{2}} + o \left( \frac{N}{n} \right)   \right).\]
Since $(\DbnN)^{2}= N (\sbn)^{2}$, this  implies \eqref{eq:cv1b}.

We then follow the steps of the analytic proof of the standard local limit theorem for a sequence of independent identically distributed random variables. The main difficulty is that the distribution of $\Sbn_{1}$ depends on $n$. To simplify notation, set $f(t)= e^{- \frac{t^{2}}{2}}$ for $t \in \R$. By Fourier inversion we have, for every $x \in \mathbb{R}$,

\begin{equation}
 \label{eq:fourierb}  p(x)= \frac{1}{2\pi} \int_{-\infty}^{\infty} e^{-itx} f(t) {\d}t.
 \end{equation}
Also, for $k \in \mathbb{Z}$,
\[ \Pr{\Sbn_{N}=k} = \frac{1}{2\pi} \int_{-\pi}^{\pi} e^{-itk} \phibn(t)^{ N} {\mathrm{d}}t.\]
In the following, we only consider values of $x \in \R$ such that $N \frac{K_{n}+1}{n-K_{n}}+x  \DbnN$ is an integer.
For such $x$,
\begin{multline*}
  \DbnN \cdot \Pr{\Sbn_{ N}= N \frac{K_{n}+1}{n-K_{n}}+x  \DbnN}\\
  =  \frac{1}{2\pi} \int^{\pi \DbnN}_{-\pi \DbnN} e^{- i t x} \left(  \phibn \left(  \frac{t}{\DbnN} \right)   e^{-it \frac{K_{n}+1}{(n-K_{n}) \DbnN}} \right)^{N} {\mathrm{d}}t.
\end{multline*}
Therefore for every fixed $A>0$, for $n$ sufficiently large,
\begin{multline*}
  \left| \DbnN \cdot \Pr{\Sbn_{N}= N \frac{K_{n}+1}{n-K_{n}}+x  \DbnN} -p( x)\right| \\
  \leq \frac{1}{2 \pi} \big(|I^{(1)}_{A}(x,N,n)| + |I^{(2)}_{A}(x,N,n)|+|I^{(3)}(x)|\big),
\end{multline*}
where
\begin{align*}
  I^{(1)}_{A}(x,N,n)&=\int_{-A}^{A} e^{-i t x} \left(   \left(  \phibn \left(  \frac{t}{\DbnN} \right)   e^{-it \frac{K_{n}+1}{(n-K_{n}) \DbnN}} \right)^{N}- f(t) \right) {\mathrm{d}}t,\\
 I^{(2)}_{A}(x,N,n)&= \int_{A<|t|< \pi \DbnN} e^{-itx - it \frac{(K_{n}+1) N}{(n-K_{n}) \DbnN}} \phibn \left(  \frac{t}{\DbnN} \right) ^{ N}{\mathrm{d}}t, \\
 I^{(3)}_{A}(x)&= \int_{|t|>A} e^{-itx} f(t) {\mathrm{d}}t.
 \end{align*}
We shall check that for every fixed $\epsilon>0$, there exists $A>0$ such that for every $n$ sufficiently large, for every $ un \leq N \leq  n$, for every $x \in \R$ (with the above integrality condition),  $|I^{(1)}_{A}(x,N,n)| \leq \epsilon$, $|I^{(2)}_{A}(x,N,n)| \leq \epsilon$ and $|I^{(3)}_{A}(x)| \leq\epsilon$.

 We choose $A >0$ such that
\begin{equation}
\label{eq:Ab} 2 \int_{A}^{\infty } e^{-  \frac{1}{9} t^{2}} {\d}t < \epsilon.
\end{equation}

\emph{Bounding $|I^{(1)}_{A}(x,N,n)|$.} Since the convergence \eqref{eq:cv1b} holds  uniformly on compact subsets of $\R$, for   $n$ sufficiently large, for every $ un \leq N \leq n$ and $x \in \R$, we have $|I^{(1)}_{A}(x,N,n)| \leq \epsilon$.

\emph{Bounding $|I^{(3)}_{A}(x)|$.} By the choice of $A$ in \eqref{eq:Ab}
since $|f(t)|=e^{-  \frac{1}{2} t^{2}} \le e^{-  \frac{1}{9} t^{2}}$,
for every $x \in \R$ we have $|I^{(3)}_{A}(x)| \leq \epsilon$.

 \emph{Bounding $|I^{(2)}_{{A} }(x,N,n)|$.}  By Lemma~\ref{lem:phib},  for every $n$ sufficiently large, for every $ un \leq N \leq n$, $A<|t| \leq \pi  \DbnN$ we have 
\[|I^{(2)}(x,N,n)| \leq 2 \int_{A}^{\pi \DbnN} e^{-  \frac{ K_{n} N }{ 8 n (\DbnN)^{2}} \cdot t^{2} } {\d}t \leq 2 \int_{A}^{\infty } e^{-  \frac{1}{9} t^{2}} {\d}t,\]
   which is less than $\epsilon$ by \eqref{eq:Ab}. This completes the proof.
\end{proof}

\paragraph*{Local limit theorem for $\Scn$.}

In the case of $\Scn$, the analysis is more subtle.
As above, we start with an estimate on the characteristic function 
of the step distribution of the random walk
\[\phicn(t)\coloneqq\Es{e^{it\Scn_{1}}}=\frac{F( \bcn e^{it})}{F(\bcn)}.\] 
To simplify notation, set  $y_{n}=1/e-\bcn$. 

\begin{lemma}\label{lem:technique}
Assume that $\tfrac{K_{n}}{\sqrt{n}} \rightarrow c>0$ as $n \rightarrow \infty$.  There exist constants $A_{0},\epsilon,\kappa>0$, which may only depend on $c$, such that the following holds.
For every $n$ sufficiently large and for every $|t| \in (\frac{A_{0}}{\Dcn},\epsilon)$, we have $|\phicn(t)| \leq e^{- \kappa |t|^{1/2}}$.
\end{lemma}

\begin{proof}We start with some preliminary observations which fix the values of the different constants.
\begin{enumerate}
\item[--]  Since $\Dcn \sim n /c$ and $y_{n} \sim \frac{c^{2}}{2en}$ (as was seen in the proof of Lemma~\ref{lem:exist}),
  we can find a constant $\gamma>0$ depending only on $c$ such that
  the conditions $t \ge  {A_0}/{\Dcn}$ and $A_{0} \geq 1$ imply
  $t \ge {\gamma A_0}/{n}$ and $t \ge \gamma y_{n}$,
  for $n$ sufficiently large.
\item[--] We then fix $A_{0} \geq 1$ such that $\frac{2c}{\sqrt{\gamma A_{0}}} \leq  \frac{1}{4\sqrt{e}}$.
\item[--] From the expansion $ \ln F \big( \tfrac{1}{e}-z \big)= 1-\sqrt{2e} \sqrt{z}+o(\sqrt{z})$, there exists $\eta>0$ such that 
\begin{equation}
\label{eq:z} \Re(z) \ge 0, \, |z| \leq \eta  \qquad \implies  \qquad  \Re \ln F \left( \frac{1}{e}-z \right) \leq  1-  \Re \sqrt{z}.
\end{equation}
Then set $\epsilon= \sqrt{\eta^{2}/(e^{-2}+\gamma^{-2})}$.
\end{enumerate}

We now turn to the main part of the proof. For $t \in \R$, since
$ \ln |\phicn(t)| =\Re \ln \phicn(t)$, we have
\begin{equation}
\label{eq:lnphi} \ln |\phicn(t)| =\Re \ln F( \bcn e^{it})- \ln F( \bcn).
\end{equation}
We first study the term $\ln F( \bcn)$.
By \eqref{eq:devF} and \eqref{eq:est_bcirc}, we have $\ln F( \bcn)=1- \tfrac{c}{\sqrt{n}}+ o \big(  \frac{1}{\sqrt{n}} \big)$.  As a consequence, for every $n$ sufficiently large, for every $  {A_{0}}/{\Dcn} \leq |t| \leq 1$, we have
\begin{equation}
\label{eq:ln}\ln F( \bcn) \geq 1- \frac{2c}{\sqrt{n}} \geq 1- \frac{2c}{\sqrt{\gamma A_{0}}} \sqrt{|t|}.
\end{equation}

Consider now the term $\Re \ln F( \bcn e^{it})$.
To this end, we set $z=z_{n}(t)=\tfrac{1}{e}-\bcn e^{it}$.
Note that we always have $\Re(z_n(t)) \ge 0$.
Observing that $e^{-it} z_{n}(t)= \tfrac{1}{e}(e^{-it}-1)+y_n$, we have
\begin{equation}                                                         
  \label{eq:module} |z_{n}(t)|^{2}=  \tfrac{1}{e^2} \big|e^{-it}-1\big|^2 
  +\tfrac{2}{e} \big\langle y_n, e^{-it}-1 \big\rangle +|y_n|^2 
  = \frac{2(1-\cos(t))}{e^{2}} -\frac{2 y_{n}(1-\cos(t))}{e}+y_{n}^{2}.
\end{equation}
For $n$ sufficiently large and $|t| \leq \epsilon$, it follows that
\begin{equation}
\label{ineq:module} |z_{n}(t)|^{2} \leq  \left( \frac{1}{e^{2}} + \frac{1}{\gamma^{2}} \right)  t^{2} \leq \eta^{2}.
\end{equation}
Using \eqref{eq:z}, we have
\[\Re \ln F( \bcn e^{it}) \leq 1-   \Re \sqrt{z_{n}(t)}.\]
Since $\Re z_{n}(t) \ge 0$, we have
\[ \Re \sqrt{z_{n}(t)}=\sqrt{\frac{\Re(z_n(t))+|z_n(t)|}{2}} \ge \sqrt{\frac{|z_n(t)|}{2}}. \]
But, by \eqref{eq:module}, for $n$ sufficiently large and  for every $|t|$ in $(A_{0}/\Dcn,1)$, 
\[|z_{n}(t)|^{2} = \frac{2(1-\cos(t))}{e^{2}} (1-y_{n} e)+y_{n}^{2} \geq \frac{1-\cos(t)}{e^{2}} \geq  \frac{t^{2}}{4e^{2}}.\] 
As a consequence $\Re \sqrt{z_{n}(t)} \geq \frac{\sqrt{|t|}}{2\sqrt{e}}$, so that 
\begin{equation}
\label{eq:reln}\Re \ln F( \bcn e^{it}) \leq 1- \frac{\sqrt{|t|}}{2\sqrt{e}}.
\end{equation}
To conclude, by combining \eqref{eq:lnphi} with \eqref{eq:ln} and \eqref{eq:reln},
we get that for every $n$ sufficiently large, for every $|t|$ in $(A_{0}/\Dcn, \epsilon)$,
\[\ln|\phicn(t)| \leq \left(  \frac{2c}{\sqrt{\gamma A_{0}}}- \frac{1}{2\sqrt{e}} \right) \cdot \sqrt{|t|} \leq -\frac{1}{4\sqrt{e}} \sqrt{|t|}.\]
This completes the proof, with $\kappa= \frac{1}{4\sqrt{e}}$.
\end{proof}

We are now ready to tackle the proof of Lemma~\ref{lem:ll2}.
\begin{proof}[Proof of Lemma~\ref{lem:ll2}]
The first step is to show the convergence
\begin{equation}
\label{eq:cv1}\Es{e^{i t \frac{\Scn_{ u K_{n} +j}}{ \Dcn }}}  \quad \mathop{\longrightarrow}_{n \rightarrow \infty} \quad \exp \left( u c^{2} \left( 1 -\sqrt{1-  \frac{2it}{c}}  \right)  \right),
\end{equation}
uniformly for $t$ in compact subsets of $\R$ and $|j| \leq n^{3/8}$.

The characteristic function $\phicn(t)=\Es{e^{it\Scn_{1}}}$  of $\Scn_{1}$
is given as $\phicn(t)= \frac{F( \bcn e^{it})}{F(\bcn)}$,
where $\bcn$ has been chosen such that $ \frac{\bcn F'(\bcn)}{F(\bcn)}= \frac{n-K_{n}}{K_{n}+1}$.
Recall that $\bcn=1/e-y_{n}$ with 
\hbox{$y_{n} \sim \frac{1}{2e}  \big( \frac{K_{n}}{n} \big)^{2} \sim \frac{c^{2}}{2e} \cdot \frac{1}{n}$}
(from \eqref{eq:est_bcirc}).
Since $ \Dcn \sim  n /c$, we have, uniformly on compact subsets of $\R$ as $n \rightarrow \infty$,
\[\bcn e^{ \frac{it}{\Dcn}}=(1/e-y_{n})e^{ \frac{it}{\Dcn}}=1/e-  \left( \frac{c^{2}}{2e} - \frac{it c}{e } \right) \cdot \frac{1}{{n}}+o \left(  \frac{1}{n} \right).\]
By \eqref{eq:devF}, we have $ \ln F \left( \frac{1}{e}-z \right)= 1-\sqrt{2e} \sqrt{z}+o(\sqrt{z})$ as $z \rightarrow 0$.
As a consequence as $n \rightarrow \infty$, \[\ln \phicn \left( \frac{t}{ \Dcn } \right)=-\sqrt{2e} \sqrt{ \frac{c^{2}}{2e} - \frac{it c}{e}} \cdot \frac{1}{\sqrt{n}} + \sqrt{2e} \sqrt{ \frac{c^{2}}{2e}} \cdot \frac{1}{\sqrt{n}}+o \left(  \frac{1}{\sqrt{n}} \right),\]
 uniformly on compact subsets of $\R$. Therefore
\[( uK_{n}  +j) \ln \phicn \left( \frac{t}{\Dcn} \right)  \quad \mathop{\longrightarrow}_{n \rightarrow \infty} \quad u c^{2}-uc \sqrt{c^{2}- 2itc},\]
 uniformly when $t$ belongs to compact subsets of $\R$ and $|j| \leq n^{3/8}$, which implies \eqref{eq:cv1}. 

 As in the proof of Lemma~\ref{lem:ll1}, the second step consists in estimating $\Pr{\Scn_{ uK_{n}  +j}=k}$ by Fourier inversion.
 We first let, for $t \in \mathbb{R}$, 
\[f(t)=\exp \left( u c^{2} \left( 1 -\sqrt{1-  \frac{2it}{c}}  \right)  \right)\]
be the expression appearing in \eqref{eq:cv1}. 
By \cite[Sec.1.2.5]{Kyp06}),  $f$ is the characteristic function of a random variable with explicit density 
\[q_u(x)=\left( \frac{u^{2} c^{3}}{2 \pi x^{3}} \right)^{1/2} \exp \left(  - \frac{c (x-uc)^{2} }{2x} \right) \mathbbm{1}_{x>0},\]
so that by Fourier inversion we have, for every $x \in \mathbb{R}$,
\begin{equation}
 \label{eq:fourier} 
 q_u(x)=  \frac{1}{2\pi}\int_{-\infty}^{\infty}  {\mathrm{d}t} e^{-itx} \exp \left( u c^{2} \left( 1 -\sqrt{1-  \frac{2it}{c}}  \right)  \right).
 \end{equation}
On the other hand, 
for $k \in \mathbb{Z}$,
\[ \Pr{\Scn_{ uK_{n}  +j}=k} = \frac{1}{2\pi} \int_{-\pi}^{\pi} e^{-itk} \phicn(t)^{ uK_{n}  +j} {\mathrm{d}}t.\]
Therefore, assuming that $x \in \R$ is chosen so that $x \Dcn$ is an integer, we get that
\[ \Dcn \cdot \Pr{\Scn_{ uK_{n}  +j}=x  \Dcn}=  \frac{1}{2\pi} \int^{\pi \Dcn}_{-\pi \Dcn} e^{- i t x} \phicn \left(  \frac{t}{\Dcn} \right) ^{ uK_{n}  +j} {\mathrm{d}}t.\]
We deduce that for every fixed $A>0$ and $0< \epsilon \leq 1$, for $n$ sufficiently large, it holds 
\begin{multline*}
  \left| \Dcn \cdot \Pr{\Scn_{ uK_{n}  +j}=x  \Dcn} - q_{u}(x)\right| \\
  \leq \frac{1}{2 \pi} (|I^{(1)}_{A}(x,j,n)| + |I^{(2)}_{\epsilon,A}(x,j,n)|+|I^{(3)}_{\epsilon}(x,j,n)|+|I^{(4)}_{A}(x)|),
\end{multline*}
where
\[I^{(1)}_{A}(x,j,n)=\int_{-A}^{A} e^{-i t x} \left(  \phicn \left(  \frac{t}{\Dcn} \right) ^{ uK_{n}  +j}- f(t) \right) {\mathrm{d}}t\]
\[ I^{(2)}_{\epsilon,A}(x,j,n)= \int_{A<|t|< \epsilon \Dcn} e^{-itx} \phicn \left(  \frac{t}{\Dcn} \right) ^{ uK_{n}  +j}{\mathrm{d}}t,\]
\[ I^{(3)}_{\epsilon}(x,j,n)=\int_{\epsilon \Dcn<|t|<\pi \Dcn} e^{-itx} \phicn  \left(  \frac{t}{\Dcn} \right) ^{ uK_{n}  +j} {\mathrm{d}}t, \qquad I^{(4)}_{A}(x)= \int_{|t|>A} e^{-itx} f(t) {\mathrm{d}}t .\]
We shall check that for every fixed $\epsilon'>0$, there exists $A>0$ and $0 < \epsilon \leq 1$ such that for every $n$ sufficiently large, for every $|j| \leq n^{3/8}$, for every $x \in \R$,
 $|I^{(1)}_{A}(x,j,n)| \leq \epsilon', |I^{(2)}_{\epsilon,A}(x,j,n)| \leq \epsilon', |I^{(3)}_{\epsilon}(x,j,n)| \leq \epsilon', |I^{(4)}_{A}(x)| \leq\epsilon'$.

Fix $\epsilon'>0$. We first explain how to choose $A$ and $\epsilon$. 
As a preliminary observation, note that there exists a constant $C_{1}>0$
(which only depend on $u$ and $c$) such that $|f(t)| \leq C_{1} e^{-C_{1}^{-1} \sqrt{t}}$ for $t \geq 0$
(this is straightforward, using the explicit expression of $f$).
Let $A_{0}$ and $\kappa$ be the constants given by Lemma~\ref{lem:technique}.
We may choose $A \geq A_{0}$ such that
\begin{equation}
\label{eq:A}2\int_{A}^{\infty} |f(t)|  \mathrm{d}t <\epsilon' \qquad \textrm{and} \qquad  2 \int_{A}^{\infty} e^{- \frac{\kappa u c^{3/2}}{2} t^{1/2}} \mathrm{d}t < \epsilon'.
\end{equation}
We then let $\epsilon$ be also given by Lemma~\ref{lem:technique}.

\emph{Bounding $|I^{(1)}_{A}(x,j,n)|$.} Since the convergence \eqref{eq:cv1} holds  uniformly on compact subsets of $\R$ and $|j| \leq n^{3/8}$, for   $n$ sufficiently large, for every $|j| \leq n^{3/8}$ and $x \in \R$, we have $|I^{(1)}_{A}(x,j,n)| \leq \epsilon'$.

\emph{Bounding $|I^{(4)}_{A}(x)|$.} By the choice of $A$ in \eqref{eq:A}, for every $n$ sufficiently large, for every $x \in \R$ we have $|I^{(4)}_{A}(x,n)| \leq \epsilon'$.

\emph{Bounding $|I^{(3)}_{\epsilon}(x,j,n)|$.}  Set $\phi(t)=\frac{F(e^{it}/e)}{F(1/e)}$, which is the characteristic function of  the probability distribution $ \frac{(i+1)^{(i-1)}}{e^{i+1} i!}$ on $\Z_{+}$. Since the latter is non lattice, we have $|\phi(t)|<1$ for $t \in (0,2\pi)$ (see e.g.~\cite[Theorem 3.5.1]{Dur10}). 
Therefore, there exists $C_2>0$ such that $|\phi(t)|<e^{-2C_2}$ for $ \epsilon \leq |t| \leq \pi$.
Since $F$ is a power series with nonnegative coefficients and converges in $1/e$,
it converges uniformly on $\{z: |z| \le 1/e\}$ and is continuous on this compact set.
It is therefore uniformly continuous and
we have $\phicn(t) \rightarrow \phi(t)$, uniformly for $t \in \R$.
In particular, for $n$ sufficiently large and $t$ with $ \epsilon \leq |t| \leq \pi$, 
we have $|\phicn(t)|<e^{-C_2}$.
 Therefore, for every $n$ sufficiently large, for every $|j| \leq n^{3/8}$ and  $x \in \R$,
 \[ \left|I^{(3)}_{\epsilon}(x,j,n) \right|  \leq 2 \int_{\epsilon \Dcn}^{\pi \Dcn} e^{-C_2 (uK_{n}  +n^{3/8})} \mathrm{d}t
 \leq 2 \pi \Dcn \cdot (e^{- C_2 u c \sqrt{n}/2} ) \le \epsilon'.\]
 
 \emph{Bounding $|I^{(2)}_{\epsilon, {A} }(x,j,n)|$.}  By Lemma~\ref{lem:technique}, for every $A<|t| \leq \epsilon \Dcn$ we have \[  \left| \phicn \left(  \frac{t}{\Dcn} \right) \right|  \leq e^{-\kappa \left(  \frac{|t|}{\Dcn} \right) ^{1/2}}.\]
Hence, for $n$ sufficiently large, for every $|j| \leq n^{3/8}$ and $x \in \R$,
 \[ \left| I^{(2)}_{\epsilon,A}(x,j,n)  \right| \leq 2 \int_{A}^{\epsilon \Dcn}\left( e^{-\kappa  \left(  \frac{t}{\Dcn} \right) ^{1/2}} \right) ^{ uK_{n}  -n^{3/8}} \mathrm{d}t .\]
 Since $K_{n} \sim c \sqrt{n}$ and $ \Dcn\sim  n /c$, it follows that for $n$ sufficiently large, for every $|j| \leq n^{3/8}$ and $x \in \R$,
  \[ \left| I^{(2)}_{\epsilon,A}(x,j,n) \right| \leq 2 \int_{A}^{\epsilon \Dcn}  e^{- \frac{\kappa u  c^{3/2}}{2} t^{1/2}} \mathrm{d}t  \leq 2 \int_{A}^{\infty} e^{- \frac{\kappa u c^{3/2}}{2} t^{1/2}} \mathrm{d}t,\]
  which is less than $\epsilon'$ by \eqref{eq:A}.   This completes the proof. 
  \end{proof}

\section{Minimal factorizations and  laminations}
\label{sec:lam}

\begin{table}[htbp]\caption{Table of the main notation and symbols appearing in Section~\ref{sec:lam}.}
\centering
\begin{tabular}{c c p{0.73 \linewidth} }
\toprule
$\mathbf{L}_{\infty} $ & & The Brownian triangulation. \\
$\mathbf{L}_{c} $ && The lamination coded by an excursion of a Lévy process with characteristic exponent given by \eqref{eq:Levyc}.\\
$(t_{1}^{(n)}, \ldots, t_{n-1}^{(n)})$  && A minimal factorization of $\mathfrak{M}_{n}$.\\
$ \PPP_{n}$, $ \dot{\PPP}_{n}$  && The non-crossing partition $ \PPP(  t_{1}^{(n)}, \ldots, t_{K_{n}}^{(n)})$ and its associated compact set. \\
$ \FFF_{n}$, $\dot{ \FFF}_{n}$&&  The non-crossing forest $ \FFF(  t_{1}^{(n)}, \ldots, t_{K_{n}}^{(n)})$ and its associated compact set.\\
$T_{n}$   &&  The tree $ \mathcal{T}(  \PPP_{n})$   coding the non-crossing partition $ \PPP_{n}$.\\
\bottomrule
\end{tabular}
\label{tab:seclam}
\end{table}

A \emph{geodesic lamination} $L$ of $\overline{\D}$ is a closed subset of $\overline{\D}$ such that $L$  can be written as the union non-crossing chords, i.e.\ which do not intersect in $\D$.
Recall that by convention, $ \{x\}$ is a chord for every $x \in \S$. In particular, $\S$ is a geodesic lamination. In the sequel, by lamination we will always mean geodesic lamination of $\overline{\D}$.  It is well known (and simple to see) that the set of all geodesic laminations of $\overline{\D}$ equipped with Hausdorff topology is compact.

\subsection{Definition of the random lamination $\mathbf{L}_{c}$}
\label{sec:deflam}

We start with defining $\mathbf{L}_{c}$ in the three cases $c=0$, $c=\infty$ and $c \in \R_{+}$. First, for $c=0$, we simply let $\mathbf{L}_{0}$ be the unit circle $\mathbb{S}$. 

\paragraph*{Case $c=\infty$.}  Let $Z=(Z_{u})_{0 \leq u \leq 1} \in \mathbb{C}([0,1],\R)$ be a continuous function.
We assume \hbox{$Z_{0}=Z_{1}=0$}, as well as the following property:
\begin{enumerate}[label=\color{blue}(C\arabic*)]
\item\label{C1} Local minima of $Z$ have distinct values,
  meaning that for every $0 \le s < t \le 1$, there exists at most one value $r \in (s, t)$ such that $Z_{r} = \inf_{[s, t]} Z$.
\end{enumerate}
Following \cite{Ald94a,LGP08}, we construct a lamination $L(Z)$ from $Z$ as follows.
For
every $s,t \in [0,1]$, we set $s \, {\sim}^{Z}\,  t$ if we have $Z_{s}
=Z_{t} = \min_{[s\wedge t, s \vee t]} Z$. We  set
 \begin{equation}
 \label{eq:defbl}
L(Z) \quad \coloneqq \quad  \bigcup_{s \sim^{Z} t} \big[e^{-2\textrm{i} \pi s},e^{-2\textrm{i} \pi t}\big].
 \end{equation}
 Then $L(Z) $ is a geodesic lamination, which is a triangulation (in the sense that the
complement of $ L(Z) $ in $\overline{ \mathbb{D}}$ is a
disjoint union of open Euclidean triangles whose vertices belong to
the unit circle).  In particular,  $L(Z) $ is maximal with respect to inclusion for geodesic laminations (see \cite[Proposition 2.1]{LGP08}). Observe that $\S \subset  L(Z)$. 
Using a continuity argument and the fact that the set of local minima is countable, it is a simple matter to see that 
 \begin{equation}
 \label{eq:adherence}
 L(Z)= \overline{ \bigcup_{\substack{s \sim^{Z} t, s \neq t\\ s,t \textrm{ are not local minima}}} \big[e^{-2\textrm{i} \pi s},e^{-2\textrm{i} \pi t}\big]}.
 \end{equation}
 (Roughly speaking, this means that $ L(Z)$ is the closure of all non-trivial chords which do not belong to triangles.)
 In a similar way, since $Z$ is continuous, we have
  \begin{equation}
 \label{eq:isolated} s \sim^{Z} t \ \implies \  \forall \epsilon>0,\, \exists s',t' \textrm{ such that }  s' \sim^{Z} t'  \textrm{ with } 0<|s-s'| \leq \epsilon, 0<|t-t'| \leq \epsilon. 
  \end{equation}
 (Roughly speaking, this means that  chords are not isolated). 
  
 We now consider the {\em Brownian excursion} $\mathbbm{e}$,
 which is informally defined as a Brownian motion
 conditioned to return to the origin at time $1$ and to stay positive on the time interval $(0,1)$.
 It is well known that it satisfies \ref{C1}
 a.s.  We then define
 \[\mathbf{L}_{\infty} \coloneqq L(\mathbbm{e})\] 
 to be the Brownian triangulation, see Figure~\ref{fig:BrownianTriangulation} for a simulation.

 \begin{figure}[ht]
   \[\begin{array}{c}
     \includegraphics[scale=.45]{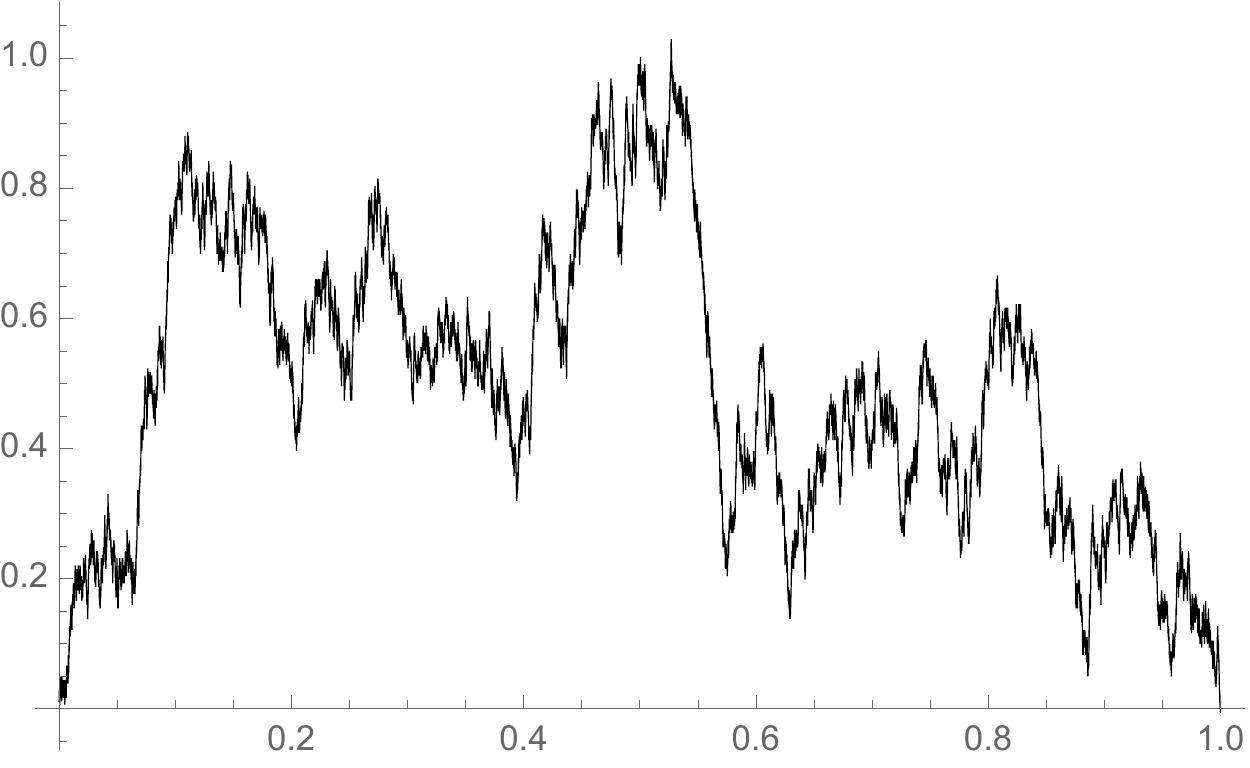} 
   \end{array} \qquad
   \begin{array}{c}
     \includegraphics[scale=.35]{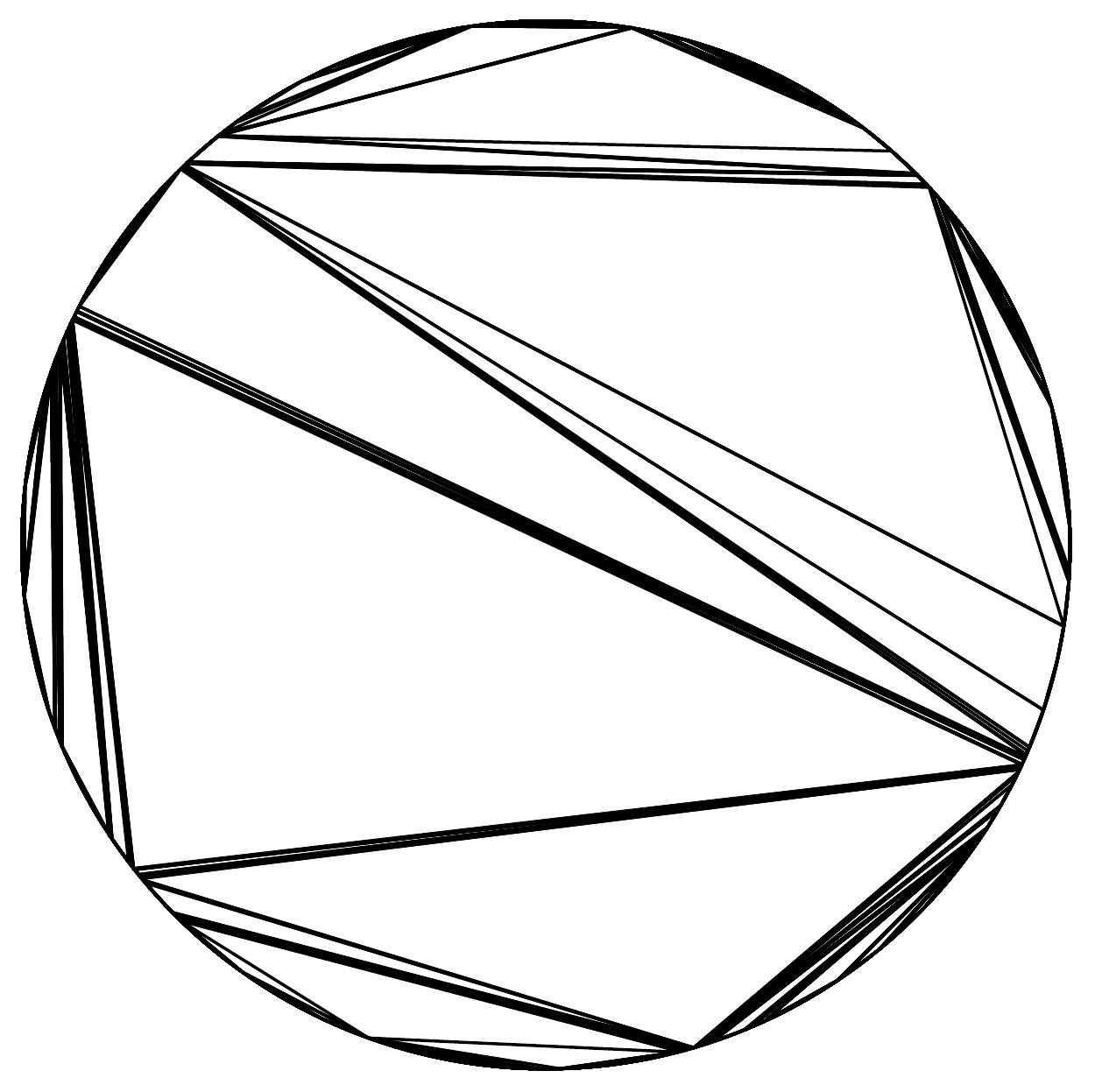}
   \end{array}\]
   \caption{A realization of the Brownian excursion $\mathbbm{e}$ and its associated lamination.
   The latter is therefore a realization of the Brownian triangulation $\mathbf{L}_{\infty}$.}
   \label{fig:BrownianTriangulation}
 \end{figure}

The Brownian triangulation has been introduced by Aldous \cite{Ald94a,Ald94b}, and it has been shown to be the universal limit of many different families of random non-crossing configurations \cite{Kor14,CK14,KM17}.

\paragraph*{Case $c >0$.}

 Let $Z \in {\D}([0, 1], \R)$ be a càdlàg function.
 For $t  \in [0,1]$, we denote $Z_{t-}$ the left-hand limit of $Z$ in $t$,
 and set $\Delta Z_{t}=Z_{t}-Z_{t-}$, with the convention $Z_{0-}=0$.
 We fix a function $Z \in {\D}([0, 1], \R)$ such that $Z_{1} = 0$, $Z_{t} > 0$ and $\Delta Z_{t} \ge 0$ for every $t \in (0,1)$,
 and we define the following four properties:
\begin{enumerate}[label=\color{blue}(H\arabic*)]
  \setcounter{enumi}{-1}
\item\label{H0} $ \{t \in (0,1) : \Delta Z_{t}>0\}$ is dense in $[0,1]$. 
\item\label{H1}
For every $0 \le s < t \le 1$, there exists at most one value $r \in (s, t)$ such that $Z_{r} = \inf_{[s, t]} Z$.
\item\label{H2}
For every $t \in [0,1)$ such that $\Delta Z_{t} > 0$, we have $\inf_{[t, t+\varepsilon]} Z < Z_{t}$ for every $0 < \varepsilon \le 1-t$;
\item\label{H3}
For every $t \in (0,1)$ such that $Z$ attains a local minimum at $t-$ (meaning that there exists $\epsilon>0$ such that $Z_{t-}=\inf_{[t-\epsilon,t+\epsilon]} Z$), if $s = \sup\{ u \in [0,t] : Z_{u} < Z_{t-}\}$, then $\Delta Z_{s} > 0$ and $Z_{s-} < Z_{t-} < Z_{s}$.
\end{enumerate}

Following \cite{Kor14}, 
we construct a lamination $L(Z)$ from a function $Z$ with the above 4 properties
(note that in contrast with \cite{Kor14}, we have one less property,
and the analogue of \ref{H2} differs since we accept here the case $Z_{0} > 0$). 
To this end, we define a relation (which is not an equivalence relation in general) on $[0, 1]$ as follows:
for every $0 \le s < t \le 1$, we set
\[
s \simeq^Z t \qquad\text{if}\qquad t = \inf \left\{u > s : Z_{u} \le Z_{s-} \right\},
\]
then for $0 \le t < s \le 1$, we set $s \simeq^Z t$ if $t \simeq^Z s$, and we agree that $s \simeq^Z s$ for every $s \in [0,1]$. We finally define a subset of $\overline{\D}$ by
\begin{equation}
L(Z) \quad \coloneqq \quad \bigcup_{s \simeq^Z t} \left[\e^{-2\pi \i s}, \e^{-2\pi\i t}\right].
\label{eq:defl_cadlag}
\end{equation}
Observe that $\S \subset L(Z)$ by definition. By using \ref{H0} and \ref{H2}, it is a simple matter to check that 
\begin{equation}\label{eq:lamination_stable2}
L(Z) \quad = \quad \overline{\bigcup_{\substack{s \simeq^Z t \\ s \neq t}} \left[\e^{-2\pi \i s}, \e^{-2\pi\i t}\right]}.
\end{equation}
Note that, if $s \simeq^Z t$ with $s<t$, we automatically have $Z_{s-}=Z_{t}$.

\paragraph*{Comment on notation.} Observe that the definition of the lamination $L(Z)$ for càdlàg functions  (Eq.~\eqref{eq:defl_cadlag}) with the above properties differs from the one 
for continuous functions (Eq.~\eqref{eq:defbl}). Since these two sets of functions are disjoint, this should not create any confusion.

\begin{lemma}
\label{lem:lamination}Assume that  $Z \in {\D}([0, 1], \R)$ satisfies $Z_{1} = 0$, $Z_{t} > 0$ and $\Delta Z_{t} \ge 0$ for every $t \in (0,1)$, \ref{H1} and \ref{H3}. Then the lamination $L(Z)$ is a geodesic lamination of $\overline{\D}$, called the lamination coded by $Z$.
\end{lemma}
This can be proved by using the same arguments as \cite[Prop. 2.9]{Kor14}, we leave the details to the reader.
Note that the assumptions \ref{H0} and \ref{H2} are not required in Lemma~\ref{lem:lamination},
but will be later needed in Proposition~\ref{prop:cvLamDeterm}.

Recall from Section~\ref{ssec:Levy} that $\Xexc$ denotes the excursion process, 
associated with the L\'evy process with characteristic exponent given by \eqref{eq:Levyc}.
\begin{lemma}Almost surely $\Xexc$  satisfies \ref{H0}, \ref{H1}, \ref{H2} and \ref{H3}.
  \label{lem:hypoX}
\end{lemma}

\begin{proof}
We mimic the proof of \cite[Proposition 2.10]{Kor14}.  
By the construction of $\Xexc$ as the Vervaat transform of $\Xbr$, and by the absolute continuity relation \eqref{eq:lawbridge}, it is sufficient to prove that analogous properties hold
for the Lévy process $X$.  The property \ref{H0} follows from the fact that the Lévy measure of $X$ is infinite. The properties \ref{H1} and \ref{H2} are consequences
of the (strong) Markov property of $X$ and the fact that $0$ is regular for $(-\infty,0)$.

For \ref{H3}, we will use the time-reversal property of $X$, which
states that if $t > 0$ and $\widehat{X}^{(t)}$ is the process defined by $\widehat{X}^{(t)}_{s}=X_{t}-X_{(t-s)-}$ for $0 \leq s<t$ and $\widehat{X}^{(t)}_{t}=X_{t}$, then the two processes $(X_{s},0 \leq s \leq t)$ and $(\widehat{X}^{(t)}_{s},0 \leq s \leq t)$ have the same law. 
Set $S_{t}= \sup_{0 \leq s \leq t} X_{s}$. By the time-reversal property of $X$,
it is sufficient to prove that if $q > 0$ is rational and $T = \inf\{t  \geq  q : X_{t} > S_{q}\}$, then
$X_{T} > S_{q}  > X_{T-}$ almost surely. This follows from the Markov property at time $q$
and the fact that for any $a > 0$, $X$ jumps a.s. across $a$ at its first passage time
above $a$ (Lévy processes which can {\em creep upwards}
are characterized e.g. in~\cite[Theorem 7.11]{Kyp06};
since the Lévy process we are considering has paths of bounded variations,
a negative drift and no Gaussian component, it cannot creep upwards
and it jumps a.s. across any level $a$ at its first passage).
\end{proof}

\begin{remark}
Almost surely $\Xexc$ also satisfies the fact that for every $t \in (0,1)$ such that $\Delta \Xexc(t) > 0$, we have $\inf_{[t-\varepsilon, t]} \Xexc > \Xexc(t-)$ for every $0 < \varepsilon \le t$ (this follows from the time-reversal property of $X$ and the fact that $0$ is irregular for $(0,\infty)$), but we will not use this property.
\end{remark}

For $c>0$, 
we now define $\mathbf{L}_{c}$ as
\[\mathbf{L}_{c} \coloneqq L(\Xexc),\]
where $\Xexc$ is the Lévy excursion process introduced in Section~\ref{ssec:Levy};
see Figure~\ref{fig:Lamination_C5} for a simulation of $\mathbf{L}_{5}$.

 \begin{figure}[ht]
   \[\begin{array}{c}
     \includegraphics[scale=.45]{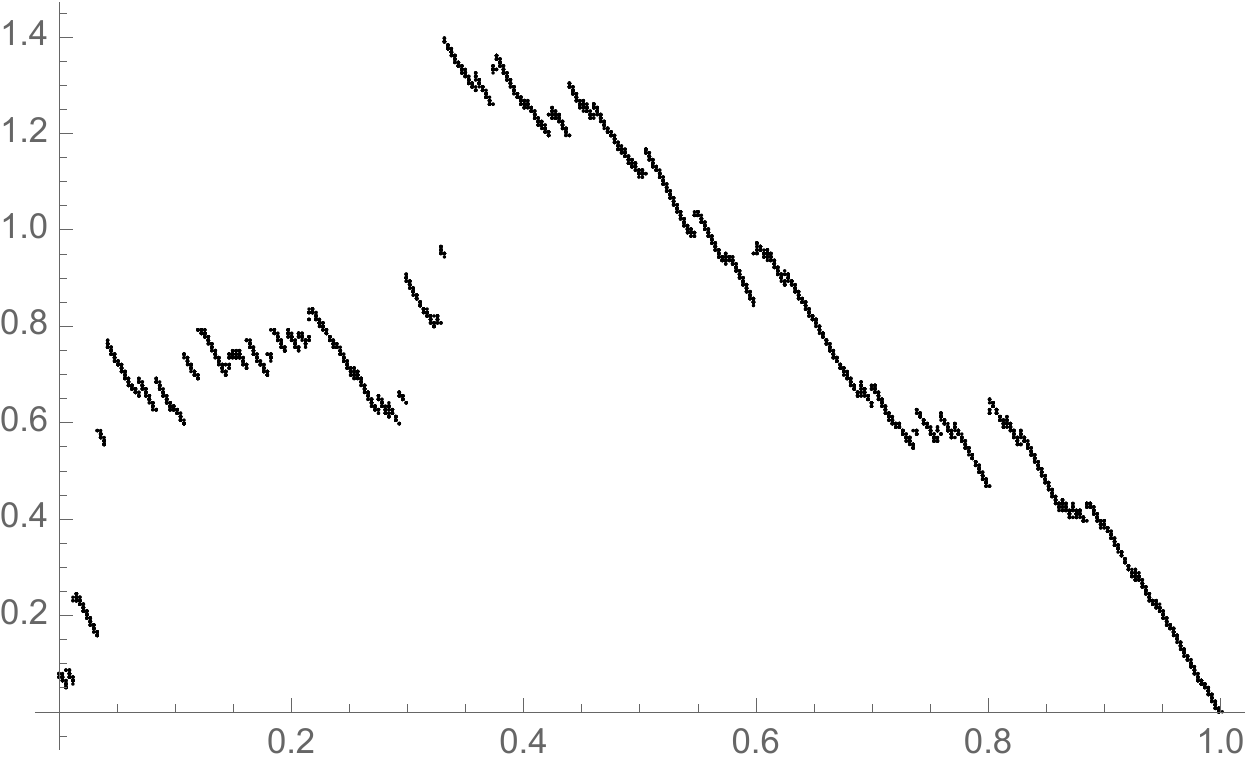} 
   \end{array} \qquad
   \begin{array}{c}
     \includegraphics[scale=.35]{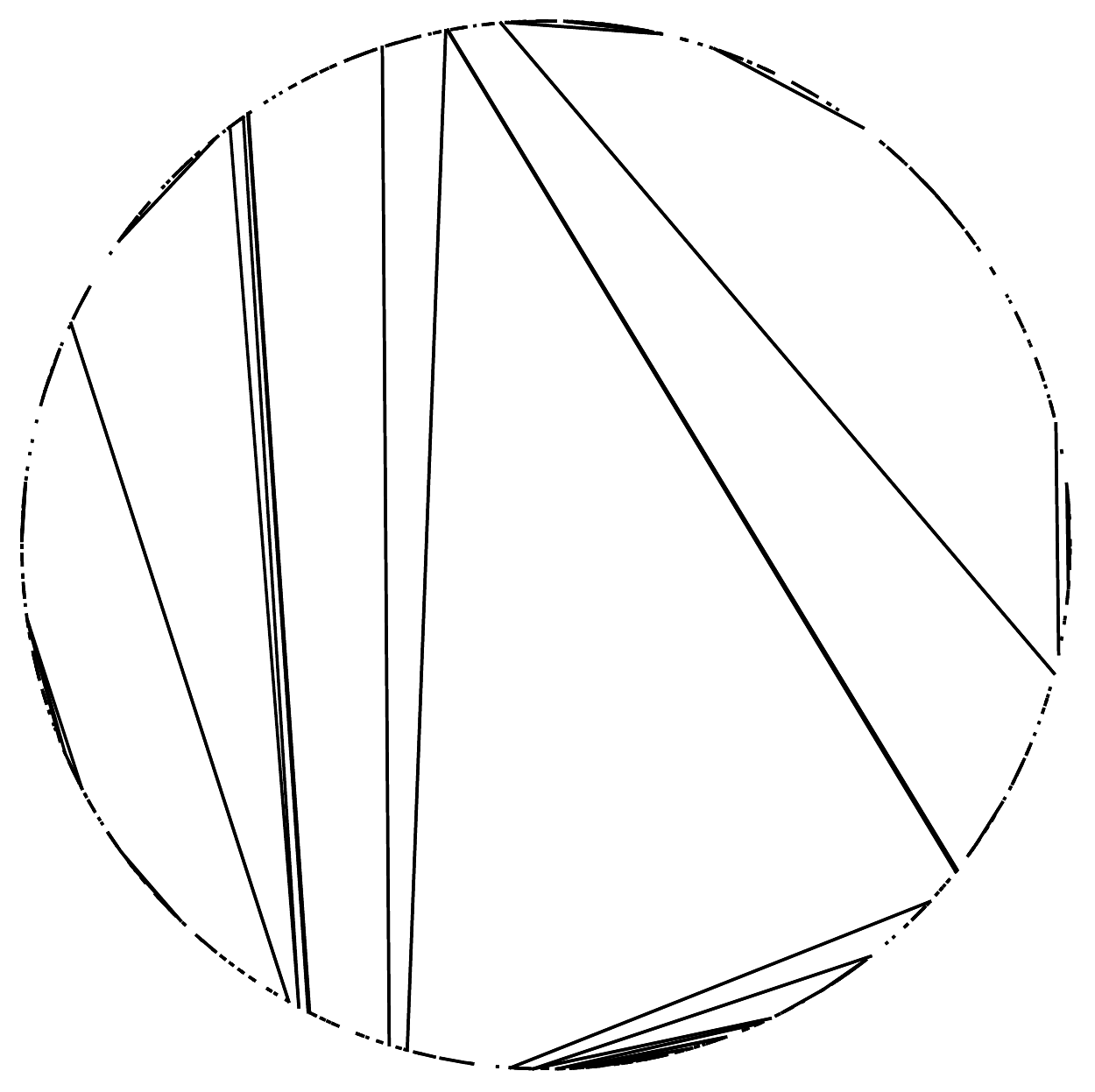}
   \end{array}\]
   \caption{A realization of the Lévy excursion process $\Xexc$ from Section~\ref{ssec:Levy} for $c=5$
   and its associated lamination.
   The latter is therefore a realization of the random lamination $\mathbf{L}_{5}$.}
   \label{fig:Lamination_C5}
 \end{figure}

In what follows, we first concentrate on the case $c\in \R_{+}$, and discuss the case $c=\infty$ in Section~\ref{sec:Linfty}.

\subsection{A symmetry result.} As a warm-up, we establish a result that will be useful twice: first, it will allow us to restrict ourselves to the case where $K_{n} \leq n/2$ in the proof of Theorem~\ref{thm:cvlam} (ii), and it will also allow us to immediately lift, for $ \dot{\mathcal{P}}^{(n)}_{K_{n}}$, (i) in  Theorem~\ref{thm:cvlam} to (iii).

Recall that if $P$ is a non-crossing partition, singleton blocks of $P$ do not appear in $\dot{P}$.  We denote by $\hat{P}$  the compact subset of $\overline{\D}$ obtained from $\dot{P}$ by adding the vertices $\{e^{-2 \i \pi k/n}, 1 \leq k \leq n\}$. This definition has its importance in view of the next result.

 \begin{lemma}
 \label{lem:reductionKn}
 Set $\hat{\PPP}^{(n)}_{i}=\hat{\PPP}(\tb_{1}^{(n)} \tb_{2}^{(n)} \cdots \tb_{i}^{(n)})$ for $1 \leq i \leq n-1$.
 Assume that there exists a random lamination $L$ such that the convergence $\hat{\PPP}^{(n)}_{n-K_{n}-1} \rightarrow  L$ holds in distribution.
Then 
\[  \hat{\PPP}^{(n)}_{K_{n}} \quad \mathop{\longrightarrow}^{(d)}_{n \rightarrow \infty} \quad L.\]
\end{lemma}

\begin{proof}
Recall that $ \mathfrak{P}_{n}$ denotes the set of all non-crossing partition of $[n]$ and that $\mathcal{K}(P)$ is the Kreweras complement of a non-crossing partition $P$.
Using the visual description of Kreweras complement (Figure~\ref{fig:complement_Kreweras}),
it is a simple matter to see that $ \max_{P \in \mathfrak{P}_{n}} d_{H} \big(\hat{P}, \widehat{\mathcal{K}(P)} \big) \rightarrow 0$ as $n \rightarrow \infty$.  As a consequence,  
    \[ \widehat{\mathcal{K} \big(\PPP^{(n)}_{n-K_{n}-1}\big)}    \quad \mathop{\longrightarrow}^{(d)}_{n \rightarrow \infty} \quad  L.\]
But by Lemma~\ref{lem:symmetry}, $\PPP^{(n)}_{K_{n}}$ and $\mathcal{K}( \PPP^{(n)}_{n-K_{n}-1})$ have the same distribution. This entails the desired result.  
\end{proof}

\subsection{Convergence to $\mathbf{L}_{c}$ for $c >0$ and proof of Theorem~\ref{thm:cvlam} (i)$_{c>0}$.}
\label{ssec:cv_pos}
The goal of this subsection is to prove Theorem~\ref{thm:cvlam} (i)$_{c>0}$, 
which deals with the case $\tfrac{K_{n}}{\sqrt{n}} \rightarrow c>0$.
Recall that if $\tau$ is a tree, $(\oB^{n}_{i}(\tau))_{1 \leq i \leq |\tau|}$ is the {\L}ukasiewicz path of the reduced black subtree of $\tau$, as introduced in Section~\ref{ss:coding}. We also set $\oB^{n}_{i}(\tau)=0$ for $i \leq 0$ or $i>|\tau|$.

For every $n \geq 1$, we let $(t_{1}^{(n)}, \ldots, t_{n-1}^{(n)}) \in \mathfrak{M}_{n}$  be a minimal factorization. If $1 \leq K_{n} \leq n-1$, to simplify notation we set
\[  \FFF_{n}= \FFF \big( t_{1}^{(n)}, \ldots, t_{K_{n}}^{(n)}\big) \qquad \text{and} \qquad \PPP_{n}= \PPP\big(  t_{1}^{(n)}  t_{2}^{(n)}  \cdots t_{K_{n}}^{(n)}\big),\]
and let $\dFn$ and respectively  $\dPn$ be their associated compact subsets of $\overline{\D}$.
Finally, we denote by $T_{n}= \mathcal{T}( \PPP_{n})$   the tree coding the non-crossing partition  $\PPP_{n}$ as in Section~\ref{sec:coding}.

Roughly speaking,  Theorem~\ref{thm:cvlam} (i)$_{c>0}$ follows from the following deterministic convergence result
(the proof of this implication is given at the end of Section~\ref{ssec:cv_pos}).

\begin{proposition}
  \label{prop:cvLamDeterm} Let $Z \in {\D}([0, 1], \R)$ be a càdlàg function satisfying \ref{H0}, \ref{H1}, \ref{H2} and \ref{H3}
  and let $\FFF_{n}$, $\PPP_{n}$, $\dFn$, $\dPn$ and $T_{n}$ be as above.
  Assume that the following four properties hold.
  \begin{enumerate}[label=\color{blue}(B\arabic*)]
  \setcounter{enumi}{-1}
\item\label{B0}  We have $\tfrac{K_n}{\sqrt{n}} \rightarrow c>0$ as $n \rightarrow \infty$.
\item\label{B1} The convergence 
\begin{equation}
 \left(  c \frac{\oB^{n}_{\lfloor u (n-K_{n}) \rfloor}(T_n)}{ n} \right)_{-1 \leq u \leq 1}
\quad \mathop{\longrightarrow}_{n \rightarrow \infty} \quad (Z_{u})_{- 1\leq u \leq 1}
\label{eq:cv_HB_Skorokhod}
\end{equation}
holds for the Skorokhod $J_{1}$ topology, where we set $Z_{u}=0$ for $u<0$.
\item\label{B2} For $n$ sufficiently large, $T_n$ has no black vertex  with at least three children.
\item\label{B3} For every $\eps>0$, there exists $n_0(\eps)>0$ such that  for every $n \ge n_0(\eps)$, the tree $T_n$ does not contain a black 
  vertex $v$ that has two children, both with at least $\eps \, n$ descendants (counting both white and black vertices). 
\end{enumerate}
Then we have the  convergence
\begin{equation}
\left(  \dFn, \dPn   \right)   \quad \mathop{\longrightarrow}_{n \rightarrow \infty} \quad ( L(Z), L(Z))
\label{eq:cv_FP_Skorokhod}
\end{equation}
in the Hausdorff topology.
\end{proposition}

Before tackling the proof of Proposition~\ref{prop:cvLamDeterm}, we state and prove two lemmas. The first one is a simple geometric lemma and the second one concerns approximation properties of the Skorokhod $J_{1}$ metric.

\begin{lemma}
\label{lem:proche}
Fix $\epsilon>0$ and assume that \ref{B2}, \ref{B3} hold. Then, for every $n$ sufficiently large,
\[d_{H}\left( \dFn  , \dPn  \right)  \leq 7\epsilon,\]
where $d_{H}$ denotes the Hausdorff distance.
\end{lemma}

\begin{proof}
Assume that $n$ is chosen sufficiently large so that \ref{B2} and \ref{B3} for our fixed value of $\epsilon$ are in force.
By \ref{B2}, $ \dPn$ is a disjoint union of line segments and of triangles. Also, by Lemma~\ref{lem:observation}, $ \dFn \subset \dPn$ and  $\dFn$ is obtained from $\dPn$ by removing an edge from every triangle. Now, if $ \Delta$ is a triangle of $\dPn$, it splits the unit circle into three arcs, and by \ref{B3}, one of these arcs has length at most $7 \epsilon$. If $ {\Delta}'$ is obtained from $\Delta$ by removing one of the three edges of $\Delta$, this implies that $d_{H}(\Delta,{\Delta}') \leq 7 \epsilon$. The desired result readily follows.
\end{proof}

\begin{lemma}
\label{lem:approxJ1}
Suppose that the assumptions of Proposition~\ref{prop:cvLamDeterm} are in force. Fix $0<s<t < 1$. 
Then the following assertions are equivalent:
\begin{enumerate}
\item[(i)] We have $s \simeq^Z t$;
\item[(ii)] There exists integers $i_{n}< j_{n}$ such that $ \tfrac{i_{n}}{n-K_{n}} \rightarrow s$, $ \tfrac{j_{n}}{n-K_{n}} \rightarrow t$ and
\[ j_{n}= \inf \left\{k>i_{n}: \,  \oB^{n}_k(T_n)=\oB^{n}_{i_{n}}(T_n)-1 \right\}.\]
\end{enumerate}
\end{lemma}

\begin{proof}
The implication $(i) \implies (ii)$ follows from standard approximation properties of the Skorokhod $J_{1}$ topology (see e.g. the proof of Lemma 4.6. in \cite{KM17}).

Now assume (ii). To simplify notation, set $ \hB^{(n)}(u)=c \tfrac{\oB^{n}_{\lfloor u (n-K_{n}) \rfloor}(T_n)}{ n}$, so that
\begin{equation}
\label{eq:ineg} \hB^{(n)}(r) \geq \hB^{(n)} \left(  \frac{i_{n}}{n-K_{n}} \right) =\hB^{(n)} \left(  \frac{j_{n}}{n-K_{n}} \right)+ \frac{c}{n}  \quad \textrm{for every } r \in \left[  \frac{i_{n}}{n-K_{n}}, \frac{j_{n}}{n-K_{n}}  \right).
\end{equation}

First assume that  $Z_{s}=Z_{s-}$.
Then $\hB^{(n)} \big(  \tfrac{i_{n}}{n-K_{n}} \big) \rightarrow Z_{s}$ 
and it follows from \eqref{eq:ineg} and from \ref{B1} that  $Z_r \ge Z_{t-}=Z_{s}$ for every $r \in (s,t)$.
In addition, we claim that $Z_{t-}=Z_{t}$.
Indeed, if $Z_{t-}<Z_{t}$ and if $Z_{u}\geq Z_{s}$ for $u$ belonging to a small neighborhood of the left of $s$, 
then \ref{H1} is not fulfilled.
On the other hand, if $Z_{t-}<Z_{t}$ and if $s=\sup\{ u \in [0,t] : Z_{u} < Z_{t-}\}$, then \ref{H3}
is not fulfilled.
Therefore, we must have $Z_{t-}=Z_{t}$.
Similarly, we necessarily have $Z_{r}>Z_{t}$ for every $r \in (s,t)$.
Indeed, assuming $Z_r=Z_t$ yields a contradiction with either \ref{H1} or \ref{H3},
using the same case distinction as above.
We conclude that $s \simeq^Z t$.

Next assume that $Z_{s}>Z_{s-}$. Then, by \ref{H2}, we can find $r \in (s,t)$ such that $Z_{r}<Z_{s}$. 
As a consequence,  $\hB^{(n)} \big(  \tfrac{i_{n}}{n-K_{n}} \big) \rightarrow Z_{s-}$
(it cannot converge to $Z_s$ because of \eqref{eq:ineg}), and $Z_{t-}=Z_{s-}$. 
Similarly as above, we can prove that $Z_{t}=Z_{t-}$.
We thus have $Z_{s-}=Z_{t} \leq Z_{r}$ for every $r \in (s,t)$, 
and this inequality is strict (otherwise, as above, this would contradict either \ref{H1} or \ref{H3}).
Therefore $s \simeq^Z t$, and this completes the proof.
\end{proof}

We are now ready to establish Proposition~\ref{prop:cvLamDeterm}.

\begin{proof}[Proof of Proposition~\ref{prop:cvLamDeterm}]
Since the set of all geodesic laminations of $\overline{\D}$ equipped with Hausdorff topology is compact, up to extraction,
we may assume that $(\dFn, \dPn)$  converges to a pair of geodesic laminations $(K,K')$ of $\overline{\D}$ and we aim at proving that $(K,K')=( L(Z), L(Z))$. By Lemma~\ref{lem:proche}, $K=K'$ and it is therefore enough to show that $K=L(Z)$. Denote by $(v^{\bullet,n}_{i})_{0 \leq i < n-K_{n}}$ the black vertices of $T_{n}$ listed in lexicographical order.

First note that there can not exist $s \in (0,1)$ such that $0 \simeq^{Z} s$, 
and since $Z_{t}>0$ for $t \in (0,1)$, there does not exist $s \in (0,1)$ such that $ s \simeq^{Z} 1$ either.

\paragraph*{Step $1$: $L(Z) \subset K$.}  Fix $0 < s \le t < 1$ such that  $s \simeq^Z t$.
We need to prove that $[e^{-2\pi\, \textrm{i}\, s}, e^{-2\pi\, \textrm{i}\, t}] \subset K$.
By \eqref{eq:lamination_stable2}, it is enough to prove it for $s<t$.
Since $K$ is a limit point of the sequence $(\dPn)_{n \ge 0}$
it is enough to check that for fixed $\epsilon>0$, for every $n$ sufficiently large,
there exists a chord of $\dPn$  within distance $12\epsilon$ of the chord $[e^{-2\pi\, \textrm{i}\, s}, e^{-2\pi\, \textrm{i}\, t}]$.

 By Lemma \ref{lem:approxJ1}, for $n$ sufficiently large, there exist integers $i_{n},j_{n}$ such that
\begin{equation}
\label{eq:injn}
    \left|\frac{i_{n}}{n-K_n} -s \right|  \le \eps, \quad \left|\frac{j_{n}}{n-K_n} -t \right|  \le \eps, \quad
    j_{n}= \inf \left\{k>i_{n}: \,  \oB^{n}_k(T_n)=\oB^{n}_{i_{n}}(T_n)-1 \right\}.
\end{equation}
Note that the last equality means that $v^{\bullet,n}_{j_{n}}$ is the first vertex greater (in the lexicographical order) than $v^{\bullet,n}_{i_{n}}$ and than all the descendants of $v^{\bullet,n}_{i_{n}}$.

We now find a chord of $ \dPn $ which is within distance $12\epsilon$ of the chord $[e^{-2\pi\, \textrm{i}\, s}, e^{-2\pi\, \textrm{i}\, t}]$ as follows. Denote by  $x^{\bullet,n}_{i_{n}}$ (resp.~$y ^{\bullet,n}_{i_{n}}$)  the label of the first (resp.~the last) corner of $v^{\bullet,n}_{i_{n}}$ visited by the contour sequence.
By construction (see in Section~\ref{sec:coding} how the partition $\PPP_{n} $ is reconstructed from $T_{n}$),
we have $[e^{-2\pi\, \textrm{i} {x^{\bullet,n}_{i_{n}}}/{n}}, e^{-2\pi\, \textrm{i} {y ^{\bullet,n}_{i_{n}}}/{n}}] \in \dPn$.
Lemma~\ref{lem:OnBlockLabels}  (iii) then implies that
\begin{equation}
  i_{n} \leq x^{\bullet,n}_{i_{n}} \leq i_{n}+ K_{n}+1, \qquad  i_{n}+n^{\bullet}_{i_{n}} \leq y^{\bullet,n}_{i_{n}} \leq   i_{n}+n^{\bullet}_{i_{n}} + 2 K_{n}+2,
  \label{eq:Control_xy}
\end{equation}
where $n^{\bullet}_{i_{n}} $ is the number of black descendants of $v^{\bullet,n}_{i_{n}}$.
But, because of \eqref{eq:injn},  the black vertices visited between $v^{\bullet,n}_{i_{n}}$ and $v^{\bullet,n}_{j_{n}}$
in the contour walk of $T_{n}$ are exactly the black descendants of $v^{\bullet,n}_{i_{n}}$,
so that $n^{\bullet}_{i_{n}}= j_{n}-i_{n}-1$. Thus
\[i_{n} \leq x^{\bullet,n}_{i_{n}} \leq i_{n}+ K_{n}+1, \qquad  j_{n}-1 \leq y^{\bullet,n}_{i_{n}} \leq    j_{n}-1 + 2 K_{n}+2.\]
Since $K_{n} \sim c \sqrt{n}$, this implies that for every $n$ sufficiently large,
\[\left| \frac{x^{\bullet,n}_{i_{n}}}{n}-s \right| \leq 2 \epsilon, \qquad   \left| \frac{y^{\bullet,n}_{i_{n}}}{n}-t \right|  \leq 2 \epsilon.\]
As a consequence, for $n$ sufficiently large, $[e^{-2\pi\, \textrm{i} {x^{\bullet,n}_{i_{n}}}/{n}}, e^{-2\pi\, \textrm{i} {y ^{\bullet,n}_{i_{n}}}/{n}}] \in\dPn$  is within distance $12\epsilon$ of the chord $[e^{-2\pi\, \textrm{i}\, s}, e^{-2\pi\, \textrm{i}\, t}]$. This shows that $ L(Z)   \subset K$.

\paragraph*{Step $2$: $ K \subset L(Z)$.} 
Since $K$ is a geodesic lamination (i.e. a union of non crossing chords) and since $L(Z)$ contains the circle $\S$,
it is enough to show that nontrivial chords of $K$ are included in $L(Z)$.
Namely, we need to prove that if $s<t$ are such that $[e^{-2\pi\, \textrm{i}\, s}, e^{-2\pi\, \textrm{i}\, t}] \subset K$,
then $s \simeq^Z t$.
By definition, $K$ is a limit point of $(\dPn)_{n \ge 0}$
so, up to extraction, we may assume that there exist sequences $(s_{n})$ and $(t_{n})$
such that $s_{n} \rightarrow s$, $t_{n} \rightarrow t$, $n s_{n}$ and $n t_{n}$ are integers, and $[e^{-2\pi\, \textrm{i}\, s_{n}}, e^{-2\pi\, \textrm{i}\, t_{n}}] \subset \dPn$.
By construction,
the condition $[e^{-2\pi\, \textrm{i}\, s_{n}}, e^{-2\pi\, \textrm{i}\, t_{n}}] \subset \dPn$
means that
$ns_n$ and $nt_n$ are the labels of two consecutive corners of the same black vertex of $T_n$;
call $v^{\bullet,n}_{i_{n}}$ this black vertex.
Combining this with the hypothesis \ref{B2}, there are two cases:
$v^{\bullet,n}_{i_{n}}$ may have 1 or 2 children, but not more.

Let $n s'_{n}$ (resp.~$n t'_{n}$) is the index of the first (resp.~last) black corner of $v^{\bullet,n}_{i_{n}}$ visited by the contour sequence (if $v^{\bullet,n}_{i_{n}}$ has one child, we simply have $(s_{n},t_{n})=(s'_{n},t'_{n})$).
By \ref{B3}, for $n \ge n_0(\eps)$, only one of the quantities $t'_n-t_n$, $t_n-s_n$ and $s_n-s'_n$
can be bigger than $\eps$. 
Since $t>s$, for $\eps$ small enough and $n$ sufficiently large, we have $t_n-s_n>\eps$.
Therefore $t'_n-t_n$ and $s_n-s'_n$ tend to $0$ and 
we have $(s'_{n},t'_{n}) \rightarrow (s,t)$.
Note that we \emph{crucially} need \ref{B3} to rule out the possibility that  $v^{\bullet,n}_{i_{n}}$ has two white children both having a macroscopic descendance, in which case the last convergence may fail.
Then, by Lemma~\ref{lem:OnBlockLabels}   (iii),
  \[ i_{n} \leq n s'_{n} \leq i_{n}+  K_{n}+1, \qquad i_{n}+n^{\bullet}_{i_{n}} \leq n t'_{n} \leq   i_{n}+n^{\bullet}_{i_{n}} +2 K_{n}+2,\]
where $n^{\bullet}_{i_{n}}$ is the number of black descendants of $ v^{\bullet,n}_{i_{n}}$. 
Hence, since $K_n/n$ tends to $0$, we have $i_{n}/(n-K_{n}) \rightarrow s$ and $  (i_{n}+n^{\bullet}_{i_{n}})/(n-K_{n}) \rightarrow t$.
Besides, by construction of $\overline{B}$, it holds that
\[i_{n}+n^{\bullet}_{i_{n}}+1= \min \{k>i_{n}: \oB^{n}_{k}(T_{n}) =  \oB^{n}_{i_{n}}(T_{n})-1\}.\]
Lemma~\ref{lem:approxJ1} implies that $s \simeq^Z t$, and this completes the proof.
\end{proof}

We now conclude the proof of Theorem~\ref{thm:cvlam} (i) in the case $c>0$.
\begin{proof}[Proof of Theorem~\ref{thm:cvlam} (i) in the case $c>0$]
  Since the space of laminations of the disk is compact,
  it is enough to show that the sequence $\big( \dFKn  ,  \dPKn \big) $ has a unique  limit point in distribution. Up to extraction, we may therefore assume that $\big( \dFKn  ,  \dPKn \big)$  converges in distribution. We shall show that the limiting distribution is the law of $(\mathbf{L}_c,\mathbf{L}_c)$.

  For each $n \ge 1$, we denote by $\tTs_n= \mathcal{T}( \PPP( \tb_{1}^{(n)} \tb_{2}^{(n)} \cdots \tb_{K_{n}}^{(n)}))$ 
  the tree associated with $\PKn$.
  From Proposition~\ref{prop:bitype}, it has the distribution of an alternating BGW tree with our favorite offspring distributions
  and a modified offspring distribution at the root (as in Theorem~\ref{thm:cvXexc} (ii)).
  We consider the following events:
  \begin{itemize}
    \item for $\epsilon >0$, we let $A_{n}(\epsilon)$ be the event ``the tree $\tTs_{n}$ does not contain a black 
    vertex $v$ that has two children, both with at least $\eps \, n$ descendants'';
        \item $B_{n}$ is the event ``$\tTs_{n}$ does not contain any black vertex with three or more children''.
  \end{itemize}
  From Lemmas~\ref{lem:AtMost2} and \ref{lem:PasDeuxGrandsSousArbres} combined with Lemma~\ref{lem:root}, we know that
   as $n \rightarrow \infty$, $\Pr{B_{n}} \rightarrow 0$ and, for each $\eps >0$, $\Pr{A_{n}(\epsilon)} \rightarrow 0$ as well.
  Therefore, we can find a subsequence $(\phi(n))_{n \geq 1}$ such that, for every $n \geq 1$,
  \[ \Pr{{B_{\phi (n)}}} \leq \tfrac{1}{2^{n}}, \qquad \Pr{{A_{ \phi(n)} (\tfrac1n)}} \leq  \tfrac{1}{2^{n}}.\]
  The Borel-Cantelli lemma ensures the almost sure convergence of 
  $\One_{B_{\phi (n)}}$ and $\One_{A_{ \phi(n)} (1/n)}$ to $0$.
  On the other hand, from Theorem~\ref{thm:cvXexc} (ii), we know that the following convergence in distribution holds
  \[  \left(  c \frac{\oB_{\lfloor u (n-K_{n}) \rfloor}(\tilde{\Ts}_{n})}{ n} : -1 \leq u \leq 1\right)  \quad \mathop{\longrightarrow}^{(d)}_{n \rightarrow \infty} \quad (\Xexc_{u}: -1 \leq u \leq 1).\]
  The three previous convergences hold jointly in distribution, so that by Skorokhod representation theorem,
  we can find a coupling of  $(\tTs_{\phi(n)})_{n \ge 1}$ such that these convergences hold almost surely.
  The sequence of associated random factorizations almost surely satisfies the assumptions
  \ref{B1}, \ref{B2} and \ref{B3} of Proposition~\ref{prop:cvLamDeterm}.
  Besides, from Lemma~\ref{lem:hypoX}, the random process $\Xexc$ almost surely 
  satisfies conditions \ref{H0}, \ref{H1}, \ref{H2} and \ref{H3}.
  From Proposition~\ref{prop:cvLamDeterm},
   we conclude that $\big( \dFKphin  ,  \dPKphin \big)$ converges in distribution
   to $(\mathbf{L}_c,\mathbf{L}_c)$, since $\mathbf{L}_c=L(\Xexc)$ by definition. 
  Therefore the limiting distribution of the whole sequence $\big( \dFKn, \dPKn\big)$
  must be $(\mathbf{L}_c,\mathbf{L}_c)$ and this completes the proof of Theorem~\ref{thm:cvlam} (i) in the case $c>0$.
\end{proof}

\subsection{Convergence to $\mathbf{L}_{\infty}$ and proof of Theorem~\ref{thm:cvlam} (ii)}
\label{sec:Linfty}

Our goal is now to establish Theorem~\ref{thm:cvlam} (ii), when  both $\frac{K_{n}}{\sqrt{n}} \rightarrow \infty$ and  $ \frac{n-K_{n}}{\sqrt{n}} \rightarrow \infty$ as $n \rightarrow \infty$.
The main difference  with the case $ \frac{K_{n}}{\sqrt{n}}  \rightarrow c>0$ occurs when we assume $ \tfrac{K_{n}}{n} \rightarrow \gamma \in (0,1)$. 
In the latter case,  both the  number of white and black vertices then have a positive proportion in the dual tree coding $ \PPP(\tb_{1}^{(n)}\tb_{2}^{(n)} \cdots \tb_{K_{n}}^{(n)}) $.

As before, $\mubn$ and $\mucn$  will denote the probability distributions satisfying the conditions of Lemma~\ref{lem:exist}.
We also set $\widetilde{\mu}^{n}_{\bullet}(i)=\mubn(i-1)$ for $i \geq 1$.
We let $\tilde{\Ts}^{n}$ be a random tree with distribution ${\BGW}^{\widetilde{\mu}_{\bullet}^{n},\mubn,\mucn}$, as defined in \eqref{eq:BGWtilde}.
In words, $\tilde{\Ts}^{n}$ is 
 an alternating two-type BGW tree (with a black root), with offspring distributions $\mubn$ and $\mucn$ except the root which has offspring distribution $\widetilde{\mu}^{n}_{\bullet}$.
We then define $\tTs_n$, as $\tTs^{n}$ conditioned on having $n-K_{n}$ black vertices and $K_{n}+1$ white vertices. 
Recall from the beginning of Section~\ref{ss:coding} that we associate
two paths $(\oB^{n}_{i}(T))_{0 \leq i \leq n-K_{n}}$ and  $(\oH^{n}_{i}(T))_{0 \leq i \leq n-K_{n}}$ with each alternating tree $T$
with $n-K_n$ black vertices.

We first give the analogue of  Theorem~\ref{thm:cvXexc} in the setting $c=\infty$.
In the following statement, as in Section~\ref{sec:deflam},
$\mathbbm{e}$ is the Brownian excursion.
We also use the notation of Lemma~\ref{lem:OnBlockLabels} and Section~\ref{ss:coding}.
 In particular,
 we denote by  $(v^{\bullet,n}_{i})_{0 \leq i < n-K_{n}}$ are the black vertices $\tTs_n$, listed in lexicographical order,
 and, for $0 \leq i < n-K_{n}$,
 we let $\ell^{\bullet,n}_{i}$ be  the number of black corners 
 branching on the right of  $\llbracket \emptyset, v^{\bullet,n}_{i} \llbracket$ for $0 \leq i < n-K_{n}$.
\begin{proposition}
\label{prop:cvBH}
Assume that $\tfrac{K_{n}}{\sqrt{n}} \rightarrow \infty$ and $ \tfrac{K_{n}}{n}  \rightarrow \gamma \in [0,1)$ as $n \rightarrow \infty$. Set $\Dcn=\sqrt{ (\scn)^{2} K_{n} }$. Then
\[  \left(  \frac{\oB^{n}_{\lfloor u (n-K_{n}) \rfloor}(\tTs_{n})}{ \Dcn} : 0\leq u \leq 1\right)  \quad \mathop{\longrightarrow}^{(d)}_{n \rightarrow \infty} \quad\mathbbm{e}.\]
In addition, if $\gamma>0$, then
\[  \sup_{0 \leq u \leq 1} \left|   \frac{\oH^{n}_{ \lfloor u (n-K_{n}) \rfloor }(\tTs_{n})}{K_{n}}-u\right|  \quad \mathop{\longrightarrow}^{(\P)}_{n \rightarrow \infty} \quad 0 \qquad \textrm{and} \qquad \frac{1}{n} \max_{1 \leq i \leq n} \ell^{\bullet,n}_{i}  \quad \mathop{\longrightarrow}^{(\P)}_{n \rightarrow \infty} \quad 0.\]
\end{proposition}
The condition $ \tfrac{K_{n}}{n}  \rightarrow \gamma \in (0,1)$ means the both the number of white and black vertices have a positive proportion in $\tilde{\Ts}_{n}$. The proof of Proposition~\ref{prop:cvBH}, very similar in spirit to that of Theorem~\ref{thm:cvXexc}, is postponed  to the end of Section~\ref{ssec:technical}.

In the next Section~\ref{sec:cvBT}, we develop deterministic tools that allow us to show that Proposition~\ref{prop:cvBH} implies Theorem~\ref{thm:cvlam} (ii).

\subsubsection{Convergence to the Brownian triangulation}
\label{sec:cvBT}
As in Section~\ref{ssec:cv_pos}, for every $n \geq 1$, we let $(t_{1}^{(n)}, \ldots, t_{n-1}^{(n)}) \in \mathfrak{M}_{n}$  be a minimal factorization. If $1 \leq K_{n} \leq n-1$, we set
\[  \FFF_{n}= \FFF \big( t_{1}^{(n)}, \ldots, t_{K_{n}}^{(n)}\big) \qquad \text{and} \qquad \PPP_{n}= \PPP\big(  t_{1}^{(n)}  t_{2}^{(n)}  \cdots t_{K_{n}}^{(n)}\big)\]
and we denote by $T_{n}= \mathcal{T}( \PPP_{n})$   the tree coding the non-crossing partition  $\PPP_{n}$ as in Section~\ref{sec:coding}.

Our goal is to prove some convergence result for the pair $\big(\dFn , \dPn  \big)$,
but it turns out to be enough to consider 
 $ \dPn$, as ensured by the following lemma.
  
 \begin{lemma}
 \label{lem:reduction_FtoK}
  Let $Z \in {\mathbb{C}}([0, 1], \R)$ be a continuous function with $Z_{0}=Z_{1}=0$ satisfying \ref{C1}.
  We assume that $\dPn \rightarrow   L(Z)$ as $n$ tends to infinity.
Then we also have  $\dFn \rightarrow   L(Z)$.
 \end{lemma}
\begin{proof}
Since the set of all geodesic laminations of $\overline{\D}$ equipped with Hausdorff topology is compact, we may assume that the convergence $\big( \dFn, \dPn  \big)   \rightarrow   (K,L(Z))$ holds, where $K$ is a lamination, and we aim at showing that $K=L(Z)$. By maximality of $L(Z)$, it is enough to show that $L(Z) \subset K$. 
Consider a chord  $\big[e^{-2\textrm{i} \pi s},e^{-2\textrm{i} \pi t}\big]$ of $L(Z)$ (with $s \neq t$),
we want to show that $\big[e^{-2\textrm{i} \pi s},e^{-2\textrm{i} \pi t}\big] \subset K$.

By \eqref{eq:isolated}, for every $\epsilon>0$, we can find another chord  $\big[e^{-2\textrm{i} \pi s'},e^{-2\textrm{i} \pi t'}\big] \subset L(Z)$ such that  $0<|s-s'| \leq \epsilon$ and $0<|t-t'| \leq \epsilon$. 
By symmetry, we can assume that  $s<s'<t'<t$.
The following argument having some geometric flavor,
looking at Figure~\ref{fig:region} while reading the proof might help.

These two chords form a region of the disk denoted by $R$. 
Again by \eqref{eq:isolated},  there exists a chord $\big[e^{-2\textrm{i} \pi u},e^{-2\textrm{i} \pi v}\big] \subset L(Z)$ such that $s<u<s'$ and $t'<v<t$.
This implies that for each $\eta>0$ small enough and $n$ larger to some threshold $n_0(\eta)$,
we can find a block $\dot{B}_{n}$ of $\dPn$ such that:
\begin{itemize}
  \item $\dot{B}_{n}$ has a chord with endpoints at distance at most $\eta$ of
    $e^{-2\textrm{i} \pi u}$ and $e^{-2\textrm{i} \pi v}$;
  \item moreover, all vertices of $\dot{B}_{n}$ are in the $\eta$ enlargement
    of one of the arcs $(e^{-2\textrm{i} \pi s},e^{-2\textrm{i} \pi s'})$ or 
    $(e^{-2\textrm{i} \pi t},e^{-2\textrm{i} \pi t'})$ (we denote arcs with parentheses to
    distinguish them from chords).
\end{itemize}
We recall that $\dFn$ is obtained from $\dPn$ by replacing each polygon 
by a spanning tree of its vertices. It is clear, that for the block $B_n$
with the above property, such a spanning tree must contain a chord that 
has one extremity on each side, i.e. one in the $\eta$ enlargement
    of $(e^{-2\textrm{i} \pi s},e^{-2\textrm{i} \pi s'})$
    and one in the $\eta$ enlargement
    of $(e^{-2\textrm{i} \pi t},e^{-2\textrm{i} \pi t'})$.
Therefore $\dFn$ contains a chord at distance $\eps+\eta$
of our original chord $\big[e^{-2\textrm{i} \pi s},e^{-2\textrm{i} \pi t}\big]$.
Since this is true for all $\eps$ and $\eta$ sufficiently small,
 this implies that  $\big[e^{-2\textrm{i} \pi s},e^{-2\textrm{i} \pi t}\big] \subset K$ and shows that $K=L(Z)$.   
   \end{proof}

\begin{figure}[ht]
  \begin{center}
  \includegraphics[scale=0.8]{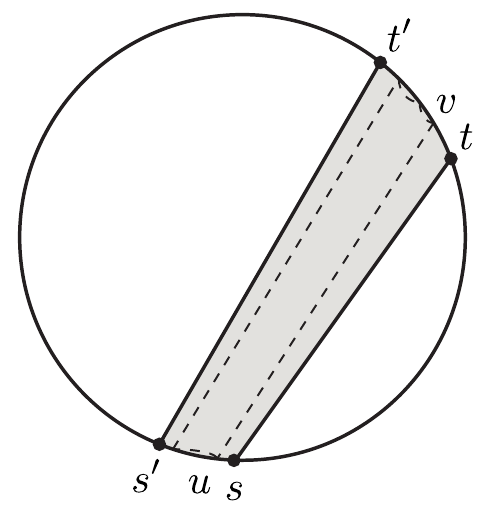}
  \caption{\label{fig:region} An illustration of the second part of the proof. The region $R$ is in gray, 
  the block $\dot{B}_{n}$ is in dashed lines.
  This block is entirely contained in a small enlargement of $R$
  and has at least one chord joining a small enlargement of the $(e^{-2\textrm{i} \pi s},e^{-2\textrm{i} \pi s'})$ arc
  with that of the $(e^{-2\textrm{i} \pi t},e^{-2\textrm{i} \pi t'})$ arc.
  No matter which spanning tree of the vertices of $\dot{B}_{n}$ is present in $\dFn$
  it will contain a chord close to $\big[e^{-2\textrm{i} \pi s},e^{-2\textrm{i} \pi t}\big]$.}
  \end{center}
  \end{figure}
  Let us point out that, although the above lemma is similar in spirit to Lemma~\ref{lem:proche}
  (both show that $\dPn$ and $\dFn$ are in some sense close),
  the arguments are quite different. 
  One the one hand, Lemma~\ref{lem:PasDeuxGrandsSousArbres}, which is the base of Lemma~\ref{lem:proche},
  has certainly no analogue when there is a constant proportion of vertices of both colors in the tree.
  On the other hand, here, we need to assume the convergence of $\dPn$ to some lamination
  encoded by a {\em continuous function} satisfying \ref{C1} to prove that it is close to $\dFn$.
  \medskip

 The gist of the proof of Theorem~\ref{thm:cvlam} (ii) is contained in the next result.
 The notation is the same as in Proposition~\ref{prop:cvBH}.
 
\begin{proposition}
  \label{prop:cvLamDetermBrownien} Let $Z \in {\mathbb{C}}([0, 1], \R)$ be a continuous function such that $Z_{0}=Z_{1}=0$ satisfying \ref{C1}. For every $n \geq 1$, let $(t_{1}^{(n)}, \ldots, t_{n-1}^{(n)}) \in \mathfrak{M}_{n}$  be a minimal factorization, and denote by $T_{n}= \mathcal{T}(  \PPP_{n})$ the tree coding the non-crossing partition $  \PPP_{n}=\PPP(  t_{1}^{(n)}, \ldots, t_{K_{n}}^{(n)})$. Assume that the following two properties hold.
  \begin{enumerate}[label=\color{blue}(D\arabic*)]
    \setcounter{enumi}{-1}
\item\label{D0} We have $\frac{K_{n}}{\sqrt{n}} \rightarrow \infty$ and $ \tfrac{K_{n}}{n} \rightarrow \gamma \in [0,1)$ as $n \rightarrow \infty$.
\item\label{D1} There exists a sequence $\Dcn \rightarrow \infty$ such that
\[  \left(  \frac{\oB^{n}_{\lfloor u (n-K_{n}) \rfloor} (T_{n})}{ \Dcn} : 0\leq u \leq 1\right)  \quad \mathop{\longrightarrow}_{n \rightarrow \infty} \quad Z.\]
\item\label{D2} If $\gamma>0$, we require that
\[ \sup_{0 \leq u \leq 1} \left|   \frac{\oH^{n}_{ \lfloor u (n-K_{n}) \rfloor  }(T_{n})}{K_{n}}-u\right|  \quad \mathop{\longrightarrow}_{n \rightarrow \infty} \quad 0 \qquad \textrm{and} \qquad \frac{1}{n} \max_{0 \leq i < n-K_{n}} \ell^{\bullet,n}_{i}    \quad \mathop{\longrightarrow}_{n \rightarrow \infty} \quad 0.\]
\end{enumerate}
Then we have the  convergence
\begin{equation}
   \dot{\PPP}_{n}    \quad \mathop{\longrightarrow}_{n \rightarrow \infty} \quad L(Z)
\label{eq:cv_FP_SkorokhodBrownien}
\end{equation}
in the Hausdorff topology.
\end{proposition}

Before proving Proposition~\ref{prop:cvLamDetermBrownien}, let us explain how Theorem~\ref{thm:cvlam} (ii)  follows from a combination of Propositions~\ref{prop:cvBH} and \ref{prop:cvLamDetermBrownien} and Lemmas~\ref{lem:reductionKn} and \ref{lem:reduction_FtoK}.

\begin{proof}[Proof of  Theorem~\ref{thm:cvlam} (ii)]
First assume that $\frac{K_{n}}{\sqrt{n}} \rightarrow \infty$ as $n \rightarrow \infty$, and that $K_{n} \leq \frac{n}{2}$ for every $n \geq 1$. Up to extraction, we can suppose that $\tfrac{K_{n}}{n} \rightarrow \gamma \in [0,\tfrac{1}{2}]$. By Skorokhod's representation theorem, we may also suppose that the convergence in  Proposition~\ref{prop:cvBH} holds almost surely. Then, by Proposition~\ref{prop:cvLamDetermBrownien}, the convergence 
\[ \dPKn    \quad \mathop{\longrightarrow}_{n \rightarrow \infty} \quad \mathbf{L}_{\infty},\]
holds almost surely. In turn, by Lemma~\ref{lem:reduction_FtoK}, this implies that the joint convergence 
 \[ \left( \dFKn , \dPKn \right)   \quad \mathop{\longrightarrow}_{n \rightarrow \infty} \quad ( \mathbf{L}_{\infty},\mathbf{L}_{\infty})\]
holds almost surely, and the desired conclusion follows.

Now assume that $\frac{n-K_{n}}{\sqrt{n}} \rightarrow \infty$ and $K_{n} > \frac{n}{2}$ as $n \rightarrow \infty$. 
The previous paragraph tells us that
\[   \dot{\mathcal{P}}^{(n)}_{n-K_{n}-1}  \quad \mathop{\longrightarrow}^{(d)}_{n \rightarrow \infty} \quad \mathbf{L}_{\infty}.\]
Lemma~\ref{lem:reductionKn} entails that the convergence $\hat{\mathcal{P}}^{(n)}_{K_{n}} \rightarrow  \mathbf{L}_{\infty}$ holds in distribution. As a consequence,  since $\hat{\mathcal{P}}^{(n)}_{K_{n}}$ is obtained from  $\dot{\mathcal{P}}^{(n)}_{K_{n}}$ by adding points of $\S$, any sub-sequential distributional limit of $\dot{\mathcal{P}}^{(n)}_{K_{n}}$ must contain all the non-trivial chords of $\mathbf{L}_{\infty}$, and by \eqref{eq:adherence} must be equal to $\mathbf{L}_{\infty}$. In other words, the convergence $\dot{\mathcal{P}}^{(n)}_{K_{n}} \rightarrow  \mathbf{L}_{\infty}$ holds in distribution, and  another application of Lemma~\ref{lem:reduction_FtoK} implies that  the joint convergence 
 \[ \left( \dFKn,   \dPKn \right)   \quad \mathop{\longrightarrow}^{(d)}_{n \rightarrow \infty} \quad ( \mathbf{L}_{\infty},\mathbf{L}_{\infty})\]
holds in distribution. This completes the proof.
\end{proof}

\begin{proof}[Proof of Proposition~\ref{prop:cvLamDetermBrownien}]
Since the set of all geodesic laminations of $\overline{\D}$ equipped with Hausdorff topology is compact, up to extraction, we may therefore assume that  $ \dot{\PPP}_{n}$  converges to a geodesic lamination $K$ of $\overline{\D}$ and we aim at proving that $K=L(Z)$. 
By maximality, it is enough to check that $ L(Z) \subseteq K$.
By \eqref{eq:adherence},  it suffices to show that if $s \sim^{Z} t$  and if $s$ nor $t$ are local minima of $Z$, with $0<s < t<$, then $\big[e^{-2\textrm{i} \pi s},e^{-2\textrm{i} \pi t}\big] \subset K$. 
Using the convergence \ref{D1}, similarly as in Eq.~\eqref{eq:injn}, we can find black vertices
$v^{\bullet,n}_{i_{n}}$ and $v^{\bullet,n}_{j_{n}}$ in $T_{n}$ 
such that:
\begin{itemize}
\item $i_{n}/(n-K_{n}) \rightarrow s$  and $j_{n}/(n-K_{n}) \rightarrow t$;
\item all the black vertices between $v^{\bullet,n}_{i_{n}}$ and $ v^{\bullet,n}_{j_{n}}$
  in the lexicographical order are descendants of  $v^{\bullet,n}_{i_{n}}$.
\end{itemize}
  Denote by  $x^{\bullet,n}_{i_{n}}$ (resp.~$y ^{\bullet,n}_{i_{n}}$)  
  the label of the first (resp.~the last) corner of $v^{\bullet,n}_{i_{n}}$ visited by the contour sequence. We claim that 
  \begin{equation}
\label{eq:cvij} \frac{x^{\bullet,n}_{i_{n}}}{n}  \quad \mathop{\longrightarrow}_{n \rightarrow \infty} \quad s, \qquad \frac{y ^{\bullet,n}_{i_{n}}}{n}  \quad \mathop{\longrightarrow}_{n \rightarrow \infty} \quad t.
\end{equation}
Since $[e^{-2\pi\, \textrm{i} {x^{\bullet,n}_{i_{n}}}/{n}}, e^{-2\pi\, \textrm{i} {y ^{\bullet,n}_{i_{n}}}/{n}}] \subset \dot{\PPP}_{n}$, this will then imply $\big[e^{-2\textrm{i} \pi s},e^{-2\textrm{i} \pi t}\big] \subset K$ and the proof will be complete.

\paragraph*{First case: $\gamma=0$.} By  Lemma~\ref{lem:OnBlockLabels}  (iii), 
we have $i_{n} \leq x^{\bullet,n}_{i_{n}} \leq i_{n}+ K_{n}+1$ and
\[x^{\bullet,n}_{i_{n}}+ (j_n-i_n-1) \leq y^{\bullet,n}_{i_{n}} \leq x^{\bullet,n}_{i_{n}}+(i_n-j_n-1) +2 K_{n}+2,\]
since the number of black descendants of $v^{\bullet,n}_{i_{n}}$ is $j_{n}-i_{n}-1$.
But we have ${K_{n}}/{n} \rightarrow 0$, $i_{n}/(n-K_{n}) \rightarrow s$  and $j_{n}/(n-K_{n}) \rightarrow t$, 
which implies \eqref{eq:cvij}.

\paragraph*{Second case: $\gamma>0$.} In this case, the number of white vertices is no longer negligible, and we use the precise identity of   Lemma~\ref{lem:OnBlockLabels} (ii). Recalling that $k_j$ is the number of children of $v^{\bullet,n}_j$
and that $\ell_{i_{n}}^{\bullet,n}$ denotes the number of black corners  
branching on the right of  $\llbracket \emptyset, v^{\bullet,n}_{i_{n}} \llbracket$,
we have
 \begin{equation}
 \label{eq:est1}x^{\bullet,n}_{i_{n}}= i_{n}+\sum_{j=0}^{i_{n}-1} k_{j} - \ell_{i_{n}}^{\bullet,n}=i_{n}+\oH^{n}_{i_{n}}(T_{n})- \ell_{i_{n}}^{\bullet,n}.
 \end{equation}
Since $i_{n}/(n-K_{n}) \rightarrow s$ and $K_{n} \sim \gamma n$,  using \ref{D2}, we get
\begin{equation}
\label{eq:est2}i_{n}+ \oH^{n}_{i_{n}}(T_{n})+ \ell_{i_{n}}^{\bullet,n}=  s(1-\gamma) n + s \gamma n + o(n)=sn+o(n).
\end{equation} 
This proves that $\tfrac{x^{\bullet,n}_{i_{n}}}{n} \rightarrow s$.
 
 In order to show that $\tfrac{y ^{\bullet,n}_{i_{n}}}{n} \rightarrow t$, note that the number of black descendants of $v^{\bullet,n}_{i_{n}}$ is again $j_{n}-i_{n}-1$ while its number of white descendants is $\oH^{n}_{j_{n}}(T_{n})-\oH^{n}_{i_{n}}(T_{n})$. Hence, by Lemma~\ref{lem:OnBlockLabels} (i) and using \ref{D2}, we get
 \[y ^{\bullet,n}_{i_{n}}=x ^{n}_{i_{n}}+\big(j_{n}-i_{n}-1\big)+\big(\oH^{n}_{j_{n}}(T_{n})-\oH^{n}_{i_{n}}(T_{n})\big)=j_{n}+\oH^{n}_{j_{n}}(T_{n})- \ell_{i_{n}}^{\bullet,n}-1= tn+o(n),\]
where the last equality is obtained exactly as before. This completes the proof.
\end{proof}

\subsubsection{Proof of the technical results}
\label{ssec:technical}

We now turn to the proof of Proposition~\ref{prop:cvBH}. 
The strategy is similar to the case $c>0$: we first establish convergence of unconditioned processes using local limit theorems, then lift them to the bridge version by absolute continuity and finally conclude by using a Vervaat transform.

 As before, we let 
$\mucn$ and $\mubn$ be the offspring distributions defined in Lemma~\ref{lem:exist} (i), and we let $\Scn_{k}$, $\Sbn_{k}$ denote respectively the sum of $k$ i.i.d.~random variables distributed according to $\mucn,\mubn$.

\paragraph*{Local limit theorems.} We first state the analogue of Lemmas~\ref{lem:ll1} and \ref{lem:ll2} in our new regime.
First set
\[\DbnN \coloneqq \sqrt{ (\sbn)^{2}  N}, \qquad \Dcn \coloneqq\sqrt{ (\scn)^{2} K_{n} },\]
and $\Dbn \coloneqq D^{\bullet}_{n,n}= \sqrt{ (\sbn)^{2}  n}$.

\begin{lemma}
\label{lem:locallimitbis}
Assume that $\tfrac{K_{n}}{\sqrt{n}} \rightarrow  \infty$ and $\tfrac{K_{n}}{n}   \rightarrow  \gamma \in [0,1)$.
For every $u \in (0,1]$,
\begin{equation}
\label{eq:ext1}\sup_{un \leq N \leq n} \sup_{k \in \Z  } \left|  \DbnN  \cdot \Pr{\Sbn_N=k} 
  -    \frac{1}{\sqrt{2 \pi}}\exp \left(-   \frac{1}{2}\left( \frac{k-  N \cdot  \frac{K_{n}+1}{n-K_{n}}}
  {  \DbnN  } \right)^{2}\right)\right|  \quad \mathop{\longrightarrow}_{n \rightarrow \infty} \quad 0
\end{equation} 
and
\begin{equation}
  \label{eq:ext2} \hspace{-3.5mm} \sup_{|j| \leq K_{n}^{3/4}} \sup_{k \in \Z} \left|  \Dcn \cdot \Pr{\Scn_{ uK_{n}  + j }=k} -  \frac{1}{\sqrt{2 \pi u}} \exp \left(  - \frac{1}{2u} \left( \frac{k- (uK_{n}+j) \frac{n-K_{n}}{K_{n}+1}}{\Dcn} \right)^{2} \right)\right|  \ \mathop{\longrightarrow}_{n \rightarrow \infty} \ 0.
\end{equation}
\end{lemma}

\begin{proof}
  We only give the main steps, since the structure of the proof is very similar to those of Lemmas~\ref{lem:ll1} and \ref{lem:ll2}. One separately treats two cases. The first one is $\gamma>0$. In this case, $\bbn$ and $\bcn$ converge to values belonging to $(0,1/e)$. In particular, there are no singularities, and one may follow the standard proof of the local limit theorem (see e.g.~\cite[Theorem 3.5.2]{Dur10}). The second case is $\gamma=0$, in which case  $ \Dbn \sim \sqrt{K_{n}}$ and $\Dcn \sim n^{3/2} /K_{n}$. The first estimate \eqref{eq:ext1} in contained in Lemma~\ref{lem:ll1}. For the second estimate, the main idea is to write, as in the proof of Lemma~\ref{lem:phib},
\[\phicn(t)=1+\textrm{i}\E[\Scn_{1}]t- \E[(\Scn_{1})^{2}] \frac{t^{2}}{2}+r^{n}(t),\]
where $|r^{n}(t)| \leq |t|^{3}  \E[(\Scn_{1})^{3}]/6$ for every $n \geq 1$ and $t \in \R$. In particular, by using the expansion of $F(1/e-z)$ around $z=0$, one sees that $ \E[(\Scn_{1})^{3}] \sim 3 (n/K_{n})^{5}$, so that for every $n$ sufficiently large and $t \in \R$, $|r^{n}(t/\Dcn)| \leq |t|^{3} \tfrac{\sqrt{n}}{K_{n}^{2}}$, and the key point is that $ \tfrac{\sqrt{n}}{K_{n}^{2}}= o(\tfrac{1}{K_{n}})$ since $K_{n}/\sqrt{n} \rightarrow \infty$. From this follows an analogue of Lemma~\ref{lem:technique}, namely that there exist constants $A_{0},\epsilon,\kappa>0$ such that  for every $n$ sufficiently large and for every $|t|$ in $(A_{0}, \epsilon \Dcn)$, we have $|\phicn( \tfrac{t}{\Dcn})| \leq e^{- \kappa  \frac{t^{2}}{K_{n}}}$. This then allows us to follow the same steps as the proof of Lemma~\ref{lem:ll2}. We leave the details to the reader.
\end{proof}

\paragraph*{Convergence of the unconditioned processes.}  Recall that $(H^{n}_{k},B^{n}_{k})_{k \geq 1}$ denotes a sequence of i.i.d.~random variables with distribution given by: for $i, j \geq 0$,
\[ \Prb{H^{n}=i,B^{n}=j}=\mubn(i) \,\Prb{\Scn_{i}=j}.\]
Also, for $i \geq 0$, we set  $\oH^{n}_{i}=H^{n}_{1}+H^{n}_{2}+ \cdots+H^{n}_{i}$ and  $\oB^{n}_{i}=B^{n}_{1}+B^{n}_{2}+ \cdots+B^{n}_{i}-i$. 
To simplify notation, set \[\hH^{(n)}_{u}=\frac{\oH^{n}_{ \lfloor u (n-K_{n}) \rfloor } - u (K_{n}+1) }{\Dbn}, \qquad \hB^{(n)}_{u}= \frac{\oB^{n}_{ \lfloor u (n-K_{n}) \rfloor }  - \frac{n-K_{n}}{K_{n}+1} \hH^{(n)}_{u}  }{\Dcn}.\]

   \begin{lemma} Assume that $\frac{K_{n}}{\sqrt{n}} \rightarrow   \infty$ and   $\tfrac{K_{n}}{n}  \rightarrow \gamma \in [0,1)$ as $n \rightarrow \infty$.
  The following convergence holds in distribution in $\mathbb{D}([0,1],\R^{2})$:
  \[ \left( \hH^{(n)}_{u} ,  \hB^{(n)}_{u} \right)_{0 \leq u \leq 1}  \quad \mathop{\longrightarrow}_{n \rightarrow \infty} \quad  (W_{u},X_{u})_{0 \leq u \leq 1},\]
 where  $W,X$ are two independent standard Brownian motions.
  \end{lemma}
  
As the proof will show, the appearance of the normalized quantity $\hH^{(n)}_{u}$ is important in the definition of $\hB^{(n)}_{u}$.

\begin{proof}
The proof is similar to that of Lemma~\ref{lem:cvnoncond}. In virtue of \cite[Theorem 16.14]{Kal02}, it is enough to check that the one-dimensional convergence holds for $u=1$.  For fixed $x,y \in \R$, by Lemma~\ref{lem:magique},
\begin{multline*}
\Pr{\oH^{n}_{  n-K_{n} }= \lfloor x \Dbn+K_{n} \rfloor ,\oB^{n}_{  n-K_{n}}= \lfloor y  \Dcn + \frac{n-K_{n}}{K_{n}+1} \lfloor x \Dbn \rfloor \rfloor} \\
=\Pr{\Sbn_{ n-K_{n} }=  \lfloor x \Dbn +K_{n} \rfloor} \Pr{\Scn_{ \lfloor x \Dbn +K_{n} \rfloor }= n-K_{n}   + \left\lfloor  y  \Dcn + \frac{n-K_{n}}{K_{n}+1} \lfloor x \Dbn \rfloor \right\rfloor}.  
\end{multline*}
By \eqref{eq:ext1} and \eqref{eq:ext2}, as $n \rightarrow \infty$, this quantity is asymptotic to $ \tfrac{1}{ \Dbn  }    p (x)  \cdot  \tfrac{1}{\Dcn} p (y)$,
where $p(x)= \tfrac{1}{\sqrt{2\pi}} e^{- {x^{2}}/{2}}$ denotes the Gaussian density.
It is then standard (see e.g.~\cite[Theorem 7.8]{Bil68}) that this implies that the convergence
  \[ \left( \hH^{(n)}_{1} ,  \frac{\oB^{n}_{ n-K_{n} }  - \frac{n-K_{n}}{K_{n}+1} \hH^{(n)}_{1}  }{\Dcn} \right)  \quad \mathop{\longrightarrow}_{n \rightarrow \infty} \quad  (W_{1},X_{1}),\]
  holds in distribution. This completes the proof.
  \end{proof}  

\paragraph*{Convergence of the bridge version.} The proof of the following lemma is similar to that of Proposition~\ref{prop:cvbridge} and is left to the reader.

\begin{lemma}Assume that $\frac{K_{n}}{\sqrt{n}} \rightarrow   \infty$ and   $\tfrac{K_{n}}{n}  \rightarrow \gamma \in [0,1)$ as $n \rightarrow \infty$. Conditionally given the event
$ \{\oH^{n}_{n-K_{n}}=K_{n}+1, \oB^{n}_{n-K_{n}}=-1 \}$, the following  convergence in distribution holds jointly
\[  \left(  \hH^{(n)}_{u},  \hB^{(n)}_{u} \right)_{0 \leq u \leq 1}  \quad \mathop{\longrightarrow}^{(d)}_{n \rightarrow \infty} \quad (W^{\mathrm{br}},\Xbr),\]
where $W^{\mathrm{br}}$ and $\Xbr$ are two independent Brownian bridges.
\end{lemma}

\begin{corollary}\label{cor:cvjointe}
Assume that, as $n\rightarrow \infty$, we have $\frac{K_{n}}{\sqrt{n}} \rightarrow   \infty$ and   $\tfrac{K_{n}}{n}  \rightarrow \gamma$
with $\gamma$ in $[0,1)$. Conditionally given the event
$ \{\oH^{n}_{n-K_{n}}=K_{n}+1, \oB^{n}_{n-K_{n}}=-1 \}$, the following  convergence in distribution holds jointly
\[\left(  \frac{\oB^{n}_{ \lfloor u (n-K_{n}) \rfloor }}{\Dcn} \right)_{0 \leq u \leq 1}  \ \mathop{\longrightarrow}^{(d)}_{n \rightarrow \infty} \  \Xbr,  \qquad  \sup_{0 \leq u \leq 1} \left|   \frac{\oH^{n}_{ \lfloor u (n-K_{n}) \rfloor }}{K_{n}}-u\right|  \ \mathop{\longrightarrow}_{n \rightarrow \infty} \ 0.\]
\end{corollary}

\begin{proof}
  Since $  \hH^{(n)}_{u}$
  converges in distribution and since $ \frac{1}{\Dcn} \cdot \frac{n-K_{n}}{K_{n}+1} \rightarrow 0$, we have that
\[  \frac{n-K_{n}}{\Dcn(K_{n}+1)} \cdot \sup_{0 \leq u \leq 1}| \hH^{(n)}_{u}| \rightarrow 0.\]
Therefore, $\frac{\oB^{n}_{ \lfloor u (n-K_{n}) \rfloor }}{\Dcn}$
has the same limit as $\hB^{(n)}_{u}$, namely $\Xbr$.

For the second convergence, we simply write
\[\frac{\oH^{n}_{ \lfloor u (n-K_{n}) \rfloor }}{K_{n}}-u= \frac{\Dbn}{K_{n}} \cdot \hH^{(n)}_{u}\]
and the result follows from the fact that $ \frac{\Dbn}{K_{n}} \rightarrow 0$.
\end{proof}

\paragraph*{Convergence of the excursion version.} 

Using the (bivariate) Vervaat transform and Lemma~\ref{lem:vervaat},
we can now state an invariance principle for $\oB^{n}$ and $\oH^{n}$
conditionally given the event
\[\mathcal C^+=\{\oH^{n}_{n-K_{n}}=K_{n}+1,\, \oB^{n}_{n-K_{n}}=-1,\, \oB^{n}_i \ge 0\text{ for all }i <n-K_n \}.\]
From Lemma~\ref{lem:codeRW},
under this conditioning $(\oB^{n}, \oH^{n})$ has the same distribution
as $(\oB(\Ts_n),\oH(\Ts_n))$,
where $\Ts_n$ is an alternating BGW tree with offspring distributions $\mubn$ and $\mucn$
conditioned on having $n-K_n$ black vertices and $K_n+1$ white vertices.
We will thus state our invariance principle in terms of $(\oB(\Ts_n),\oH(\Ts_n))$.

\begin{proposition} 
\label{prop:final}
Assume that $\frac{K_{n}}{\sqrt{n}} \rightarrow   \infty$ and   $\tfrac{K_{n}}{n}  \rightarrow \gamma \in [0,1)$ as $n \rightarrow \infty$. The following convergences hold jointly in distribution \[  \left(  \frac{\oB_{\lfloor u (n-K_{n}) \rfloor}(\Ts_{n})}{ \Dcn} : 0\leq u \leq 1\right)  \ \mathop{\longrightarrow}^{(d)}_{n \rightarrow \infty} \ \mathbbm{e}, \qquad  \sup_{0 \leq u \leq 1} \left|   \frac{\oH_{ \lfloor u (n-K_{n}) \rfloor  }(\Ts_{n})}{K_{n}}-u\right|  \ \mathop{\longrightarrow}_{n \rightarrow \infty} \ 0.\]
\end{proposition}
This proposition follows from Corollary~\ref{cor:cvjointe},
in the same way that Theorem~\ref{thm:cvXexc} (i) was established using Proposition~\ref{prop:cvbridge}.
Again, we leave details to the reader.

\medskip

\paragraph*{End of the proof of Proposition~\ref{prop:cvBH}}

We consider separately the cases $\gamma=0$ and $\gamma>0$.
The first one is effortless.

\begin{proof}[Proof of Proposition~\ref{prop:cvBH} when  $\tfrac{K_{n}}{\sqrt{n}} \rightarrow \infty$ and $ \tfrac{K_{n}}{n}  \rightarrow 0$ as $n \rightarrow \infty$] In this case, the conclusions of Lemma~\ref{lem:root} are true with the same proof, where the occurrences of the quantity $p(0) q_{1}(c)$ should be replaced by $p(0)^{2}$, and where instead of using the local limit Lemmas~\ref{lem:ll1} and \ref{lem:ll2}, one uses \eqref{eq:ext1} and \eqref{eq:ext2}.  This shows that as before,   $\Ts_{n}$ and $\tTs _{n}$ can be coupled so that $\Ts_{n}=\tTs_{n}$ with probability tending to $1$.  Proposition~\ref{prop:cvBH} then readily follows from Proposition~\ref{prop:final}.
\end{proof}

We now concentrate on the case  $\tfrac{K_{n}}{\sqrt{n}} \rightarrow \infty$ and $ \tfrac{K_{n}}{n}  \rightarrow \gamma \in (0,1)$ as $n \rightarrow \infty$, which requires new ideas. 
In this setting, we set $m_\bullet:=\frac{\gamma}{1-\gamma}$ and $m_\circ:=\frac{1-\gamma}{\gamma}$,
which are the limits of the means $m_\bullet^n$ and $m_\circ^n$, of $\mubn$ and $\mucn$ respectively.
By an argument similar to that used in Lemma~\ref{lem:exist}, there exist unique parameters $(\ab,\bb,\ac,\bc)$ with $\bb,\bc \in (0,1/e)$
such that the following equations define probability distributions with means $m_\bullet$ and $m_\circ$:
 \begin{equation}
 \label{eq:mubc}\mub(i)= \ab \cdot  (\bb)^{i} \cdot \frac{(i+1)^{i-1}}{i!}  \quad (i \geq 0), \qquad \muc(i)= \ac \cdot  (\bc)^{i} \cdot \frac{(i+1)^{i-1}}{i!}  \quad (i \geq 0).
 \end{equation}
 It is then immediate that $(\abn,\bbn,\acn,\bcn)$ converges as $n \rightarrow \infty$ to the vector $(\ab,\bb,\ac,\bc)$.
 As a consequence, one easily check that the distributions $\mubn$ and $\mucn$
 have {\em uniform exponential moments}, in the sense that
 \begin{equation}
   \label{eq:expmoment} \exists \delta>0,\,  \qquad  \sup_{n \ge 1} \, \sum_{i=0}^{\infty} e^{\delta i} \mubn(i) < \infty 
   \quad \textrm{and} \quad \sup_{n \ge 1}\, \sum_{i=0}^{\infty} e^{\delta i} \mucn(i) < \infty.
 \end{equation}
 This result will be useful later. 
 
We first establish a result which will allow us to replace $\tTs_{n}$ with $\Ts_{n}$,
that is to forget about the specific offspring distribution of the root.
As before, we denote $\tilde{\Ts} ^{n}$ the unconditioned version of $\tTs_n$  
with distribution ${\BGW}^{\widetilde{\mu}_{\bullet}^{n},\mubn,\mucn}$, as defined in \eqref{eq:BGWtilde}. 
Recall that the superscript $n$ indicates that the offspring distributions depend on $n$, and not a conditioning.

\begin{lemma}\label{lem:tilde}The following assertions are satisfied.
\begin{enumerate}
\item[(i)] We have
\[\Pr{\abs{\circ_{\tTs^{n}}}=K_{n}+1, \abs{\bullet_{\tTs^{n}}}=n-K_{n}}   \ \mathop{\sim}_{n \rightarrow \infty} \ \left( \sum_{a=1}^{\infty} a \mub(a-1) \right)\cdot \frac{1}{ 2\pi \sb \sc  \gamma \sqrt{\gamma(1-\gamma)} } \cdot \frac{1}{n^{2}},\]
where $(\sb)^{2}$, $(\sc)^{2}$ denote respectively the variance of $\mub$, $\muc$.
\item[(ii)] The couple $\big(\oH^{n}_{1}(\tTs_{n}),\oB^{n}_{1}(\tTs_{n})\big)$ converges in distribution as $n \rightarrow \infty$.
\item[(iii)] Let  $ \mathscr{F}^{\bullet,n}_{k} =(\Ts^{\bullet,n}_{1}, \ldots, \Ts^{\bullet,n}_{k})$ denote a collection of $k$ i.i.d. BGW trees distributed as  $ \Tsbn$ (with a black root). Then, for fixed $i_{0}$, conditionally given $ \{ \abs{\circ_{\mathscr{F}^{\bullet,n}_{k}}}=K_{n}+1-i_{0},\abs{\bullet_{\mathscr{F}^{\bullet,n}_{k}}}=n-K_{n}-1 \} $, the random variable $ \min_{1 \leq i \leq k} (\abs{\circ_{\mathscr{F}^{\bullet,n}_{k}}}- \abs{\circ_{\Ts^{\bullet,n}_{i}}},\abs{\bullet_{\mathscr{F}^{\bullet,n}_{k}}}- \abs{\bullet_{\Ts^{\bullet,n}_{i}}})$ converges in distribution, where the minimum refers to the lexicographical order on $\Z_{+}^{2}$.
\end{enumerate}
\end{lemma}

\begin{proof} We follow the strategy of \cite[Lemma 5.7]{KM16}.
For $j \geq 0$, let $ \mathscr{F}^{\circ,n}_{j}$ denote a collection of $j$ i.i.d. BGW trees distributed as  $ \Tscn$ (with a white root). By considering the trees (with white roots) grafted on the root of $\tTs^{n}$  and using \eqref{eq:forestc}, write 
\begin{align*}
&\Pr{\abs{\circ_{\tTs^{n}}}=K_{n}+1, \abs{\bullet_{\tTs^{n}}}=n-K_{n}} \\
& \qquad\qquad = \sum_{a=1}^{\infty} \mubn(a-1) \Pr{\abs{ \circ_{ \mathscr{F}^{\circ,n}_{a}}}=K_{n}+1, \abs{\bullet_{ \mathscr{F}^{\circ,n}_{a}}}=n-K_n-1} \\
& \qquad \qquad =\sum_{a=1}^{\infty} a \mubn(a-1) \frac{1}{K_{n}+1} \Pr{\Sbn_{n-K_{n}-1}=K_{n}+1-a} \Pr{\Scn_{K_{n}+1}=n-K_{n}-1}.
\end{align*}
The desired result then follows by combining \eqref{eq:ext1} and \eqref{eq:ext2} with the dominated convergence theorem.

For (ii), fix $a,b \geq 1$. Since $(\oH^{n}_{1}(\tTs_{n}),\oB^{n}_{1}(\tTs_{n}))=(a,b-1)$ if and only if the root of $\tTs_{n}$ has $a$ (white) children, which altogether have $b$ (black) children, using \eqref{eq:forestb},  the quantity $\Prb{\oH^{n}_{1}(\tTs_{n}),\oB^{n}_{1}(\tTs_{n})=(a,b-1)}$ is equal to 
\[
 \frac{\mubn(a-1) \Pr{\Scn_{a}=b}  \Pr{\abs{ \circ_{ \mathscr{F}^{\bullet,n}_{b}}}=K_{n}+1-a, \abs{\bullet_{ \mathscr{F}^{\bullet,n}_{b}}}=n-K_n-1}}{\Pr{\abs{\circ_{\tTs^{n}}}=K_{n}, \abs{\bullet_{\tTs^{n}}}=n-K_{n}}},
\]
where, for $j \geq 1$,  $ \mathscr{F}^{\bullet,n}_{j}$ denotes a collection of $j$ i.i.d. BGW trees distributed as  $ \Tsbn$ (with a black root).
By \eqref{eq:forestb}, the quantity $\Prb{\abs{ \circ_{ \mathscr{F}^{\bullet,n}_{b}}}=K_{n}+1-a, \abs{\bullet_{ \mathscr{F}^{\bullet,n}_{b}}}=n-K_n-1}$ is equal to
\[\ \frac{b}{n-K_{n}-1} \Pr{\Sbn_{n-K_{n}-1}=K_{n}+1-a} \Pr{\Scn_{K_{n}+1-a} =n-K_{n}-1-b}.\]
 By  combining \eqref{eq:ext1} and \eqref{eq:ext2} with (i), we finally get that
 \[\Prb{\oH^{n}_{1}(\tTs_{n}),\oB^{n}_{1}(\tTs_{n})=(a,b-1)}  \quad \mathop{\longrightarrow}_{n \rightarrow \infty} \quad  \frac{\gamma}{1-\gamma} \cdot \frac{\mub(a-1)  b\Pr{S^{\circ}_{a}=b}}{\sum_{i=1}^{\infty} i \mub(i-1)},\]
 where $(S^{\circ}_{i})_{i \geq 1}$ is a random walk with jump distribution $\muc$.
Since $\Es{S^{\circ}_{1}}= \frac{1-\gamma}{\gamma}$, the limiting quantity indeed defines a probability distribution.

For (iii), fix integers $u,v \geq 0$. Observe that for $n$ sufficiently large, the following union is disjoint
\begin{align*}
  & \scalebox{.8}{$\left\{  \min_{1 \leq i \leq k} (\abs{\circ_{\mathscr{F}^{\bullet,n}_{k}}}- \abs{\circ_{\Ts^{\bullet,n}_{i}}},\abs{\bullet_{\mathscr{F}^{\bullet,n}_{k}}}- \abs{\bullet_{\Ts^{\bullet,n}_{i}}})=(u,v), \ \abs{\circ_{\mathscr{F}^{\bullet,n}_{k}}}=K_{n}+1-i_{0},\ \abs{\bullet_{\mathscr{F}^{\bullet,n}_{k}}}=n-K_{n}-1  \right\}$} \\
  & \scalebox{.8}{$ \quad =\bigcup_{i=1}^{k} \left\{ \abs{\circ_{\Ts^{\bullet,n}_{i}}}=\abs{\circ_{\mathscr{F}^{\bullet,n}_{k}}}-u, \abs{\bullet_{\Ts^{\bullet,n}_{i}}}=\abs{\bullet_{\mathscr{F}^{\bullet,n}_{k}}}-v,  \abs{\circ_{\mathscr{F}^{\bullet,n}_{k}}}=K_{n}+1-i_{0},\abs{\bullet_{\mathscr{F}^{\bullet,n}_{k}}}=n-K_{n}-1\right\}. $}
\end{align*}
Hence, by exchangeability, the probability of the latter event is
\[\!\!
   k \Pr{\abs{\circ_{\Ts^{\bullet,n}_{1}}}=K_{n}+1-i_{0}-u, \abs{\bullet_{\Ts^{\bullet,n}_{1}}}=n-K_{n}-1-v}\Pr{\abs{\circ_{\mathscr{F}^{\bullet,n}_{k-1}}}=u,\abs{\bullet_{\mathscr{F}^{\bullet,n}_{k-1}}}=v}\]
Now, using  \eqref{eq:forestb},
\begin{multline*}
\frac{k \Pr{\abs{\circ_{\Ts^{\bullet,n}_{1}}}=K_{n}+1-i_{0}-u, \abs{\bullet_{\Ts^{\bullet,n}_{1}}}=n-K_{n}-1-v}}{\Pr{\abs{\circ_{\mathscr{F}^{\bullet,n}_{k}}}=K_{n}+1-i_{0},\abs{\bullet_{\mathscr{F}^{\bullet,n}_{k}}}=n-K_{n}-1}}\\
 = \frac{ \frac{k}{n-K_{n}-1-v} \Pr{\Sbn_{n-K_{n}-1-v}=K_{n}+1-i_{0}-u} \Pr{\Scn_{K_{n}+1-i_{0}-u} =n-K_{n}-2-v} }{ \frac{k}{n-K_{n}-1} \Pr{\Sbn_{n-K_{n}-1}=K_{n}+1-i_{0} } \Pr{\Scn_{K_{n}+1-i_{0}}=n-K_{n}-2} }.
\end{multline*}
By  \eqref{eq:ext1} and \eqref{eq:ext2}, this quantity tends to $1$ as $n \rightarrow \infty$. By combining Eqs.~\eqref{eq:forestb}, \eqref{eq:ext1} and \eqref{eq:ext2} we also get that
\[\Pr{\abs{\circ_{\mathscr{F}^{\bullet,n}_{k-1}}}=u,\abs{\bullet_{\mathscr{F}^{\bullet,n}_{k-1}}}=v}  \quad \mathop{\longrightarrow}_{n \rightarrow \infty} \quad \Pr{\abs{\circ_{\mathscr{F}^{\bullet}_{k-1}}}=u,\abs{\bullet_{\mathscr{F}^{\bullet}_{k-1}}}=v},\]
 where $ \mathscr{F}^{\bullet}_{k-1}$ denotes a collection of $k-1$ i.i.d.\ alternating BGW trees with offspring distribution $\mub$ and $\muc$ and black roots. 
 
 We conclude that 
 \[\scalebox{.8}{$\Pr{\min_{1 \leq i \leq k} (\abs{\circ_{\mathscr{F}^{\bullet,n}_{k}}}- \abs{\circ_{\Ts^{\bullet,n}_{i}}},\abs{\bullet_{\mathscr{F}^{\bullet,n}_{k}}}- \abs{\bullet_{\Ts^{\bullet,n}_{i}}}) =(u,v) \, \Big| \, \abs{\circ_{\mathscr{F}^{\bullet,n}_{k}}}=K_{n}+1-i_{0},\abs{\bullet_{\mathscr{F}^{\bullet,n}_{k}}}=n-K_{n}-1 }$} \]
converges to $\Pr{\abs{\circ_{\mathscr{F}^{\bullet}_{k-1}}}=u,\abs{\bullet_{\mathscr{F}^{\bullet}_{k-1}}}=v}$ as $n \rightarrow \infty$.
This completes the proof.
\end{proof}

We are now ready to establish Proposition~\ref{prop:cvBH} in the final case.

\begin{proof}[Proof of  Proposition~\ref{prop:cvBH}  when $\tfrac{K_{n}}{\sqrt{n}} \rightarrow \infty$ and $ \tfrac{K_{n}}{n}  \rightarrow \gamma \in (0,1)$.] 
Lemma~\ref{lem:tilde} (ii) and (iii) shows that with probability tending to $1$ as $n \rightarrow \infty$, among all the trees grafted on a (black) grandchild of the root of $\tTs_{n}$, only one has $n-K_{n}+o(n)$ black vertices and $K_{n}+o(n)$ white vertices, while all the others have $o(n)$ vertices. It is therefore sufficient to establish the desired result with $\tTs_{n}$ replaced with $\Ts_{n}$.  The first two convergences in  Proposition~\ref{prop:cvBH}  are then obtained from Proposition~\ref{prop:final}. 

It remains to check that 
\[ \frac{1}{n} \min_{0 \leq i < n-K_{n}} \ell^{\bullet,n}_{i}  \quad \mathop{\longrightarrow}^{(\P)}_{n \rightarrow \infty} \quad 0,\]
where $\ell^{\bullet,n}_{i}$ denotes  the number of black corners  branching on the right of  $\llbracket \emptyset, v^{\bullet,n}_{i} \llbracket$ for $0 \leq i < n-K_{n}$ (recall that $v^{\bullet,n}_{i} $ is the $i$-th black vertex of $\Ts_{n}$ in lexicographical order). 
The proof is divided in several steps: (i) Reduction to a one-type tree (ii) Reduction to a non-conditioned statement (iii) Large deviation estimates.

\paragraph*{Step 1: reduction to a one-type tree.}

%%%%%%%%%%%%%%%%%%%%%%
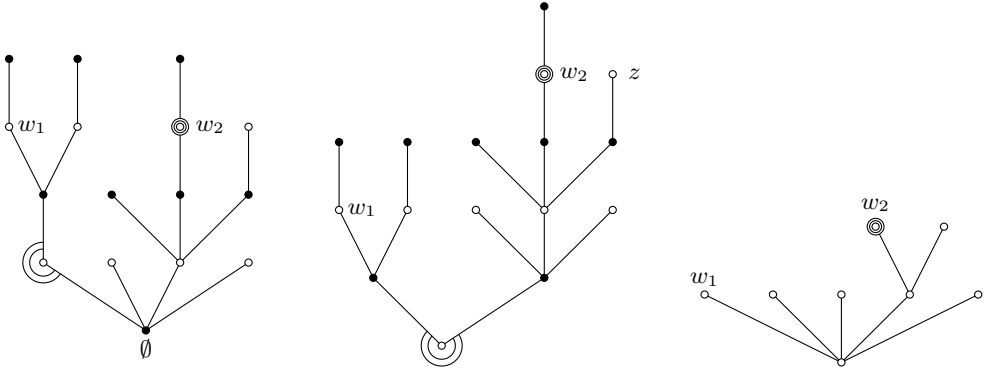
\begin{figure}[t] \centering
\begin{scriptsize}
\begin{tikzpicture}[scale=.9]
\coordinate (0) at (0,0);
	\coordinate (1) at (-1.5,1);
		\coordinate (11) at (-1.5,2);
			\coordinate (111) at (-2,3);
				\coordinate (1111) at (-2,4);
			\coordinate (112) at (-1,3);
				\coordinate (1121) at (-1,4);
	\coordinate (2) at (-.5,1);
	\coordinate (3) at (.5,1);
		\coordinate (31) at (-0.5,2);
		\coordinate (32) at (0.5,2);
			\coordinate (321) at (0.5,3);
				\coordinate (3211) at (0.5,4);
		\coordinate (33) at (1.5,2);
			\coordinate (331) at (1.5,3);
	\coordinate (4) at (1.5,1);
\draw
	(0) -- (1)	(0) -- (2)	(0) -- (3)	(0) -- (4)
	(1) -- (11)
	(11) -- (111) -- (1111)	(11) -- (112) -- (1121)
	(3) -- (31)	(3) -- (32) -- (321) --(3211)	(3) -- (33) -- (331)
;
\draw[fill=black]
	(0) circle (1.5pt)
	(11) circle (1.5pt)
	(31) circle (1.5pt)
	(32) circle (1.5pt)
	(33) circle (1.5pt)
	(1111) circle (1.5pt)
	(1121) circle (1.5pt)
	(3211) circle (1.5pt)
;
\draw[fill=white]
	(1) circle (1.5pt)
	(2) circle (1.5pt)
	(3) circle (1.5pt)
	(4) circle (1.5pt)
	(111) circle (1.5pt)
	(112) circle (1.5pt)
	(321) circle (1.5pt)
	(321) circle (2.5pt)
	(321) circle (3.5pt)
	(331) circle (1.5pt)
	
;

% Labels
\draw
	(321)++(.1,0) node[right] {$w_{2}$}
	(111)  node[right] {$w_{1}$}
	(0) node[below] {$\emptyset$}
;
   \draw [domain=-33.69:-270] plot ({-1.5+0.2*cos(\x)}, {1+0.2*sin(\x)});
   \draw [domain=-33.69:-270] plot ({-1.5+0.3*cos(\x)}, {1+0.3*sin(\x)});

\end{tikzpicture}
\qquad\quad
\begin{tikzpicture}[scale=.9]
\coordinate (0) at (0,0);
	\coordinate (1) at (-1,1);
		\coordinate (11) at (-1.5,2);
			\coordinate (111) at (-1.5,3);
		\coordinate (12) at (-0.5,2);
			\coordinate (121) at (-0.5,3);
	\coordinate (2) at (1.5,1);
		\coordinate (21) at (0.5,2);
		\coordinate (22) at (1.5,2);
			\coordinate (221) at (0.5,3);
			\coordinate (222) at (1.5,3);
				\coordinate (2221) at (1.5,4);
				\coordinate (22211) at (1.5,5);
			\coordinate (223) at (2.5,3);
				\coordinate (2231) at (2.5,4);
		\coordinate (23) at (2.5,2);
\draw
	(0) -- (1)	(0) -- (2)
	(1) -- (11)	(1) -- (12)
	(11) -- (111)
	(12) -- (121)
	(2) -- (21)	(2) -- (22)	(2) -- (23)
	(22) -- (221)	(22) -- (222) -- (2221) --(22211)	(22) -- (223) -- (2231)
;
%
%Racine
   \draw [domain=33.69:-225] plot ({0.2*cos(\x)}, {0.2*sin(\x)});
   \draw [domain=33.69:-225] plot ({0.3*cos(\x)}, {0.3*sin(\x)});
% Labels
\draw
	(11) node[right] {$w_{1}$}
	(2221)++(.1,0) node[right] {$w_{2}$}
	(2231)++(.1,0) node[right] {$z$}

;
\draw[fill=black]
	(1) circle (1.5pt)
	(2) circle (1.5pt)
	(111) circle (1.5pt)
	(121) circle (1.5pt)
	(221) circle (1.5pt)
	(222) circle (1.5pt)
	(223) circle (1.5pt)
	(22211) circle (1.5pt)
;
\draw[fill=white]
	(0) circle (1.5pt)
	(11) circle (1.5pt)
	(12) circle (1.5pt)
	(21) circle (1.5pt)
	(22) circle (1.5pt)
	(23) circle (1.5pt)
	(2221) circle (1.5pt)
	(2221) circle (2.5pt)
	(2221) circle (3.5pt)
	(2231) circle (1.5pt)
;
\end{tikzpicture}
\quad
\begin{tikzpicture}[scale=.9]
\coordinate (0) at (0,0);
	\coordinate (1) at (-2,1);
	\coordinate (2) at (-1,1);
	\coordinate (3) at (0,1);
	\coordinate (4) at (1,1);
		\coordinate (41) at (0.5,2);
		\coordinate (42) at (1.5,2);
	\coordinate (5) at (2,1);
\draw
	(0) -- (1)	(0) -- (2)	(0) -- (3)	(0) -- (4)	(0) -- (5)
	(4) -- (41)	(4)--(42);

\draw
	(1) node[above] {$w_{1}$}
	(41)++(0,0.1) node[above] {$w_{2}$}
;

\draw[fill=white]
	(0) circle (1.5pt)
	(1) circle (1.5pt)
	(2) circle (1.5pt)
	(3) circle (1.5pt)
	(4) circle (1.5pt)
	(5) circle (1.5pt)
	(41) circle (1.5pt)
	(41) circle (2.5pt)
	(41) circle (3.5pt)
	(42) circle (1.5pt)
;
\end{tikzpicture}
\end{scriptsize}
\caption{From left to right: a two-type tree $T$ with black root, the two-type tree $T^{\circ}$ with white root obtained be re-rooting $T$ at its first white corner (which is highlighted), and finally its associated reduced white subtree $\hat{T}^{\circ}$. The  number of black corners branching on the right of $\llbracket \emptyset, w_{2} \llbracket$ in $T$ is $3$, which is   the same as the number of black corners branching on the right of $\llbracket \emptyset, w_{2} \llbracket$ in $T^{\circ}$, and which is less or equal than the number of corners branching to the right of $\llbracket \emptyset, w_{2} \llbracket$ in $\hat{T^{\circ}}$ (which is $4$).  The  number of black corners branching on the right of $\llbracket \emptyset, w_{1} \llbracket$ in $T$ is $6$, which is indeed  the same as the number of black corners branching on the right of $\llbracket \emptyset, w_{1} \llbracket$ in $T^{\circ}$ plus the number of children of $\emptyset$ in $T$, and which is less or equal than the number of corners branching to the right of $\llbracket \emptyset, w_{1} \llbracket$ in $\hat{T^{\circ}}$ (which is $5$) plus one.
}
\label{fig:whiteblack}
\end{figure}

The first idea is to reduce the statement to a one-type tree, a reduced white subtree.  Specifically, we proceed in several elementary steps:
\begin{itemize}
  \item For $0 \le j\le K_n$, we denote $v^{\circ,n}_{j} $ the $j$-th white vertex of $\Ts_{n}$ in lexicographical order
    and $\ell^{\circ,n}_{j}$ the number of black corners  branching on the right of  $\llbracket \emptyset, v^{\circ,n}_{j} \llbracket$.
    If $v^{\circ,n}_{j}$ is the parent of $v^{\bullet,n}_{i}$, we have $\ell^{\circ,n}_{j}=\ell^{\bullet,n}_{i}$.
    We conclude that 
    \[ \max_{0 \leq i < n-K_{n}} \ell^{\bullet,n}_{i}= \max_{0 \leq j \leq  K_{n}} \ell^{\circ,n}_{j}.\]
\item  Denote by $\Ts_{n}^{\circ}$ the tree obtained by re-rooting $\Ts_{n}$ at its first white corner (see Figure~\ref{fig:whiteblack}). Then for every white vertex $w$, the number of black corners branching on the right of $\llbracket \emptyset, w \llbracket$ is the same in $ \Ts_{n}$ and in  $\Ts_{n}^{\circ}$ (except for white vertices which are descendants of the first child of $\emptyset$ in $\Ts_{n}$, for which one has to add the number of children of $\emptyset$ in $\Ts_{n}$). See Figure~\ref{fig:whiteblack} for an example.
\item Denote by $\hat{\Ts_{n}^{\circ}}$ the reduced white-subtree of $\Ts_{n}^{\circ}$. Then for every white vertex $w$, the number of black corners branching on the right of $\llbracket \emptyset, w \llbracket$ in $\Ts_{n}^{\circ}$ is less than or equal to   the number  of  corners branching on the right of $\llbracket \emptyset, w \llbracket$  in  $\hat{\Ts_{n}^{\circ}}$ (except for white vertices which are the descendants of the first child of $\emptyset$ in $\Ts_{n}$, for which one has to add $1$).  See Figure~\ref{fig:whiteblack} for an example. The reason why there may not be equality is that the reduction operation forgets the genealogy of black vertices: for instance, if in Figure~\ref{fig:whiteblack} the vertex $z$ were a younger sibling of $w_{2}$, the reduced white subtree would not have changed, but  the number of black corners branching on the right of $\llbracket \emptyset, w \llbracket$ in $\Ts_{n}^{\circ}$ would have increased.
\end{itemize}

Thanks to the previous observations, it therefore remains to check that
\[ \frac{1}{n}  \max_{0 \leq i  \leq K_{n}} \hat{\ell}^{\circ,n}_{i}  \quad \mathop{\longrightarrow}^{(\P)}_{n \rightarrow \infty} \quad 0,\]
where $\hat{\ell}^{\circ,n}_{i}$ is the number of  corners  branching on the right of  $\llbracket \emptyset, v^{\circ,n}_{i} \llbracket$, where $v^{\circ,n}_{i} $ is the $i$-th  vertex in lexicographical order of the white reduced tree $\hat{\Ts^{\circ}_{n}}$.

In turn, if $\hat{R}^{\circ,n}_{i}$ denotes the number of white \emph{vertices}  branching on the right of  $\llbracket \emptyset, v^{\circ,n}_{i} \llbracket$ in $\hat{\Ts^{\circ}_{n}}$, we have
\[\max_{0 \leq i  \leq K_{n}} \hat{\ell}^{\circ,n}_{i} \leq \mathsf{Height}(\hat{\Ts^{\circ}_{n}})+ \max_{0 \leq i  \leq K_{n}} \hat{R}^{\circ,n}_{i}.\]

The advantage of considering  $\hat{R}^{\circ,n}_{i}$ is that this quantity is simply the value at time $i$ of the {\L}ukasiewicz path of $\hat{\Ts^{\circ}_{n}}$ (this  is a well-known fact of the depth-first search of trees). But Proposition~\ref{prop:cvBH}, applied with colors exchanged, shows that  the {\L}ukasiewicz path of $\hat{\Ts^{\circ}_{n}}$, scaled in time by a factor $n$ and in space by a factor $\sqrt{n}$ possesses a functional scaling limit. As a consequence, $\tfrac{1}{\sqrt{n}} \max_{0 \leq i \leq  K_{n}} \hat{R}^{\circ,n}_{i}$ converges in distribution. Therefore, it remains to check that 
\[\frac{\mathsf{Height}(\hat{\Ts^{\circ}_{n}})}{n}  \quad \mathop{\longrightarrow}^{(\P)}_{n \rightarrow \infty} \quad 0.\]
To this end, we will establish that for every fixed $\epsilon>0$,
\begin{equation}
\label{eq:height}\Pr{\mathsf{Height}(\hat{\Ts^{\circ}_{n}}) \geq n^{ \tfrac{1}{2} +\epsilon}}   \quad \mathop{\longrightarrow}_{n \rightarrow \infty} \quad 0,
\end{equation}
which is a stronger result. This is classical (see e.g.~\cite{BK00,MM03}) for one-type (critical, finite variance) BGW trees conditioned on having a fixed size. The main issue here is that even though $\hat{\Ts^{\circ}_{n}}$ may be seen as a ``one-type tree'', $\hat{\Ts^{\circ}_{n}}$ is \emph{not} a BGW tree conditioned on having a fixed size (this would have been the case if $\Ts_{n}$ was obtained by  conditioning $\Ts^{n}$ on having a fixed number of white vertices only). Another (smaller) issue is that  offspring distributions depend on $n$.

\paragraph*{Step 2: Reduction to a non-conditioned statement.}
In order to get rid of the conditioning, we start by adapting Section~\ref{ss:coding} in the setting where the root of the two-type tree is white. We fix a tree $\tau$, and we now assume that white vertices are at even generation and black vertices are at odd generation (this is the case of $\Ts^{\circ}_{n}$).  Denote by $(v^{\circ}_{i}(\tau))_{0 \leq i < \abs{\circ_{\tau}}} $ the white vertices of $\tau$ listed in lexicographical order.  For $1 \leq i \leq \abs{\circ_{\tau}}$, let $H_{i}(\tau)$ be the number of black children of $v^{\circ}_{i-1}(\tau)$ and let $W_{i}(\tau)$ be the number of white grandchildren of $v^{\circ}_{i-1}(\tau)$. Finally, set $\oH_{0}(\tau)=\oH_{0}(\tau)=0$ and  for $1 \leq i \leq \abs{\circ_{\tau}}$ set 
\[\oH_{i}(\tau)=H_{1}(\tau)+ \cdots+H_{i}(\tau), \qquad  \overline{W}_{i}(\tau)=W_{1}(\tau)+ \cdots+W_{i}(\tau)-i.\]
Now note that the tree whose {\L}ukasiewicz path is $\overline{W}(\tau)$ is precisely the \emph{reduced white subtree} of $\tau$.
Let $(H^{n},W^{n})$ be a random variable such that 
 \begin{equation}
    \Pr{H^{n}=i,W^{n}=j}=\mucn(i) \Pr{\Sbn_{i}=j}, \qquad  i, j \geq 0.
    \label{eq:HW}
  \end{equation}
Let $(H^{n}_{i},W^{n}_{i})_{i \geq 1}$ be a sequence of i.i.d.~random variables distributed as $(H^{n},W^{n})$. Set $\overline{H}^{n}_{i}=H^{n}_{1}+H^{n}_{2}+ \cdots+H^{n}_{i}$ and  $\overline{W}^{n}_{i}=W_{1}^{n}+W^{n}_{2}+ \cdots+W^{n}_{i}-i$. Then Lemma~\ref{lem:codeRW} (ii) (applied with colors switched) tells us that $(H_{i}(\Ts_{n}^{\circ}) ,W_{i}(\Ts_{n}^{\circ}))_{1 \leq i \leq K_{n}+1}$ has the same distribution as $(H^{n}_{i},W^{n}_{i})_{1 \leq i \leq K_{n}+1}$ conditioned on the event 
\[ \mathcal{C}^{+}_{n} \coloneqq \{\oH^{n}_{K_{n}+1}=n-K_{n},\overline{W}^{n}_{K_{n}+1}=-1 \textrm{ and } \overline{W}^{n}_{i} \geq 0 \textrm{ for every } 1 \leq i < K_{n}+1\}.\]
We now observe that $u^{\circ}_{i}(\Ts^{\circ}_{n})$ is an ancestor of $u^{\circ}_{k}(\Ts^{\circ}_{n})$ in $\Ts^{\circ}_{n}$
(or equivalently in the white reduced tree $\hat{\Ts^{\circ}_{n}}$)
if and only if $\overline{W}_{i}(\hat{\Ts^{\circ}_{n}})=\min_{i \leq j \leq k}  \overline{W}_{j}(\hat{\Ts^{\circ}_{n}})$.
Therefore the height of $u^{\circ}_{k}(\Ts^{\circ}_{n})$ in the white reduced tree $\hat{\Ts^{\circ}_{n}}$
 is equal to $ \# \{0 \leq i \leq k-1: \overline{W}_{i}(\hat{\Ts^{\circ}_{n}})=\min_{i \leq j \leq k}  \overline{W}_{j}(\hat{\Ts^{\circ}_{n}}) \}$. As a consequence, we can write
\[\mathsf{Height}(\hat{\Ts^{\circ}_{n}})= \Psi\big(  \overline{W}_{i}(\hat{\Ts^{\circ}_{n}}) :0 \leq i \leq K_{n}+1 \big),\]
where for any discrete trajectory $\omega=(\omega_{0},\omega_{1}, \ldots, \omega_{K_{n}+1})$ we have set
\[\Psi(\omega)= \max_{0 \leq k \leq K_{n}+1} \# \left\{0 \leq i \leq k-1: \omega_{i}=\min_{i \leq j \leq k}  \omega_{j} \right \}.\]

As a consequence, for fixed $\epsilon>0$,
\begin{align*}
\Pr{\mathsf{Height}(\hat{\Ts^{\circ}_{n}}) \geq n^{ \tfrac{1}{2} +\epsilon}} &= \frac{\Pr{\Psi( \overline{W}_{i} :0 \leq i \leq K_{n}+1) \geq n^{{1}/{2}+\epsilon},{ \mathcal{C}^{+}_{n} }}}{\Pr{ \mathcal{C}^{+}_{n} }} \\
& \leq  \frac{\Pr{\Psi( \overline{W}^{n}_{i} :0 \leq i \leq K_{n}+1) \geq n^{{1}/{2}+\epsilon}}}{\Pr{ \mathcal{C}^{+}_{n} }}
\end{align*} 
 We now claim that for every $\eps>0$, there exists $\epsilon'>0$ such that for every $n$ sufficiently large
\begin{equation}
\label{eq:decroissance} \P\big(\Psi( \overline{W}^{n}_{i} :0 \leq i \leq K_{n}+1) \geq n^{{1}/{2}+\epsilon}
\big) \leq e^{-n^{\epsilon'}}.
\end{equation}
The same reasoning that led to Lemma~\ref{lem:tilde} shows that there exists a constant $C>0$ such that $\Pr{ \mathcal{C}^{+}_{n} } \sim \tfrac{C}{n^{2}}$ as $n \rightarrow \infty$, so that the estimate \eqref{eq:decroissance} indeed implies \eqref{eq:height}.

\paragraph*{Step 3: large deviation estimates.}

Our goal is now to prove the large deviation estimate \eqref{eq:decroissance}.
For a discrete trajectory 
\hbox{$\omega=(\omega_0,\dots,\omega_{k})$},
in addition to $\Psi(\omega)$, we also let 
\[\Phi(\omega)=\# \left\{0 < i \leq k: \omega_{i}=\max_{0 \leq j \leq i} \, \omega_{j} \right \}\]
be the number of weak records of $\omega$.
By a standard time-reversal argument, we have, for every fixed $k \geq 0$,
\[
\# \left\{0 \le i <k:  \overline{W}^{n}_{i}=\min_{i \leq j \leq k}  \overline{W}^{n}_{j} \right\}
\quad \stackrel{(d)}{=} \quad 
\Phi(\overline{W}^{n}_0,\dots,\overline{W}^{n}_k).
\] 
By definition, $\Psi( \overline{W}^{n}_{i} :0 \leq i \leq K_{n}+1)$ is the maximum
of the left-hand side of the above display for $k$ in $\{0,\dots,K_n\}$.
Thus we have
\begin{multline*}
\P\Big(\Psi( \overline{W}^{n}_{i} :0 \leq i \leq K_{n}+1) \geq n^{1/2+\epsilon}
\Big)  \leq  \sum_{0 \leq k \leq K_{n+1}} \P\Big(\Phi( \overline{W}^{n}_{i} :0 \leq i \leq k) \geq n^{1/2+\epsilon} \Big)\\
 \leq (K_n+2) \cdot \P\Big(\Phi( \overline{W}^{n}_{i} :0 \leq i \leq K_{n}+1) \geq n^{1/2+\epsilon} \Big),
\end{multline*}
where the second inequality follows form the monotonicity of $\Phi$.
As a consequence, \eqref{eq:decroissance} will follow if we establish that 
for every $\eps>0$, there exists $\epsilon'>0$ such that for every $n$ sufficiently large
\begin{equation}                                                                
  \label{eq:decroissancePhi} \P\big(\Phi( \overline{W}_{i} :0 \leq i \leq K_{n}+1) \geq n^{1/2+\epsilon}
  \big) \leq e^{-n^{\epsilon'}}.                                                  
\end{equation}

 To this end, we closely follow \cite{MM03} (see also \cite[Sec.~1.3]{LG05}). Set
\[\nu^{n}(k)=\sum_{ i=0}^{\infty}\mucn(i) \Pr{\Sbn_{i}=k},\qquad k \geq 0,\]
so that  $(\overline{W}^{n}_{i})_{i \geq 0}$ is a random walk  with jump distribution 
$\Pr{\overline{W}^{n}_{1}=k}=\nu^{n}(k+1)$.
Observe that $\nu^{n}$ has expectation $1$,
so that $\overline{W}^{n}$ is centered. 
We denote by $\sigma_{n}^{2}$  the variance of $\overline{W}^{n}_1$,
which converges to a positive constant when $n$ tends to infinity
(since the parameters $a_\bullet^n,b_\bullet^n,a_\circ^n,b_\circ^n$
of the laws $\mubn$ and $\mucn$ converge to those of $\mub$ and $\muc$).
Finally we let
\[
  T^{n}= \inf \Big\{k \geq 1 : \overline{W}^{n}_{k}= \max_{0 \leq i \leq k} \overline{W}^{n}_{i} \Big\}, \quad M^n = \max_{0 \le i \le {K_n+1}} \overline{W}^{n}_i.
\]
be respectively the first ladder time and the maximum of $\overline{W}^{n}$ up to time $K_{n}+1$.

Observe that the quantity $ \P\big(\Phi( \overline{W}_{i} :0 \leq i \leq K_{n}+1) \geq n^{1/2+\epsilon}
  \big) $ is bounded from above by
\begin{equation}
\label{eq:somme} \P\bigg( \big|M^n-\frac{\sigma_n^2}{2}\Phi(\overline{W}^{n}_0,\dots,\overline{W}^{n}_{K_n+1})\big|
  \geq n^{1/4+\eps} \bigg) + \P\bigg( M^n  \geq n^{1/2+\eps}- \frac{2}{\sigma_{n}^{2}} n^{1/4+\epsilon} \bigg)
\end{equation}  
In order to bound from above the first term of \eqref{eq:somme}, we use the fact that there exists $\epsilon_{1}>0$ such that for every $n$ sufficiently large 
\begin{equation}
  \P\bigg( \big|M^n-\frac{\sigma_n^2}{2}\Phi(\overline{W}^{n}_0,\dots,\overline{W}^{n}_{K_n+1})\big|
  >n^{1/4+\eps} \bigg) < e^{-n^{\eps_{1}}},
\label{eq:Phi_Close_M}
\end{equation}
as in \cite[Lemma 1.11]{LG05} (we warn the reader that in loc. cit.,
$\Phi(\overline{W}^{n}_0,\dots,\overline{W}^{n}_{K_n+1})$ is denoted by $K_n$).
We cannot directly apply \cite[Lemma 1.11]{LG05} because our offspring distributions depend on $n$;
we note however that the same proof can be applied,
provided that we have {\em uniform} deviation estimates for sums of i.i.d. variables 
distributed as $\overline{W}^{n}_{1}$ or $\overline{W}^{n}_{T^n}-\tfrac12 \sigma_n^2$;
here, {\em uniform} means that the quantities $\eps$ and $\eps_{1}$ appearing
in the exponents should be independent of $n$.
Such estimates directly follow from the estimate for its Laplace transform,
given in Lemma~\ref{lem:moments} (ii) and (iii) below (see, e.g., the proof of \cite[Lemma 1.12]{LG05}).
Therefore \eqref{eq:Phi_Close_M} is proved.

In order to bound from above the second term of \eqref{eq:somme}, we use the fact that there exists $\epsilon_{2} \in (0,\epsilon/4)$ such that for every $n$ sufficiently large
\begin{equation}
\label{eq:Mpasgrand}
\P\bigg( M^n  >n^{1/2+\eps}- \frac{2}{\sigma_{n}^{2}} n^{1/4+\epsilon} \bigg) \leq  n \max_{1 \leq k \leq n} \P\bigg( \overline{W}^{n}_{k}>n^{1/2+\eps}- \frac{2}{\sigma_{n}^{2}} n^{1/4+\epsilon} \bigg) \leq n e^{-n^{\epsilon_{2}}}.
\end{equation}This similarly follows the proof of \cite[Lemma 1.12]{LG05} combined with Lemma~\ref{lem:moments} (ii) below.

By combing Eqs.~\eqref{eq:Mpasgrand} and \eqref{eq:Phi_Close_M} with \eqref{eq:somme}, we get \eqref{eq:decroissancePhi} and this completes the proof.
\end{proof}

 \begin{lemma}
 \label{lem:moments}
 The following assertions hold.
 \begin{enumerate}
 \item[(i)] Let $(Y_{n})_{n \geq 1}$ be a sequence of centered random variables such that there exists $a>0$ and $C_{1}>0$ such that $\Es{e^{a|Y_{n}|}} \leq C_{1}$ for every $n \geq 1$. Then there exists $C_{2}>0$ such that for $\lambda>0$ sufficiently small, for every $n \geq 1$, $\E\big[e^{\lambda Y_{n}}\big] \leq e^{C_{2} \lambda ^{2}}$.
 \item[(ii)]  There exists a constant $C>0$ such that for $\lambda>0$ sufficiently small, for every $n  \geq 1$, $\E\big[e^{\lambda \overline{W}^{n}_{1}}\big] \leq e^{C \lambda ^{2}}$.
 \item[(iii)] There exists a constant $C>0$ such that for $\lambda>0$ sufficiently small, for every $n \geq 1$, $\E\big[{e^{\lambda(\overline{W}^{n}_{T^{n}} - {\sigma_{n}^{2}}/{2} ) }}\big] \leq e^{C \lambda ^{2}}$.
 \end{enumerate}
 \end{lemma}
 
 \begin{proof}
 The first assertion is a standard result concerning sub-exponential distributions. However, since we need estimates uniform in $n$, we give a proof for completeness. First,  using Markov's inequality, write, for $t \geq 0$ and $n \geq 1$,
 \[\P(|Y_{n}| \geq t)= \P(e^{a|Y_{n}|} \geq e^{at}) \leq C_{1} e^{-at}.\]
 Hence, for $n,k \geq 1$,
 \[\Es{|Y_{n}|^{k}}=\int_{0}^{\infty} k t^{k-1} \Pr{|Y_{n}| \geq t} {\d}t \leq  C_{1} k \int_{0}^{\infty} t^{k-1} e^{-a t} {\d}t.\]
 In particular, $\Es{|Y_{n}|^{2}} \leq 2C_{1}/a^{2}$.
 Now, for $\lambda>0$ small enough, since $\E[Y_{n}]=0$,
 \[ \Es{e^{\lambda Y_{n}}}=1+ \sum_{k=2}^{\infty} \Es{Y_{n}^{k}} \frac{\lambda^{k}}{k!}  \leq 1+ \frac{C_{1}}{a^{2}} \lambda^{2}+  C_{1}  \int_{0}^{\infty}  \left( \sum_{k=3}^{\infty}  kt^{k-1} \frac{\lambda^{k}}{k!} \right) e^{-a t} {\d}t.\]
 By calculating explicitly the last sum and then the integral, we finally get that for $\lambda \in (0,a/2)$, for every $n \geq 1$, $\E\big[e^{\lambda Y_{n}} \big] \leq  1+ \tfrac{C_{1}}{a^{2}} \lambda^{2}+  \tfrac{2C_{1}}{a^{3}} \lambda^{3} \leq 1+ \big( \tfrac{C_{1}}{a^{2}}+ \tfrac{2C_{1}}{a^{3}} \big) \lambda^{2}$, which gives the desired result with $C_{2}= \tfrac{C_{1}}{a^{2}}+ \tfrac{2C_{1}}{a^{3}}$. 
 \medskip

 For the second assertion, by (i), it is enough to show that there exists $a>0$ and $C_{1}>0$ such that $\E\big[e^{a |\overline{W}^{n}_{1}|}\big] \leq C_{1}$ for every $n \geq 1$.
 Since $\overline{W}^{n}_{1}$ takes value in $\{-1,0,1,\dots\}$, it is equivalent to prove
 that
 $\E\big[e^{a (\overline{W}^{n}_{1}+1)}\big] \leq C'_{1}$ for every $n \geq 1$
 (where $C'_1$ is another constant).
 To this end, write
 \begin{equation}
   \E\big[e^{a (\overline{W}^{n}_{1}+1)}\big]  \leq  \sum_{ i=0}^{\infty}\mucn(i)  \E \big[ e^{a \Sbn_{i}} \big] =   \sum_{ i=0}^{\infty}\mucn(i)  \E \big[e^{a \Sbn_{1}} \big]^{i}.
   \label{eq:Tech5}
 \end{equation}
 Let $\delta>0$ be such that  \eqref{eq:expmoment}  holds.
 Hence $\sup_{n \ge 1} \E[e^{\delta \Sbn_{1}}]<\infty$
 and from the first assertion,
 for any $A>1$, we have $\sup_{n \ge 1} \E[e^{a \Sbn_{1}}] <A$ for $a$ sufficiently small.
 Choosing $A=e^\delta$, and going back to \eqref{eq:Tech5}, we get
 \[\E\big[e^{a (\overline{W}^{n}_{1}+1)}\big]  \leq \sum_{ i=0}^{\infty}\mucn(i) e^{\delta i}.\]
 By \eqref{eq:expmoment}, the right-hand side is finite and bounded, as a function of $n$,
 which concludes the proof of (ii).
 \medskip

For the last assertion, we first observe that by \cite[Lemma 1.9]{LG05}, the law of $\overline{W}^{n}_{T^{n}}$ is given by
\[\Pr{\overline{W}^{n}_{T^{n}}=k}=\nu^{n}([k,\infty)), \qquad k \geq 0.\]
Since $\sigma_{n}^{2}$ converges as $n \rightarrow \infty$, 
it is enough to show that there exists $a>0$ and $C>0$
such that $\E\big[e^{a |\overline{W}^{n}_{T^{n}}|}\big] \leq C$ for every $n \geq 1$.
Note that the absolute values are here superfluous since $\overline{W}^{n}_{T^{n}}$
is nonnegative.

The proof of (ii) shows that there exists two constants $C',b>0$ such that 
\[\Pr{\overline{W}^{n}_{T^{n}}=k}=\nu^{n}([k,\infty)) = \P(\overline{W}^{n}_{1} \geq k-1)  \leq e^{-b(k-1)} \E\big[ e^{b\, \overline{W}^{n}_1} \big]\le  C' e^{-bk}\] 
for every $n \geq 1$ and $ k \geq 0$. 
This implies the existence of $a>0$ and $C>0$ such that $\E\big[e^{a \overline{W}^{n}_{T^{n}}}\big] \leq C$ for every $n \geq 1$. This concludes the proof.
 \end{proof}

\subsection{Completing the proof of Theorem \ref{thm:cvlam}.} So far, we have established Theorem \ref{thm:cvlam} in the following cases:
\begin{itemize}
\item (i)$_{c=0}$ only for $ \dFKn$ in Section~\ref{ssec:partialproof};
\item (i)$_{c>0}$ in Section~\ref{ssec:cv_pos};
\item (ii) in Section~\ref{sec:Linfty}.
\end{itemize}

In this Section, we treat the missing cases. More precisely, in Section~\ref{sssec:iiiF}, we show that the convergence $\dFKn \rightarrow \mathbf{L}_{\infty}$  in case (iii)  follows from (ii) by a maximality argument. Then, in Section~\ref{sssec:P0}, we establish, in case (i)$_{c=0}$, the  convergence $\dPKn \rightarrow \mathbf{L}_{0}$ by using, surprisingly, the convergence $ \dFKn \rightarrow \mathbf{L}_{\infty}$ just established in Section~\ref{sssec:iiiF}. We conclude in Section~\ref{sssec:iiiP} that   the last missing convergence $\dPKn \rightarrow \mathbf{L}_{c}$ in case (iii) holds by using a short symmetry argument,  as well as again the convergence $ \dFKn \rightarrow \mathbf{L}_{\infty}$  established in Section~\ref{sssec:iiiF}.

\subsubsection{Proof of  Theorem~\ref{thm:cvlam} (iii) for $ \dFKn$.}
\label{sssec:iiiF}
Suppose that $ \frac{n-K_{n}}{\sqrt{n}} \rightarrow c$ as $n \rightarrow \infty$, with $c \geq 0$. We choose $K'_{n}$ such that $ \frac{K'_{n}}{\sqrt{n}} \rightarrow \infty$ and $ \frac{n-K'_{n}}{\sqrt{n}} \rightarrow \infty$ as $n \rightarrow \infty$, and $K'_{n} \leq K_{n}$ for $n$ sufficiently large. Then by  Theorem~\ref{thm:cvlam} (ii), which has been established in Section~\ref{sec:Linfty}, $ \dot{\FFF}^{(n)}_{K'_{n}}$ converges in distribution to $\mathbf{L}_{\infty}$, which is maximal for inclusion. Since  $ \dot{\FFF}^{(n)}_{K'_{n}}\subset\dFKn$, it readily follows by a compactness argument that $ \dFKn$ also converges in distribution to $\mathbf{L}_{\infty}$ (and in fact we have the joint convergence $ (\dot{\FFF}^{(n)}_{K'_{n}},\dFKn) \to (Z,Z)$, where $Z$ has the distribution of $\mathbf{L}_{\infty}$). This completes the proof. \qed

\begin{remark}
  The same argument can be used to deduce the convergence of $\dFKn$ to $\mathbf L_\infty$ in the case
  $c=\infty,\, \gamma>0$ from the case $c=\infty,\, \gamma=0$.
  Since in Section~\ref{sec:Linfty}, the most difficult part is the case $\gamma>0$,
  we could have shortened significantly the proof if we were only interested in $\dFKn$
  and not in $\dPKn$. However, it is not true in general that  if $\dFKn \rightarrow \mathbf{L}_{\infty}$, then    $(\dFKn,\dPKn) \rightarrow ( \mathbf{L}_{\infty}, \mathbf{L}_{\infty})$ (but it is true that if $\dPKn \rightarrow \mathbf{L}_{\infty}$ then $(\dFKn,\dPKn) \rightarrow ( \mathbf{L}_{\infty}, \mathbf{L}_{\infty})$, see Lemma~\ref{lem:reduction_FtoK}).
\end{remark}

\subsubsection{Proof of Theorem~\ref{thm:cvlam} (i)$_{c=0}$ for $ \dPKn$.}
\label{sssec:P0}

Assume that $K_{n} \rightarrow \infty$ and that $\frac{K_n}{\sqrt{n}} \to 0$.
We want to prove that $\dPKn$ tends to $\mathbb{L}_0=\S$.
Recall from Lemma~\ref{lem:observation} that 
the blocks of $\PKn$ are the connected components of $\FKn$; in particular, $ \S \cap \dPKn= \S \cap \dFKn$).
Since $\S \cap \dFKn$ converges to $\S$ in distribution 
(this was established in the second part of the convergence in distribution 
$\dFKn \rightarrow \mathbf{L}_0$ in Section~\ref{ssec:partialproof}), 
any limit point of $\dPKn$ contain $\S$ and
it is enough to check that 
$\Lambda(\dPKn) \rightarrow 0$  in probability as $n \rightarrow \infty$,
where $\Lambda(L)$ is the Euclidean length of the longest chord of the lamination $L$.

In the first part of the proof of the convergence in distribution $\dFKn \rightarrow \mathbf{L}_0$ in Section~\ref{ssec:partialproof}, it was shown that  $\Lambda(\dFKn) \rightarrow 0$ in probability as $n \rightarrow \infty$.
The problem is that $\Lambda(\dFKn) \rightarrow 0$ does not automatically imply that $\Lambda(\dPKn) \rightarrow 0$, 
since $\dFKn$ may have a connected component made of many small chords
that creates a long chord in $\dPKn$ (for instance, if $\dFKn= \cup_{k=1}^{\lfloor n/2 \rfloor} [e^{-2 \i \pi k/n},e^{-2 \i \pi (k+1)/n}]$, then $\dPKn$ has a chord of length approximately $2$ while $\dFKn$ only has chords of length order $1/n$). 

To prove $\Lambda(\dPKn) \rightarrow 0$, we shall additionally use the convergence
of $\dFKn \rightarrow \mathbf{L}_0$ (proved in Section~\ref{ssec:partialproof} since we are in the case $c=0$)
and of $\dot{\mathcal{F}}^{(n)}_{n-1}$ to the Brownian triangulation $\mathbf{L}_{\infty}$ (proved in Section~\ref{sssec:iiiF}).
From Skorokhod's representation theorem, we can assume that the convergence 
$\big(\dFKn, \Lambda(\dFKn), \dot{\mathcal{F}}^{(n)}_{n-1}\big) \rightarrow (\S,0,\mathbf{L}_{\infty})$ holds almost surely.
Note also that, almost surely, for every $\epsilon>0$, for every arc ${A} \subset \S$ of length at least $\epsilon$, 
$\mathbf{L}_{\infty}$ contains a chord of positive Euclidean length with endpoints belonging to $ {A}$ (indeed, this easily follows from the fact that local minima times are almost surely dense for the Brownian excursion).

We argue by contradiction, and assume the existence of $\epsilon>0$ such that, 
up to extraction,  $\Lambda(\dPKn) \geq 2\epsilon$ for every $n \geq 1$.
By compactness, up to extraction, we may assume that 
$\dPKn$ contains a chord $ [e^{-2 \i \pi s^{n}},e^{-2 \i \pi t^{n}}]$
such that $s^{n} \rightarrow s$ and $t^{n} \rightarrow t$,
with $\min(t-s,1-t+s) \geq \epsilon$.
Since $\Lambda(\dFKn) \rightarrow 0$, for $n$ sufficiently large 
the chord $[e^{-2 \i \pi s^{n}},e^{-2 \i \pi t^{n}}]$ is not in $\FKn$. But since the blocks of $\PKn$ are the connected components of $\FKn$, $n s^{n}$ and $n t^{n}$ must belong to the same connected component $ \mathcal{C}_{n}$  of $\FKn$. By considering the path that joins them in $\dFKn$, we get the existence of $s_{1}^{n}, \ldots, s_{d_{n}}^{n}$ with $s_{1}^{n}=s^{n}$, $s_{d_{n}}=t^{n}$, and $n s_{i}^{n} \in \mathcal{C}_{n}$ for every $1 \leq i \leq d_{n}$, with either $ \{s_{1}^{n}, \ldots, s_{d_{n}}^{n}\}  \subset[s^{n},t^{n}]$, or $ \{s_{1}^{n}, \ldots, s_{d_{n}}^{n}\}  \subset [0,1] \backslash [s^{n},t^{n}]$. To simplify, let us treat the first case (the argument in the second case is exactly the same). Then $s_{1}^{n} \leq s_{2}^{n} \leq \cdots \leq s_{d_{n}}^{n}$, and since $\Lambda(\dFKn) \rightarrow 0$, we have $\max_{1 \leq i \leq d_{n}-1} (s^{n}_{i+1}-s_{i}^{n}) \rightarrow 0$ as $n \rightarrow \infty$. 
Since $\dFKn$ is a non-crossing tree,  this implies that for every fixed $\eta>0$, for $n$ sufficiently large, in $\dFKn$ there is no chord of Euclidean length at least $\eta$ with endpoints belonging to $[s^{n},t^{n}]$. Since $  \dot{\mathcal{F}}^{(n)}_{n-1}$ is still a non-crossing tree obtained from $\dFKn$ by adding chords, so the latter property also holds for $ \dot{\mathcal{F}}^{(n)}_{n-1}$ (see Figure~\ref{fig:zoneinterdite} for an illustration).
  But $ \dot{\mathcal{F}}^{(n)}_{n-1} \rightarrow \mathbf{L}_{\infty}$. Therefore, in $\mathbf{L}_{\infty}$,  there is no chord of positive Euclidean length with endpoints belonging to the arc $ \{ e^{-2 \i \pi u} : u \in (s,t) \}$. This is a contradiction, and this completes the proof.
\qed

\begin{figure}[t]
  \begin{center}
  \includegraphics[scale=1]{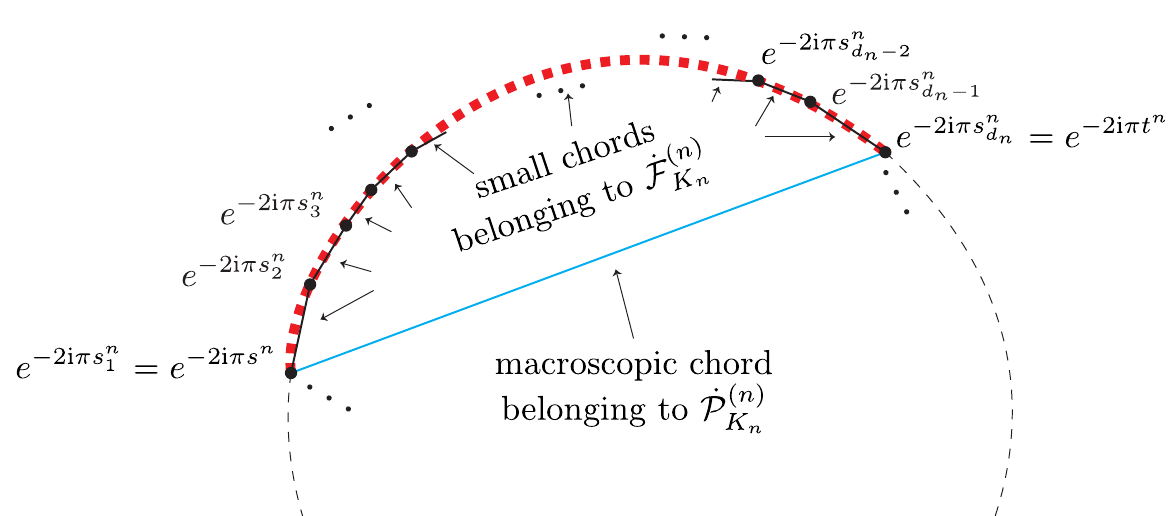}
  \caption{\label{fig:zoneinterdite} Illustration of the proof of Theorem~\ref{thm:cvlam} (i)$_{c=0}$ for $ \dPKn$: the blue ``macroscopic'' chord belongs to $\dPKn$. Since  the blocks of $\PKn$ are the connected components of $\FKn$ and since all the chords of $\dFKn$ are ``small'', there exists a path of such chords (in black) joining the two endpoints of the ``macroscopic'' chord. It is then not possible to add a new  ``macroscopic'' chord with endpoints  belonging to the bold dashed red arc.}
  \end{center}
  \end{figure}

\subsubsection{Proof of  Theorem~\ref{thm:cvlam} (iii) for $ \dPKn$}
\label{sssec:iiiP}
Assume that $ \frac{n-K_{n}}{\sqrt{n}} \rightarrow c$ as $n \rightarrow \infty$, with $c \geq  0$.
We first establish that
\begin{equation}
\label{eq:cvhat} \hat{\PPP}^{(n)}_{K_{n}}  \quad \mathop{\longrightarrow}^{(d)}_{n \rightarrow \infty} \quad \mathbf{L}_{c}.
\end{equation}
By Lemma~\ref{lem:reductionKn}, it suffices to show that $ \hat{\PPP}^{(n)}_{n-K_{n}} \rightarrow \mathbf{L}_{c}$ as $n \rightarrow \infty$. If $n-K_{n} \rightarrow \infty$,  this follows from Theorem~\ref{thm:cvlam} (i) for $\dot{\PPP}^{(n)}_{n-K_{n}}$ (which was established in Section~\ref{ssec:cv_pos} for $c>0$ and in Section~\ref{sssec:P0} for $c=0$). If $ \limsup_{n \rightarrow \infty} (n-K_{n})<\infty$, then   $ \hat{\FFF}^{(n)}_{n-K_{n}}$ has a bounded number of chords, whose Euclidean length, divided by $n$, tends to $0$ with probability tending to one as $n \rightarrow \infty$ (this was established in Section~\ref{ssec:partialproof}), and this readily implies the convergence $ \hat{\PPP}^{(n)}_{n-K_{n}} \rightarrow \mathbf{L}_{0}$. This establishes \eqref{eq:cvhat}.

Now assume that $c>0$. Since $\hat{\mathcal{P}}^{(n)}_{K_{n}}$ is obtained from  $\dPKn$ by adding points of $\S$ and since $\hat{\mathcal{P}}^{(n)}_{K_{n}} \rightarrow \mathbf{L}_{c}$ in distribution, any sub-sequential distributional limit of $\dPKn$ must contain all the chords of $\mathbf{L}_{c}$, and by \eqref{eq:lamination_stable2} must be equal to $\mathbf{L}_{c}$. In other words, the convergence $\dPKn \rightarrow  \mathbf{L}_{c}$ holds in distribution.

Finally assume that $c=0$.  Since $\hat{\mathcal{P}}^{(n)}_{K_{n}}$ is obtained from  $\dPKn$ by adding points of $\S$, \eqref{eq:cvhat} shows that any subsequential distribution limit of $\dPKn$ is included in $\mathbf{L}_{0}=\S$. Therefore, it is enough to show that the convergence $\S \cap \dPKn \rightarrow \S$  holds in distribution as $n \rightarrow \infty$. To this end, the key idea is to use the fact that $\S \cap \dPKn=\S \cap \dFKn$. Indeed, by  Theorem~\ref{thm:cvlam} (iii) for $ \dFKn$ (this was established in Section~\ref{sssec:iiiF}), $\dot{\mathcal{F}}^{(n)}_{K_{n}} \rightarrow \mathbf{L}_{\infty}$. By \eqref{eq:isolated}, 
there exists an endpoint of some non-trivial chords of $\mathbf{L}_{\infty}$ arbitrarily close to any fixed point of $\S$,
which implies that   $\S \cap \dot{\mathcal{F}}^{(n)}_{K_{n}} \rightarrow \S$ in distribution as $n \rightarrow \infty$, and hence  $\S \cap \dot{\mathcal{P}}^{(n)}_{K_{n}} \rightarrow \S$ in distribution as $n \rightarrow \infty$. This completes the proof.
\qed
\medskip

\begin{remark}~
\begin{itemize}
\item In view of the previous results, we have established that a condensation phenomenon occurs in the dual tree $ \mathcal{T}(\PKn)$ when either $K_{n} \rightarrow \infty$ and $\tfrac{K_{n}}{\sqrt{n}} \rightarrow 0$, or $\tfrac{n-K_{n}}{\sqrt{n}} \rightarrow 0$. Indeed, when  $\tfrac{n-K_{n}}{\sqrt{n}} \rightarrow 0$, the fact that the convergence $\dPKn {\rightarrow} \S$ holds in distribution indicates that the dual tree  $ \mathcal{T}(\PKn)$ contains a unique (black) vertex of macroscopic degree (of order $n$)  such that the subtrees grafted on it all have size $o(n)$.  When $K_{n} \rightarrow \infty$ and $\tfrac{K_{n}}{\sqrt{n}} \rightarrow 0$, the same phenomenon happens with condensation on a white vertex.  It seems difficult to establish these facts directly (such as in \cite{AL11,Jan12,Kor15}).
\item If $K \geq 1$ is fixed, as $n \rightarrow \infty$, we believe that $(\dot{\mathcal{P}}^{(n)}_{K},\dot{\mathcal{F}}^{(n)}_{K})$ converge jointly in distribution to the same limit, a collection of $K$ i.i.d.~uniformly distributed points on $\S$, and that this could be established by using Proposition~\ref{prop:lawproduct}.
\end{itemize}
\end{remark}

%\bibliographystyle{alpha}
%\bibliography{bibli.bib}

\newcommand{\etalchar}[1]{$^{#1}$}
\providecommand{\bysame}{\leavevmode\hbox to3em{\hrulefill}\thinspace}
\providecommand{\MR}{\relax\ifhmode\unskip\space\fi MR }
\providecommand{\MRhref}[2]{%
  \href{http://www.ams.org/mathscinet-getitem?mr=#1}{#2}
}
\providecommand{\href}[2]{#2}

\end{document}